\documentclass[11pt]{amsart}
\usepackage[utf8]{inputenc}

\usepackage[plainpages=false]{hyperref}
\usepackage{amsfonts,latexsym,rawfonts,amsmath,amssymb,amsthm, mathrsfs, lscape}
\usepackage{verbatim}

\usepackage[all]{xy}
\usepackage{graphicx,psfrag}

\usepackage{array,tabularx}

\usepackage{setspace}
\usepackage{anysize}

\newtheorem{thm}{Theorem}
\newtheorem{cor}{Corollary}[section]
\newtheorem{clm}{Claim}[section]
\newtheorem{prop}{Proposition}[section]

\theoremstyle{remark}
\newtheorem{rmk}{Remark}[section]

\theoremstyle{definition}
\newtheorem{defn}{Definition}[section]

\newtheorem{theorem}{Theorem}[section]

\newtheorem{lemma}[theorem]{Lemma}
\marginsize{3cm}{3cm}{2.5cm}{2.5cm}
\def\p{\partial}
\def\R{\mathbb{R}}
\def\C{\mathbb{C}}
\def\N{\mathbb{N}}

\def\l{\lambda}

\def\i{\sqrt{-1}}

\def\I {\mathbb{I}}

\def\cC{{\mathcal C}}

\def\cD{\mathcal D}
\def\cE{{\mathcal E}}
\def\cF{{\mathcal F}}

\def\cH{{\mathcal H}}

\def\cK{{\mathcal K}}
\def\cL{{\mathcal L}}

\def\cW{{\mathcal W}}

\def \bp {\overline{\partial}}

\def \Ker {\text{Ker}}
\def\Aut{\text{Aut}}

\begin{document}

\title{Geometric pluripotential theory on Sasaki manifolds}
\author{Weiyong He; Jun Li}
\address{Department of Mathematics, University of Oregon, Eugene, OR 97403. }
\email{whe@uoregon.edu;}
\address{College of Mathematics and Econometrics, Hunan University, Hunan, 410082, P. R. China}
\email{lijun1985@hnu.edu.cn}

\begin{abstract}
We extend profound results in pluripotential theory on K\"ahler manifolds \cite{D4} to Sasaki setting via its transverse K\"ahler structure. As in K\"ahler case, these results form a very important piece to solve the existence of Sasaki metrics with constant scalar curvature (cscs) in terms of properness of $\cK$-energy, considered by the first named author in \cite{he182}. One main result is to generalize T. Darvas' theory on the geometric structure of the space of K\"ahler potentials in Sasaki setting. Along the way we  extend most of corresponding results in pluripotential theory to Sasaki setting via its transverse K\"ahler structure. \end{abstract}

\maketitle

\section{Introduction}
Sasaki manifolds have gained their prominence in physics and in algebraic geometry and Riemannian geometry \cite{BG}. 
There are tremendous work in the last two decades in Sasaki geometry, in particular on Sasaki-Einstein manifolds, see \cite{BG, GMSW, BGK, FOW, MSY, HeSun, CS} and reference therein. 
On the other hand, Sasaki geometry is an odd dimensional analogue of K\"ahler geometry and almost all results in K\"ahler geometry have their counterparts in Sasaki geometry.  Calabi's extremal metric \cite{Ca1, Ca2} (and csck) has played a very important role in K\"ahler geometry and  it has a direct adaption in Sasaki setting \cite{BGS1}.  
In 1997, S. K. Donaldson \cite{Don97} proposed an extremely fruitful program to approach existence of csck (extremal metrics) on a compact K\"ahler manifold with a fixed K\"ahler class. 
Donaldson's program has also been extended to Sasaki setting, see \cite{GZ1, he14} for example. 

A major problem in K\"ahler geometry is to characterize exactly when a K\"ahler class contains a csck (extremal). The analytic part for existence of csck is to solve a fourth order highly nonlinear elliptic equation, the scalar curvature type equation. This problem is regarded as a very hard problem in the field. Recently Chen and Cheng \cite{CC1, CC2, CC3} have solved  a major conjecture that existence of csck is equivalent to well studied conditions such as properness of Mabuchi's $K$-energy, or geodesic stability. The first named author \cite{he182} proved the following counterpart in Sasaki setting,

\begin{thm}[\cite{he182}]\label{cscs}There exists a Sasaki metric with constant scalar curvature if and only if the $\cK$-energy is reduced proper with respect to $\text{Aut}_0(\xi, J)$, the identity component of automorphism group which preserves the Reeb vector field and transverse complex structure. 
\end{thm}

The proof of Theorem \ref{cscs} is an adaption of recent breakthrough of Chen-Cheng \cite{CC3} on the existence of csck in K\"ahler setting to Sasaki setting. Technically the arguments consist of two major parts: a priori estimates of nonlinear PDE and pluripotential theory.  
Building up on previous development of pluripotential theory, T. Darvas \cite{D1, D2} has developed profound theory to study the geometric structure of space of K\"ahler potentials. Among others, he introduced a Finsler metric $d_1$, and proved very effective estimates of distance function $d_1$ in terms of well studied energy functionals such as Aubin's $I$-functional.  Darvas's results turn out to be very useful to understand the geometric structure of space of K\"ahler potentials, in particular in the study of csck \cite{DR, BDL2, CC3}. 
In this paper we extend many results in pluripotential theory on K\"ahler manifolds, notably in \cite{GZ, D1, D2} to Sasaki setting. These results play an important role in the proof of Theorem \ref{cscs}. To prove these results, we need to explore the geometric structures of Sasaki manifolds, in particular the K\"ahler cone structure and transverse K\"ahler structure. 
 
Let $(M, g)$ be a compact Riemannian manifold of dimension $2n+1$, with a Riemannian metric $g$. Sasaki manifolds have very rich geometric structures and have many equivalent descriptions. A probably most straightforward formulation is as follows: its metric cone \[X=\R_{+}\times M, \bar g_X=dr^2+r^2 g.\] is a K\"ahler cone. Hence  there exists a complex structure $J$ on $X$ such that $(g_X, J)$ defines a K\"ahler structure.  We identify $M$ with its natural embedding $M\rightarrow \{r=1\}\subset X$. 
 The $1$-form $\eta$ is given by $\eta=J(r^{-1}dr)$ and it defines a contact structure on $M$. The vector field $\xi:=J(r\p_r)$ is a nowhere vanishing, holomorphic Killing vector field and it is called the \emph{Reeb vector field} when it is restricted on $M$. The integral curves of $\xi$ are geodesics, and give rise  to a foliation on $M$, called the \emph{Reeb foliation}. Then there is a K\"ahler structure on the local leaf space of the Reeb foliation, called the \emph{transverse K\"ahler structure}.  A standard example of a Sasaki manifold is the odd dimensional round sphere $S^{2n+1}$. The corresponding K\"ahler cone is $\C^{n+1}\backslash\{0\}$ with the flat metric and its transverse K\"ahler structure descends to $\mathbb{CP}^n$ with the Fubini-Study metric. 
 
We can also formulate Sasaki geometry, in particular the transverse K\"ahler structure via its contact bundle $\cD=\text{Ker}(\eta)\subset TM$. The complex structure $J$ on the cone descends to the contact bundle via $\Phi:=J|_\cD$. The Sasaki metric can be written as follows,
 \begin{equation*}
 g=\eta\otimes\eta+g^T,
 \end{equation*}
 where $g^T$ is the transverse K\"ahler metric, given by $g^T:=2^{-1}d\eta(\Phi\otimes \mathbb{I})$. The transverse K\"ahler form is denoted by $\omega^T=2^{-1}d\eta$. 
We shall study the transverse K\"ahler geometry of Sasaki metrics, with the Reeb vector field $\xi$ and transverse complex structure (equivalently the complex structure $J$ on the cone) both fixed.
This means that we fix the basic K\"ahler class $[\omega^T]$ with $\omega^T=2^{-1}d\eta$ and study the Sasaki structures induced by the space of transverse K\"ahler potentials,
\begin{equation*}
\cH=\{\phi\in C_B^\infty(M): \omega_\phi=\omega^T+dd^c_B \phi>0\},
\end{equation*}
where $C_B^\infty(M)$ is the space of smooth \emph{basic functions}. The main result in the paper is,

\begin{thm}\label{pluri01}$(\cE_p(M, \xi, \omega^T), d_p)$ is a complete geodesic metric space for $p\in [1, \infty)$, which is the metric completion of $(\cH, d_p)$. For any $u, v\in \cE_p$, 
$d_p(u, v)$ is realized by a unique finite energy geodesic $\cE_p$ connecting $u$ and $v$. 
There exists  a uniform constant $C=C(n, p)>1$ such that
\begin{equation*}\label{d01}
C^{-1} I_p(u, v)\leq d_p(u, v)\leq CI_p(u, v),
\end{equation*}
where the energy functional $I_p$ is given by
\[
I_p(u, v)=\|u-v\|_{p, u}+\|u-v\|_{p, v}
\] 
Moreover, we have
\begin{equation*}\label{d02}
d_p(u, \frac{u+v}{2})\leq C d_p(u, v). 
\end{equation*}
\end{thm}
We refer to Section 3 for notions such as $\cE_p, d_p$. Theorem \ref{pluri01} is the counterpart of main results in \cite{D2} in Sasaki setting. 
An important notion in the study of csck is the convexity of $\cK$-energy along $C^{1, \bar 1}$ geodesics \cite{BB} (see also \cite{CLP}), which was generalized to Sasaki setting by \cite{JZ, VC}. Given the results above, one can then extend $\cK$-energy to $\cE_1$-class and keep its convexity along finite energy geodesics as in \cite{BDL}. Moreover, this allows to define precisely the properness of $\cK$-energy in terms of the distance $d_1$.  One can then prove Theorem \ref{cscs} using a priori estimates of scalar curvature type equation together with properness assumption, where the effective estimates of $d_1$ in Theorem \ref{pluri01} play an important role; for details see \cite{he182}. 

Along the way to prove Theorem \ref{pluri01}, it is necessary to extend results as in \cite{GZ, D4} to Sasaki setting. Certainly the essential ideas lie in results in K\"ahler setting and T. Darvas' lecture notes \cite{D4} is an excellent reference. On the other hand, we should emphasize that in Sasaki setting, there are many new difficulties when the Reeb foliation is irregular. We have to utilize the K\"ahler cone structure and transverse K\"ahler structure in an effective way. For example, one can use Type-I deformation to approximate irregular structure by quasiregular structure. Such an approximation is very useful at times for extension to Sasaki setting.  We also construct explicit holomorphic charts on the K\"ahler cone out of its transverse K\"ahler structure, see Lemma \ref{chart}. This very explicit relation between the holomorphic charts and foliations charts of transverse K\"ahler structure seems to appear in literature for first time, to the authors' knowledge. This explicit construction of holomorphic charts builds a very straightforward relation between plurisubharmonic functions on cone and (transverse) plurisubharmonic functions via transverse K\"ahler structure. This construction plays an important role in our arguments.  \\

We organize the paper as follows. In Section 2 we introduce basic notations and concepts of Sasaki geometry. We study the geometric structure of the space of transverse K\"ahler potentials using geodesic equation and pluripotential theory in Section 3. In Section 4 we prove the main theorem. We include a brief discussion of Sasaki-extremal metric in Section 5.  Appendix (Section 6) contains various topics in pluripotential theory, including complex Monge-Ampere operator and various energy functionals on $\cE_1$; we prove various results which are stated in \cite{he182}[Section 2.2].\\

{\bf Acknowledgement:} The first named author wants to thank Prof. Xiuxiong Chen for encouragements. The first named author is also grateful for T. Darvas for his enlightening influence in pluripotential theory, which makes it possible for us to extend relevant results in pluripotential theory to Sasaki setting. The first named author is supported in part by an NSF grant, award no. 1611797. The second named author wants to thank Prof. Xiangyu Zhou and Prof. Yueping Jiang for help and encouragements. He is partially supported by NSFC 11701164.

\numberwithin{equation}{section}
\numberwithin{thm}{section}

\section{Preliminary on Sasaki geometry}

A good reference on Sasaki geometry can be found in the monograph \cite{BG} by Boyer-Galicki. 
Let $M$ be a compact differentiable manifold of dimension $2n+1 (n\geq 1)$. A Sasaki structure on $M$ is defined to be a K\"ahler cone structure on $X=M\times \R_{+}$, i.e. a K\"ahler metric $(g_X, J)$ on $X$  of the form $$g_X=dr^2+r^2g,$$ where $r>0$ is a coordinate on $\R_{+}$, and $g$ is a Riemannian metric on $M$. We call $(X, g_X, J)$ the \emph{K\"ahler cone} of $M$. We also identify $M$ with the link $\{r=1\}$ in $X$ if there is no ambiguity.  Because of the cone structure, the K\"ahler form on $X$ can be expressed as 
$$\omega_X=\frac{1}{2}\sqrt{-1}\p\bp r^2=\frac{1}{2}dd^c r^2.$$
We denote by $r\p_r$ the homothetic vector field on the cone, which is easily seen to be a real holomorphic vector field. 
A tensor $\alpha$ on $X$ is said to be of homothetic degree $k$ if 
$$\cL_{r\p_r} \alpha=k\alpha.$$
In particular, $\omega$ and $g$ have homothetic degree two, while $J$ and $r\p_r$ has homothetic degree zero.
We define the \emph{Reeb vector field} $$\xi=J(r\p_r).$$
Then $\xi$ is a holomorphic Killing field on $X$ with homothetic degree zero. Let $\eta$ be the dual one-form to $\xi$:
 \[\eta(\cdot)=r^{-2}g_X(\xi, \cdot)=2d^c \log r=\i (\bp-\p)\log r\ .\] 
We also use $(\xi, \eta)$ to denote the restriction of them on $(M, g)$.  Then we have 
\begin{itemize}
\item $\eta$ is a contact form on $M$, and $\xi$ is a Killing vector field on $M$ which we also call the Reeb vector field;
\item $\eta(\xi)=1, \iota_{\xi} d\eta(\cdot)=d\eta (\xi, \cdot)=0$;
\item the integral curves of $\xi$ are geodesics.  
\end{itemize}

The Reeb vector field $\xi$ defines a foliation $\cF_\xi$ of $M$ by geodesics. There is a classification of Sasaki structures according to the global property of the leaves. If all the leaves are compact, then $\xi$ generates a circle action on $M$, and the Sasaki structure is called {\it quasi-regular}. In general this action is only locally free, and we get a polarized orbifold structure on the leaf space. If the circle action is globally free, then the Sasaki structure is called {\it regular}, and the leaf space is a polarized K\"ahler manifold. If $\xi$ has a non-compact leaf the Sasaki structure is called {\it irregular}. 

One can also understand Sasaki structure through contact metric structure. 
There is an orthogonal decomposition of the tangent bundle \[TM=L\xi\oplus \cD,\] where $L\xi$ is the trivial  bundle generalized by $\xi$, and $\cD=\Ker (\eta)$.
The metric $g$ and the contact form $\eta$ determine a $(1,1)$ tensor field $\Phi$ on $M$ by
\[
g(Y, Z)=\frac{1}{2} d\eta(Y, \Phi Z), Y, Z\in \Gamma(\cD). 
\]
$\Phi$ restricts to an almost complex structure on $\cD$: \[\Phi^2=-\mathbb{I}+\eta\otimes \xi. \] 

Since both $g$ and $\eta$ are invariant under $\xi$, there is a well-defined K\"ahler structure $(g^T, \omega^T, J^T)$ on the local leaf space of the Reeb foliation. We call this a \emph{transverse K\"ahler structure}. In the quasi-regular case, this is the same as the K\"ahler structure on the quotient. Clearly 
$\omega^T=2^{-1}d\eta. $
The upper script $T$ is used to denote both the transverse geometric quantity, and the corresponding quantity on the bundle $\cD$. For example we have on $M$
$$g=\eta\otimes  \eta+g^T.$$
From the above discussion it is not hard to see that there is an intrinsic formulation of a Sasaki structure as a compatible integrable pair $(\eta, \Phi)$, where $\eta$ is a contact one form and $\Phi$ is a almost CR structure on $\mathcal D=\Ker \eta$. Here ``compatible" means  first that 
$d\eta(\Phi U, \Phi V)=d\eta(U, V)$ for any $U, V\in \mathcal D$, and $d\eta(U, \Phi U)>0$ for any non zero $U\in \mathcal D$. Further we require $\mathcal L_{\xi}\Phi=0$, where $\xi$ is the unique vector field with $\eta(\xi)=1$, and $d\eta(\xi, \cdot)=0$. 
$\Phi$ induces a splitting $$\mathcal  D\otimes \C=\mathcal D^{1,0}\oplus \mathcal D^{0,1}, $$
 with $\overline{\mathcal D^{1,0}}=\mathcal D^{0,1}$. 
``Integrable" means that $[\mathcal D^{0,1}, \mathcal D^{0,1}]\subset \mathcal D^{0,1}$. This is equivalent to that the induced almost complex structure on the local leaf space of the foliation by $\xi$ is integrable. For more discussions on this, see \cite{BG} Chapter 6. 

\begin{defn} A $p$-form $\theta$ on $M$ is called basic if
\[
\iota_\xi \theta=0, L_\xi \theta=0.
\]
Let $\Lambda^p_B$ be the  bundle of basic $p$-forms and $\Omega^p_B=\Gamma(S, \Lambda^p_B)$ the set of sections of $\Lambda^p_B$.  
\end{defn}
The exterior differential preserves basic forms. We set $d_B=d|_{\Omega^p_B}$. 
Thus the subalgebra $\Omega_{B}(\cF_\xi)$ forms a subcomplex of the de Rham complex, and its cohomology ring $H^{*}_{B}(\cF_\xi)$  is called the {\it basic cohomology ring}. When $(M, \xi, \eta, g)$ is a Sasaki structure, there is a natural splitting of $\Lambda^p_B\otimes \C$ such that
\[
\Lambda^p_B\otimes \C=\oplus \Lambda^{i, j}_B,
\]
where $\Lambda^{i, j}_B$ is the bundle of type $(i, j)$ basic forms. We thus have the well defined operators
\[
\begin{split}
\p_B: \Omega^{i, j}_B\rightarrow \Omega^{i+1, j}_B,\\
\bar\p_B: \Omega^{i, j}_B\rightarrow \Omega^{i, j+1}_B.
\end{split}
\]
Then we have $d_B=\p_B+\bar \p_B$. 
Set $d^c_B=\frac{1}{2}\i\left(\bar \p_B-\p_B\right).$ It is clear that
\[
d_Bd_B^c=\i\p_B\bar\p_B, d_B^2=(d_B^c)^2=0.
\]
We shall  recall the transverse complex (K\"ahler) structure on local coordinates.  Let $U_\alpha$ be an open covering of $M$ and $\pi_\alpha: U_\alpha\rightarrow V_\alpha\subset \C^n$ submersions
such that 
\[
\pi_\alpha\circ \pi^{-1}_\beta: \pi_\beta(U_\alpha\cap U_\beta)\rightarrow \pi_\alpha (U_\alpha\cap U_\beta)
\]
is biholomorphic when $U_\alpha\cap U_\beta$ is not empty. One can choose  local coordinate charts $(z_1, \cdots, z_n)$ on $V_\alpha$ and local coordinate charts $(x, z_1, \cdots, z_n)$ on  $U_\alpha\subset M$  such that $\xi=\p_x$, where we use the notations
\[
\p_x=\frac{\p}{\p x}, \p_i=\frac{\p}{\p z_i}, \bar \p_{ j}=\p_{\bar j}=\frac{\p}{\p \bar z_{ j}}=\frac{\p}{\p z_{\bar j}}. 
\]
The map $\pi_\alpha: (x, z_1, \cdots, z_n)\rightarrow (z_1, \cdots, z_n)$ is then the natural projection. There is an isomorphism, for any $p\in U_\alpha$,
\[
d\pi_\alpha: \cD_p\rightarrow T_{\pi_\alpha(p)}V_\alpha. 
\]
Hence the restriction of $g$ on $\cD$ gives an Hermitian metric $g^T_\alpha$ on $V_\alpha$ since $\xi$ generates isometries of $g$. 
One can verify that there is a well defined K\"ahler metric $g_\alpha^T$ on each $V_\alpha$ and 
\[
\pi_\alpha\circ \pi^{-1}_\beta: \pi_\beta(U_\alpha\cap U_\beta)\rightarrow \pi_\alpha (U_\alpha\cap U_\beta)
\]
gives an isometry of K\"ahler manifolds $(V_\alpha, g^T_\alpha)$.  The collection of K\"ahler metrics $\{g^T_\alpha\}$ on $\{V_\alpha\}$ can be used as an alternative definition of the transverse K\"ahler metric. The (local) transverse holomorphic (K\"ahler) structure is essential for us and we shall use these these charts enormously. We summarize as follows, 

\begin{defn}[Local foliation charts]
We can choose the open covering $\{U_\alpha\}$ of $M$ such that it a local product structure for each $\alpha$, determined by its foliation structure and transverse complex structure. That is, there are charts
\begin{equation*}
\Psi_\alpha: U_\alpha\rightarrow W_\alpha\subset \R\times \C^n,
\end{equation*}
where $W_\alpha= (-\delta, \delta) \times V_\alpha.$ For a point $p\in W_\alpha$, we write $p=(x, z_1, \cdots, z_n)$ with $\xi=\p_x$ and $V_\alpha= B_r(0)\subset \C^n$ for $0<r$ . We assume that $\delta, r$ are sufficiently small depending only on $(M, \xi, \eta, g)$, and $\omega^T_\alpha$ is uniformly equivalent to an Euclidean metric on each $V_\alpha=B_r\subset \C^n$,
\[
\frac{1}{2}\delta_{i\bar j}\leq \omega^T_\alpha\leq 2\delta_{i\bar j}
\] 
\end{defn}

In Sasaki geometry, it is often mostly convenient to work with these  charts when we need to consider the Sasaki structure locally. 
For each $U_\alpha$, we assume it is contained in  the geodesic normal neighborhood of its ``center", $\Psi_\alpha^{-1}(0, 0, \cdots, 0)$, by choosing $\delta, r$ small enough. We call these charts \emph{foliation charts}.  
The existence of foliation charts is well-known in the subject, see \cite{GKN}; in particular, any Sasaki metric $g$ can be locally expressed in terms of a real function  of $2n$ variables. Given a foliation chart $W_\alpha=(-\delta, \delta)\times V_\alpha$,  for $(x, z_1, \cdots, z_n)\in U_\alpha$, locally there exists a strictly plurisubharmonic function $h: V_\alpha\rightarrow \R$, and  the Sasaki structure reads
\begin{equation}\label{foliation02}
\begin{split}
&\xi=\p_x; \; \eta=dx-\sqrt{-1} \sum_i(h_i dz^i-h_{\bar i} dz^{\bar i})\\
&\omega^T=\sqrt{-1} h_{i\bar j} dz^i\wedge dz^{\bar j};\; g=\eta\otimes \eta+2h_{i\bar j} dz^i\otimes dz^{\bar j}
\end{split}
\end{equation}
If we consider a Sasaki structure induced by a transverse K\"ahler potential $\phi$, then locally we have $h\rightarrow h+\phi$. In particular, we have
 \[\eta_\phi=\eta+\sqrt{-1}(\bar\p-\p) \phi, \omega_\phi=\omega^T+\sqrt{-1}\p\bar\p\phi.\]

We shall also use holomorphic charts on its K\"ahler cone $X$. There exist indeed  holomorphic charts on the K\"ahler cone $X$ which are closely related to foliation charts on $M$. This seems to be much less well-known and we shall describe them now.

\begin{lemma}[Holomorphic coordinates on the K\"ahler cone]\label{chart}For a Sasaki structure locally generated by a plurisubharmonic function $h: V_\alpha\rightarrow \R$ in foliations charts on $M$, then the following gives a local holomorphic structure on its K\"ahler cone $X$, for $w=(w_0, \cdots, w_n)\in \tilde U_\alpha\subset \C\times V_\alpha$, 
\begin{equation}\label{holo01}
w_0=\log r-h(z, \bar z)+\sqrt{-1}x, w_i=z_i, i=1, \cdots, n, z=(z_1, \cdots, z_n)
\end{equation}
The holomorphic structure $J$ is given by the holomorphic coordinates $w=(w_0, \cdots, w_n)$,
\begin{equation}\label{holo02}
J\frac{\p}{\p w_i}=\sqrt{-1}\frac{\p }{\p w_i}, i=0, \cdots, n.
\end{equation}
\end{lemma}
\begin{proof}Given \eqref{foliation02}, it is straightforward to check that \eqref{holo01} gives a holomorphic chart satisfying \eqref{holo02}. 
\end{proof}

\begin{rmk}These holomorphic charts would be very useful for us later, in particular  when we consider plurisubharmonic functions on $X$ and transverse plurisubharmonic functions on $M$.
The explicit holomorphic charts given above seem to appear in literature first time to our knowledge, while the foliation charts are well-known. 
\end{rmk}

When the Reeb vector field $\xi$ is irregular, the local foliation charts satisfy cocycle condition but they do not give a manifold (or orbifold) structure of the quotient $M/\cF_\xi$. 
We shall recall \emph{Type-I deformation} defined in \cite{BGM}. Let $(M, \xi_0, \eta_0, g_0)$ be a compact Sasaki manifold, denote its automorphism group by  $\text{Aut}(M, \xi_0, \eta_0, g_0)$. We fix a torus \[T\subset  \Aut(M, \xi_0, \eta_0, g_0)\; \text{such that}\; \xi_0\in \mathfrak{t}=\text{Lie algebra}(T).\]

\begin{defn}[Type-I deformation]\label{type-01}
Let $(M, \xi_0, \eta_0, g_0)$ be a $T$-invariant Sasaki structure. For any $\xi\in \mathfrak{t}$ such that $\eta_0(\xi)>0$. We define a 
 new Sasaki structure on $M$ explicitly as 
\begin{equation*}\label{e-2-7}
\eta=\frac{\eta_0}{\eta_0(\xi)}, \Phi=\Phi_0-\Phi_0\xi\otimes \eta, g=\eta\otimes \eta+\frac{1}{2}d\eta(\I\otimes \Phi).
\end{equation*}
\end{defn}
Note that under Type-I deformation, the essential change is the Reeb vector field $\xi_0 \leftrightarrow \xi$ and this construction can be done vice versa.

\section{The space of transverse K\"ahler potentials}

In this section we consider the space of transverse K\"ahler potentials on a compact Sasaki manifold through its transverse K\"ahler structure.
It turns out to be necessary to consider these objects not only from point of view of PDE, but also from the point of view of pluripotential theory. 
Geometric pluripotential theory on K\"ahler manifolds turns out to be one crucial piece in the proof of properness conjecture \cite{BDL2, CC3}. We refer \cite{GZ, D4} and references therein for details of pluripotential theory. We need to extend these results to Sasaki manifolds.  This would form a crucial piece for existence of cscs on Sasaki manifolds as well, see \cite{he182} for details. We start with the basic notion of quasiplurisubharmonic functions on Sasaki manifolds.

\subsection{The quasiplurisubharmonic functions on Sasaki manifolds}

Denote $\cH=\{\phi \in C^\infty_B(M): \omega_\phi=\omega^T+\sqrt{-1}\p_B\bar\p _B \phi>0\}$,
the space of transverse K\"ahler potentials on a Sasaki manifold $(M, \xi, \eta, g)$.  Given $\phi\in \cH$, it defines a new Sasaki structure, $(M, \xi, \eta_\phi, g_{\eta_\phi})$ as follows,
\begin{equation*}
\eta_\phi=\eta+2d^c_B\phi, \omega_\phi=\omega^T+\sqrt{-1}\p_B\bar\p_B \phi, g_{\eta_\phi}=\eta_\phi\otimes \eta_\phi+\omega_\phi
\end{equation*}

The most relevant results in pluripotential theory for us lie in  in \cite{GZ}, \cite{BBEGZ}[Section 2] and \cite{D4}. Part of them has been done by van Covering \cite{VC}[Section 2], including the Monge-Ampere operator and weak convergence, with main focus on $L^\infty$ and $C^0$ potentials. We shall need most of results on the energy classes $\cE$ and $\cE_p$ (defined below).  \\

Given a Sasaki structure $(M, \xi, \eta, g)$, we recall the following definition,
\begin{defn}An $L^1$, upper semicontinuous (usc) function $u: M\rightarrow \R\cup \{-\infty\}$ is called a transverse $\omega^T$-plurisubharmonic (TPSH for short) if $u$ is invariant under the Reeb flow, and $u$ is $\omega^T$-plurisubharmonic on each local foliation chart $V_\alpha$, that is $\omega^T_\alpha+\sqrt{-1}\p_B\bar\p_B u\geq 0$ as a positive closed $(1, 1)$-current on $V_\alpha$.  \end{defn}
It is apparent that the definition above does not depend on the choice of foliation charts. Indeed, $u$ is invariant along the flow of $\xi$ and we extend $u$ trivially in the cone direction to a function on cone.  Using the holomorphic structure on the cone (see Lemma \ref{chart}), $u$ is a TPSH if and only if $\omega^T+\sqrt{-1}\p\bar \p u\geq 0$ is a closed, positive $(1, 1)$ current $X$.
We use the notation,
\begin{equation*}
\text{PSH}(M, \xi, \omega^T)=\{u\in L^1(M), u\; \text{is usc and invariant under the Reeb flow}; \omega_u\geq 0\}
\end{equation*}

One of the cornerstones of Bedford-Taylor theory \cite{BT1} is to associate a complex Monge-Ampere measure to a bounded psh function. Their construction generalizes to bounded K\"ahler potentials in a straightforward manner \cite{GZ} and it has direct adaption to Sasaki setting. We refer to \cite{VC}[Section 2] and Section \ref{CMA001} for definition of complex Monge-Ampere measures $\omega_u^n\wedge \eta$ for $u\in \text{PSH}(M,\xi,\omega^T)\cap L^{\infty}$ on Sasaki manifolds, which is a direct adaption of Bedford-Taylor theory \cite{BT1}. 

\begin{prop}\label{weakcon0}
Suppose that the sequence $u_j \in \text{PSH}(M,\xi,\omega^T)\cap L^{\infty}$ decreases to $u \in \text{PSH}(M,\xi,\omega^T)\cap L^{\infty}$. Then for $k=1, \cdots, n$, we have the following weak convergences of complex Monge-Ampere measures, \begin{equation}\label{measure001}\omega_{u_j}^k\wedge (\omega^T)^{n-k} \wedge \eta \rightarrow \omega_u^k\wedge(\omega^T)^{n-k} \wedge \eta\end{equation}
\end{prop}
\begin{proof}By applying a partition of unity subordinated to covering by foliation charts, we need to show that for $f\in C^\infty$, supported on a foliation chart $W_\alpha=(-\delta, \delta)\times V_\alpha$
\begin{equation}\label{measure002}
\int_M f \omega_{u_j}^k\wedge (\omega^T)^{n-k} \wedge \eta\rightarrow \int_M f \omega_{u}^k\wedge (\omega^T)^{n-k} \wedge \eta
\end{equation}
We should emphasize that $f$ is not a basic function in general. The weak convergence in K\"ahler setting implies that for each $x\in (-\delta, \delta)$
\[
\int_{V_\alpha} f(x, z, \bar z) \omega_{u_j}^k\wedge (\omega^T)^{n-k}\rightarrow \int_{V_\alpha} f(x, z, \bar z) \omega_{u}^k\wedge (\omega^T)^{n-k} .
\]
Note that for each $x$, $f$ is supported on $V_\alpha$. Taking integration with respect to $dx$, this leads to \eqref{measure002}, since on $W_\alpha$, $\omega_{u}^k\wedge (\omega^T)^{n-k} \wedge \eta=\omega_{u}^k\wedge (\omega^T)^{n-k} \wedge dx$ as a product measure. 

\end{proof} 

The following Bedford-Taylor identity in Sasaki setting would be used numerously,

\begin{prop}\label{measure1}For $u, v\in \text{PSH}(M, \xi, \omega^T)\cap L^\infty$, 
\begin{equation}\label{si1}
\chi_{\{u>v\}}\omega^n_{\max(u, v)}\wedge \eta=\chi_{\{u>v\}} \omega^n_u\wedge \eta. 
\end{equation}
\end{prop}
\begin{proof}We only need to prove this in foliation charts. Recall for each foliation chart $W_\alpha=(-\delta, \delta)\times V_\alpha$,  $V_\alpha=B_r(0)\subset \C^n$ gives the local transverse complex structure. For a point $p\in W_\alpha$, we write $p=(x, z)$ with $\xi=\p_x$.  Given $u\in \text{PSH}(M, \xi, \omega^T)\cap L^\infty$ it defines a K\"ahler current $\omega_u^n$ on $V_\alpha$.  Since both $u$ and $v$ are basic functions, $u, v$ are independent of $x$ in $W_\alpha$. Hence on $W_\alpha\cap \{u>v\}=(-\delta, \delta)\times \{z\in V_\alpha: u>v\}.$ Note that $\omega^T_u\wedge \eta$ is invariant along the Reeb direction, and it coincides with the product measure $dx\wedge \omega_u^n$ on $W_\alpha=(-\delta, \delta)\times V_\alpha$. 
On each $W_\alpha$, we have
\[
\begin{split}
&\chi_{\{(x, z)\in W_\alpha: u>v\}}\omega^n_{\max(u, v)}\wedge \eta=\chi_{\{z\in V_\alpha: u>v\}}\omega^n_{\max(u, v)}\wedge dx\\
&\chi_{\{(x, z)\in W_\alpha: u>v\}} \omega^n_u\wedge \eta=\chi_{\{z\in V_\alpha: u>v\}}\omega^n_{u}\wedge dx.
\end{split}
\]
To prove \eqref{si1}, it reduces to show that 
\[
\chi_{\{z\in V_\alpha: u>v\}}\omega^n_{\max(u, v)}=\chi_{\{z\in V_\alpha: u>v\}}\omega^n_{u}. 
\]
This is just the Bedford-Taylor identity \cite{BT1}.  
\end{proof}
It is possible to generalize the Bedford-Taylor constructions to a much larger class on a compact K\"ahler manifold, see Guedj-Zeriahi \cite{GZ}. The reference \cite{D4}[Section 2] is sufficient for our purpose. These constructions in K\"ahler setting have a direct extension to Sasaki setting, where Proposition \ref{measure1} plays an important role. First we prove the following well-known result in pluripotential theory. 

\begin{prop}\label{upperbound01}There exists $C=C(M, g)$ such that for any $u\in \text{PSH}(M, \xi, \omega^T)$,
\begin{equation*}
\sup_M u\leq \frac{1}{\text{Vol}(M)}\int_M u d\mu_g+C
\end{equation*}
\end{prop}

\begin{proof}When $u$ is $C^2$ this is obvious by the fact that $\Delta_g u+n\geq 0$. In general we can prove this using sub-mean value property of plurisubharmonic functions, similar as in \cite{D4}[Lemma 3.45]. In this proof we can either use foliation charts on $M$ or K\"ahler cone structure on $X=C(M)$. We use foliation charts in this argument. 

We assume $\sup_M u=0$ and want to show that the integration of $u$ is uniformly bounded below. 
We cover $M$ by nested foliation charts $U_k\subset W_k\subset M$ such that there exist diffeomorphisms $\varphi_k: B(0, 4)\times (-2\delta, 2\delta)\rightarrow W_k$ with $\varphi_k: (B(0, 1)\times (-\delta, \delta))=U_k$, where $\delta$ is a fixed positive constant and $B(0, 1)\subset B(0, 4)\subset \C^n$ are Euclidean balls in $\C^n$. We assume that $(z, x)\in B(0, 4)\times (-2\delta, 2\delta)$ such that $z\in B(0, 4)$ represents transverse holomorphic charts and $x\in (-2\delta, 2\delta)$ represents the Reeb direction (i.e. $\xi=\p_x$). 
On each $W_k$, there exists a function $\psi_k=\psi_k(z)$ such that $\omega^T=\sqrt{-1}\p_z\bar\p_z \psi_k$.  
Note that we only need to show that, there exists a uniformly bounded constant $C>0$, such that
\[
\int_{U_k} u d\mu_g\geq -C, k\in \{1, \cdots, N\}
\]
Note that $u$ is basic, we have
\[
\int_{B(0, 1)\times (-\delta, \delta)} u\circ \varphi_k d\mu_{x, z}=2\delta \int_{B(0, 1)} u\circ \varphi_k(z, x_0) d\mu_z, x_0\in (-\delta, \delta)
\]
where $d\mu_{x, z}$ and $d\mu_z$ are Euclidean measure on $\C^n\times \R$ and $\C^n$ respectively. Hence we only need to show that 
\begin{equation}\label{step0}
\int_{B(0, 1)} u\circ \varphi_k(z, x_0) d\mu_z\geq -C, k\in \{1, \cdots, N\}
\end{equation}
Note that by our construction, $(\psi_k+u)\circ \varphi_k$ is independent of $x$ and is plurisubharmonic on $B(0, 4)$ for each $k$. As $u$ is usc, its supremum is realized at some point $p_1\in M$ such that $u\leq u(p_1)=0$. Since $U_k$ covers $M$, we can assume $p_1\in U_1$ with the coordinate $\varphi_1(z_1, x_1)=p_1$ for some $(z_1, x_1)\in B(0, 1)\times (-\delta, \delta)$. Note that since $u$ is basic, hence it is independent of $x$-coordinate we can also take $x_1=0$.  Since $B(z_1, 2)\subset B(0, 4)$, we have the following sub-mean value property for $(\psi_1+u)\circ \varphi_1$,
\[
\psi_1\circ \varphi_1(z_1, 0)=(\psi_1+u)\circ \varphi_1(z_1, 0)\leq \frac{1}{\mu(B(z_1, 2))}\int_{B(z_1, 2)} (\psi_1+u)\circ \varphi_1(z, 0) d\mu_z
\]
Since $u\leq 0$ and $B(0, 1)\subset B(z_1, 2)$, there exists $C_1>0$, independent of $u$, such that
\begin{equation}\label{step01}
\int_{B(0, 1)} u\circ \varphi_1 d\mu_z\geq -C_1. 
\end{equation}
Since $\{U_k\}_k$ covers $M$, we can assume $U_1$ intersects $U_2$. We can choose $r_2>0$, such that $\varphi_2(B(z_2, r_2)\times (\delta_1, \delta_2))\subset U_1\cap U_2$ for some $B(z_2, r_2)\subset B(0, 4)$ and $-\delta<\delta_1<\delta_2<\delta$. 
Since $u\leq 0$, it follows that there exists $\tilde C_1>0$, independent of $u$ ($\tilde C_1$ depends only on $C_1$, $r_2$ and $\psi_2$), such that
\[
\frac{1}{\mu(B(z_2, r_2))}\int_{B(z_2, r_2)}  (u+\psi_2) \circ \varphi_2 d\mu_z\geq -\tilde C_1. 
\]
Since $(u+\psi_2) \circ \varphi_2$ is plurisubharmonic in $B(0, 4)$, we can obtain that 
\[
\frac{1}{\mu(B(z_2, 2))}\int_{B(z_2, 2)}  (u+\psi_2) \circ \varphi_2 d\mu_z\geq \frac{1}{\mu(B(z_2, r_2))}\int_{B(z_2, r_2)}  (u+\psi_2) \circ \varphi_2 d\mu_z\geq -\tilde C_1. 
\]
Since $u\leq 0$ and $B(0, 1)\subset B(z_2, 2)$, we obtain for some $C_2>0$
\[
\int_{B(0, 1)} u\circ \varphi_2 d\mu_z\geq -C_2
\]
We continue this process to consider that $U_1\cup U_2$ intersects a member, say $U_3$. After at most $N-2$ step we prove \eqref{step0}. 
\end{proof}

As a direct consequence, we know the following (see \cite{Demailly}[Proposition I.5.9]),
\begin{prop}\label{compactness001}
The set $\cC=\{u\in \text{PSH}(M, \xi, \omega^T): \sup_M u\leq C\}$ is bounded in $L^1$ and it is precompact in $L^1(d\mu_g)$ topology. 
\end{prop}

\begin{proof}By the above we know that $\sup_M u$ bounded above implies that $\int_M |u|d\mu_g$ is uniformly bounded. By the Motel property of subharmonic functions and plurisubharmonic functionals \cite{Demailly}[Proposition I.4.21, Proposition I.5.9] that $\cC$ is precompact with respect to $L^1(d\mu_g)$ topology. Note that in Sasaki setting we apply the compactness of plurisubharmonic functions to nested foliations charts $U_k\subset W_k$ as above for $\omega^T_k$-plurisubharmonic functions locally, that $\cC$ is precompact in $L^1$ topology in each $U_k$. After passing by subsequence if necessary, we can then get weak compactness of $\cC$ with respect to $L^1(d\mu_g)$ topology. 
\end{proof}

Let $v\in \text{PSH}(M, \xi, \omega^T)$. 
For $h\in \R$, we denote $v_h=\max\{v, -h\}$ to be the \emph{canonical cutoffs} of $v$. By Proposition \ref{upperbound01}, $v_h\in L^\infty$. It is evident that $v_h$ is invariant under the Reeb flow and hence $v_h\in \text{PSH}(M, \xi, \omega^T)\cap L^\infty$. If $h_1<h_2$, then Proposition \ref{measure1} implies that
\[
\chi_{\{v>-h_1\}}\omega^n_{v_{h_1}}\wedge \eta=\chi_{\{v>-h_1\}} \omega^n_{v_{h_2}}\wedge \eta\leq \chi_{\{v>-h_2\}} \omega^n_{v_{h_2}}
\]
Hence $\chi_{\{v>-h\}}\omega^n_{v_{h}}\wedge \eta$ is an increasing sequence of Borel measure on $M$ with respect to $h$. This leads to the following definition, 

\begin{defn}
We define
\begin{equation}\label{measure2}
\omega^n_v\wedge \eta:=\lim_{h\rightarrow \infty}\chi_{\{v>-h\}}\omega^n_{v_{h}}\wedge \eta
\end{equation}
\end{defn}

We shall emphasize that by the definition above, we have for any Borel set $B\subset M$, 
\begin{equation}\label{measure21}
\int_B \omega^n_v\wedge \eta=\lim_{h\rightarrow \infty}\int_B \chi_{\{v>-h\}}\omega^n_{v_{h}}\wedge \eta
\end{equation}
Hence the convergence in \eqref{measure2} is a stronger notion than the weak convergence of measures.

To proceed, we need the following approximation of TPSH functions. Our proof uses the K\"ahler cone structure and builds up on Blocki-Kolodziej \cite{BK}. 
\begin{lemma}\label{BK}Given $u\in \text{PSH}(M, \xi, \omega^T)$, there exists a decreasing sequence $\{u_k\}_k\subset \cH$ such that $u_k$ converges to $u$. 
\end{lemma}

\begin{proof}
First we assume that $u$ has zero Lelong number. 
Recall $X$ is the K\"ahler cone and we identify $M$ with the link $\{r=1\}\subset X$.  For $u\in \text{PSH}(M, \xi, \omega^T)$, we extend $u$ to be a function on $X$ such that $u(r, p)=u(p)$, for any $r>0$.
We recall that $\omega^T=\frac{1}{2}d\eta=dd^c(\log r)=\sqrt{-1}\p\bar\p (\log r)$. Hence for $u\in \text{PSH}(M, \xi, \omega^T)$, we have the following,
\[
\sqrt{-1}\p\bar \p (\log r+u)\geq 0
\]
In other words, $v=u+\log r$ is a plurisubharmonic function on $X$. This is transparent in foliations charts and corresponding holomorphic charts as in Lemma \ref{chart}. 
Let $h_\alpha$ be a local potential for $\omega^T$ in a foliation chart $V_\alpha$, and we write $h=h(w_1, \bar w_1, \cdots, w_n, \bar w_n)$ in the holomorphic chart on cone, then 
 $\log r=h_\alpha+\text{Re}(w_0)$. Denote $\omega_X$ to be the K\"ahler form on $X$. Since $u$ has zero Lelong number,  applying Blocki-Kolodziej \cite{BK}[Theorem 2], we  get a sequence of functions $v_k$ converges to $u$, decreasing in $k$, such that on $X^{'}\subset X$
\begin{equation}\label{approx001}
\sqrt{-1}\p\bar \p (v_k)+\omega^T+k^{-1}\omega_X\geq 0, X^{'}=\left\{2^{-1}\leq r\leq 2\right\}
\end{equation}
We can assume in addition that $v_k$ is invariant under the flow of $\xi$, by taking average with respect to the torus action generated by $\xi\in \text{Aut}(\xi, \eta, g)$. 
We define a basic function $u_k$ on $M$ such that, by taking $r=1$, $u_k=v_k|_{r=1}$. 

Now for any point on $X^{'}$, we choose holomorphic charts $\tilde U_\alpha$ as in Lemma \ref{chart} to cover $X^{'}$. We write the function in a holomorphic chart as \[v_k=v_k(\text{Re}(w_0), x, w_1, \bar w_1 \cdots, w_n, \bar w_n).\]  We recall the relation between the holomorphic charts and  the foliation charts,
\begin{equation}\label{chart02}
w_0=\log(r)+\sqrt{-1}x-h_\alpha(z, \bar z), w_i=z_i, i=1, \cdots, n.
\end{equation}
Note we assume that $v_k$ is invariant under the flow of $\xi$, hence $v_k$ is independent of $x=\text{Im}(w_0)$. 
We write $v_k$ as follows, using \eqref{chart02}, 
\[
v_k(\text{Re}(w_0), w_1, \bar w_1, \cdots, w_n, \bar w_n)=v_k(\log r-h(z, \bar z), z, \bar z)
\]
Locally this gives
\begin{equation}\label{basic01}
u_k(z, \bar z)= v_k(-h_\alpha(z, \bar z), z, \bar z). 
\end{equation}
The tangent space $T_pX$ is given by, in terms of coordinate $(r, x, z_1, \cdots, z_n)$,
\[
T_pX\otimes \C=\text{span}\left\{\frac{\p}{\p r}, r^{-1}\frac{\p }{\p x}, X_i=\frac{\p}{\p z_i}+\sqrt{-1} h_i \frac{\p }{\p x}, \bar X_j=\frac{\p}{\p \bar z_j}-\sqrt{-1} h_{\bar j} \frac{\p }{\p x} \right\}
\]
Note that the contact bundle $\cD_p=\text{span}\{X_i, X_{\bar i}, i=1, \cdots, n\}$.
For $p\in M\subset X$, we can assume that $h(z, \bar z)=\p h=\bar\p h=0, h_{i\bar j}=\delta_{i\bar j}$ at $p$, and hence \[T_pX=T_pM\oplus \left\{\frac{\p}{\p r}\right\}=\text{span}\left\{\frac{\p}{\p z_i}, \frac{\p}{\p \bar z_j}, r^{-1}\frac{\p}{\p x}, \frac{\p }{\p r}\right\}\]
By \eqref{approx001},  we compute (at $p$), 
\begin{equation}
\left(\sqrt{-1}\p\bar\p v_k+\omega^T+k^{-1}\omega_X\right)\left(\frac{\p}{\p z_i}, -\sqrt{-1}\frac{\p}{\p \bar z_i}\right)=-\p_t v_k+1+k^{-1}+(v_k)_{i\bar i}\geq 0,
\end{equation}
where $t$ stands for the first argument of $v_k$. 
This is equivalent to the following, on $M$ we have, 
\[
\sqrt{-1}\p_B\bar \p_B u_k+(1+k^{-1})\omega^T\geq 0.
\]
It is clear that $u_k$ converges to $u$, deceasing in $k$. Without loss of generality, we can assume that $u\leq -1$ and $u_k\leq 0$. It follows that $k(k+2)^{-1}u_k\in \cH$ such that $k(k+2)^{-1}u_k$ converges to $u$, decreasing in $k$. This completes the proof when $u$ has zero Lelong number.

Now suppose $u\in \text{PSH}(M, \xi, \omega^T)$. We consider the canonical cutoffs $u_j=\max\{u, -j\}\in \text{PSH}(M, \xi, \omega^T)\cap L^\infty$. By the above we know that for each $j$, there exists a sequence of smooth functions $\{v_j^k\}_k\subset \cH$ which decreases to $u_j$.  By adding a small constant $k^{-1}$ to each $v^k_j$, we can assume that $\{v^k_j\}_k$ strictly decreases (for each $j$). Then for each $k$, we can find $k_{j+1}$ such that 
\begin{equation}\label{selection01}v_{j+1}^{k_{j+1}}<v^k_j.\end{equation}Indeed we consider the open set $U^l:=\{x\in M: v_{j+1}^l<v^k_j\}$. Clearly $\{U^l\}_l$ is an increasing sequence of open sets such that $\cup_l U^l=M$, since \[\lim_{l\rightarrow \infty} v_{j+1}^l=u_{j+1}\leq u_j<v^k_j.\]  Since $M$ is compact, there exists $k_{j+1}$ such that $U^{k_{j+1}}=M$. By \eqref{selection01}, we can find a sequence $\{v_{j}^{k_j}\}_j$ inductively such that $v_j^{k_j}\searrow u$. This completes the proof. 
\end{proof}

\begin{rmk}The K\"ahler cone structure, in particular the relation between holomorphic charts and foliation charts as in Lemma \ref{chart} play a very important role in Sasaki setting. If the Reeb vector field is irregular, the approximation from transverse K\"ahler structure can produce local approximation. But it seems to be hard to patch such a local construction together when the Reeb vector field is irregular. Instead we do approximation on the K\"ahler cone.
We shall mention that in \eqref{selection01}, the assumption that each sequence $\{v_j^k\}_k$ strictly decreases is necessary. For example, we can take $u=1$ over $[0, 1]$, $v=0$ over $[0, 1)$ and $v(1)=1$. We can choose $u_k=1$ for each $k$, and $v_k(x)=x^k+k^{-1}$. Then $v\leq u$ and $\{u_k\}_k$ decreases to $u$ and $v_k$ (strictly) decreases to $v$. But for $\{u_k\}_k$ and $\{v_k\}_k$, \eqref{selection01} does not hold: given $u_k$, there does not exist $l$ such that $v_{l}\leq u_k$ since $v_{l}(1)>1$ for all $l$. 
\end{rmk}

As a direct consequence, we have the following (just as in K\"ahler setting),
\begin{prop}\label{volume}
For $u\in \text{PSH}(M, \xi, \omega^T)\cap L^\infty$,
\begin{equation}\label{approx0}
\text{Vol}(M):=\int_M  \omega^n_u\wedge \eta=\int_M \omega^n_T\wedge \eta 
\end{equation}
\end{prop}

\begin{proof}
By Lemma \ref{BK}, we can choose a smooth sequence $u_k$ converges to $u$ as a decreasing sequence. It then follows from Bedford-Taylor theory (see Proposition \ref{weakcon0}) that $\omega_{u_k}^n\wedge \eta$ converges to $\omega_u^n\wedge \eta$ weakly, we obtain \eqref{approx0}. 
\end{proof}

It is then clear that, given \eqref{measure2}, we have only $\int_M \omega^n_v\wedge \eta\leq \text{Vol}(M)$ for $v\in \text{PSH}(M, \xi, \omega^T)$. 
\begin{defn}We define the full-mass elements in $\text{PSH}(M, \xi, \omega^T)$ as
\begin{equation}
\cE(M, \xi, \omega^T):=\{v: v\in \text{PSH}(M, \xi, \omega^T)\; \text{such that}\; \int_M  \omega^n_v\wedge \eta=\text{Vol}(M)\}
\end{equation}
\end{defn}

As in K\"ahler case, many of the properties that hold for bounded TPSH functions hold for elements of $\cE(M, \xi, \omega^T)$ as well. We include the \emph{comparison principle, monotonicity property} and \emph{generalized Bedford-Taylor identity} as follows. These  properties are proved in \cite{GZ} for K\"ahler setting. Given \eqref{si1} and \eqref{approx0}, our proof follows almost identical as in K\"ahler setting (see \cite{GZ}[Theorem 1.5, Proposition 1.6, Corollary 1.7]). Nevertheless we include the details. 

\begin{prop}[Comparison principle]Suppose $u, v\in \cE(M, \xi, \omega^T)$. Then
\begin{equation}\label{cp01}
\int_{\{v<u\}}\omega_u^n\wedge \eta\leq \int_{\{v<u\}} \omega_v^n \wedge \eta. 
\end{equation}
\end{prop}

\begin{proof}First we show \eqref{cp01} for $u, v$ bounded. Using \eqref{si1} we write 
\begin{equation*}
\begin{split}
\int_{\{v<u\}}\omega_u^n\wedge\eta=&\int_{\{v<u\}}\omega^n_{\max\{u, v\}}\wedge \eta=\int_M \omega^n_{\max\{u, v\}}\wedge \eta-\int_{\{u\leq v\}}\omega^n_{\max\{u, v\}}\wedge \eta\\
\leq&\int_M \omega^n_{\max\{u, v\}}\wedge \eta-\int_{\{u<v\}}\omega^n_{\max\{u, v\}}\wedge \eta\\
\leq&\text{Vol}(M)-\int_{\{u<v\}}\omega^n_{\max\{u, v\}}\wedge \eta.
\end{split}
\end{equation*}
Using Proposition \ref{volume} and Proposition \ref{measure1} we write the above as
\begin{equation*}
\int_{\{v<u\}}\omega_u^n\wedge\eta\leq \int_{M}\omega^n_v\wedge \eta-\int_{\{u<v\}}\omega^n_v\wedge \eta\leq\int_{\{v\leq u\}}\omega^n_v\wedge \eta
\end{equation*}
Replacing $v$ by $v+\epsilon$, we have
\begin{equation*}
\int_{\{v+\epsilon<u\}}\omega_u^n\wedge\eta\leq \int_{\{v+\epsilon \leq u\}}\omega^n_v\wedge \eta
\end{equation*}
We get \eqref{cp01} for bounded potentials by letting $\epsilon\rightarrow 0$, 
noting that \[\{v<u\}=\cup_{\epsilon>0}\{v+\epsilon<u\}=\cup_{\epsilon>0}\{v+\epsilon\leq u\}.\]

In general, let $u_l=\max\{u, -l\}, v_{k}=\max\{v, -k\}, l, k\in \N$ be the canonical cutoffs of $u, v$ respectively. We apply \eqref{cp01} for these to get
\begin{equation*}
\int_{\{v_l<u_k\}} \omega^n_{u_k}\wedge \eta\leq \int_{\{v_l<u_k\}} \omega^n_{v_l}\wedge \eta. 
\end{equation*}
Together with the inclusions
$
\{v_l<u\}\subset \{v_l< u_k\}\subset\{v<u_k\}
$
we have
\begin{equation}\label{cp02}
\int_{\{v_l<u\}} \omega^n_{u_k}\wedge \eta\leq \int_{\{v<u_k\}} \omega^n_{v_l}\wedge \eta. 
\end{equation}
Letting $l\rightarrow \infty$ and using the definition \eqref{measure2} on $\omega^n_{v_l}\wedge \eta$, \eqref{cp02} gives
\[
\int_{\{v<u\}} \omega^n_{u_k}\wedge \eta\leq \int_{\{v<u_k\}} \omega^n_{v}\wedge \eta. 
\]
Letting $k\rightarrow \infty$ and using the definition \eqref{measure2} on $\omega^n_{u_k}\wedge \eta$, we get
\[
\int_{\{v<u\}} \omega^n_{u}\wedge \eta\leq \int_{\{v\leq u\}} \omega^n_{v}\wedge \eta. 
\]
The replacing $v$ by $v+\epsilon$ in the above inequality, we can then argue as in the bounded case, taking the limit $\epsilon\rightarrow 0$ yields \eqref{cp01}. 
\end{proof}

\begin{prop}[Monotonicity property]\label{mp01}Suppose $u\in \cE(M, \xi, \omega^T)$ and $v\in \text{PSH}(M, \xi, \omega^T)$. If $u\leq v$ then $v\in \cE(M, \xi, \omega^T)$.
\end{prop}
\begin{proof}This is proved in \cite{GZ}[Proposition 1.6] and our argument is almost identical. First we show that $\psi=v/2\in \cE(M, \xi, \omega^T)$. We can assume that $u\leq v<-2$, hence $\psi<-1$. This normalization gives the following inclusions for the canonical cutoffs $u_j, v_j, \psi_j$,
\[
\{\psi\leq -j\}=\{\psi_j\leq -j\}\subset \{u_{2j}<\psi_j-j+1\}\subset \{u_{2j}\leq -j\}
\]
By Proposition \ref{cp01} and the inclusions above, we have
\[
\int_{\{\psi_j\leq -j\}}\omega_{\psi_j}^n\wedge \eta\leq \int_{ \{u_{2j}<\psi_j-j+1\}}\omega_{\psi_j}^n\wedge \eta\leq \int_{ \{u_{2j}<\psi_j-j+1\}}\omega_{u_{2j}}^n\wedge \eta\leq \int_{\{u_{2j}\leq -j\}}\omega_{u_{2j}}^n\wedge \eta. 
\]
Note that we have
\[
\int_{\{u_{2j}\leq -j\}}\omega_{u_{2j}}^n\wedge \eta=\text{Vol}(M)-\int_{\{u_{2j}>-j\}} \omega_{u_{2j}}^n\wedge \eta.
\]
Applying Proposition \ref{measure1} to $\max\{u_{2j}, -j\}=u_{j}$ on the set $\{u_{2j}>-j\}=\{u_j>-j\}$, we have
\[
\int_{\{u_{2j}>-j\}} \omega_{u_{2j}}^n\wedge \eta=\int_{\{u_{j}>-j\}} \omega_{u_{j}}^n\wedge \eta.
\]
It then follows that 
\[
\int_{\{u_{2j}\leq -j\}}\omega_{u_{2j}}^n\wedge \eta=\int_{\{u_j\leq -j\}}\omega_{u_{j}}^n\wedge \eta=\int_{\{u\leq -j\}}\omega_{u_{j}}^n\wedge \eta.
\]
By definition of $u\in \cE(M, \xi, \omega^T)$, it follows that, as $j\rightarrow \infty$, 
\[
\int_{\{\psi_j\leq -j\}}\omega_{\psi_j}^n\wedge \eta\leq \int_{\{u\leq -j\}}\omega_{u_{j}}^n\wedge \eta\rightarrow 0. 
\]
Hence $\psi=v/2\in \cE(M, \xi, \omega^T)$. To show that $v\in \cE(M, \xi, \omega^T)$, we observe that $\{v\leq -2j\}=\{\psi\leq -j\}$ and $\omega_{\psi_j}\geq \omega_{v_{2j}}/2$, hence
\[
\int_{\{v\leq -2j\}} \omega^n_{v_{2j}}\wedge \eta\leq 2^n \int_{\{v\leq -2j\}} \omega_{\psi_j}^n\wedge \eta\leq 2^n\int_{\{\psi\leq -j\}}\omega_{\psi_j}^n\wedge \eta.
\]
By letting $j\rightarrow \infty$, we can then conclude that $v\in \cE(M, \xi, \omega^T)$. 
\end{proof}

\begin{prop}[Generalized Bedford-Taylor identity]\label{GBTI}For $u\in \cE(M, \xi, \omega^T)$, $v\in \text{PSH}(M, \xi, \omega^T)$, then $\max\{u, v\}\in \cE(M, \xi, \omega^T)$ and
\begin{equation}\label{si2}
\chi_{\{u>v\}}\omega^n_{\max(u, v)}\wedge \eta=\chi_{\{u>v\}} \omega^n_u\wedge \eta. 
\end{equation}
\end{prop} 

\begin{proof}Our argument is identical to the K\"ahler setting; see \cite{GZ}[Corollary 1.7] and \cite{D4}[Lemma 2.5]. Proposition \ref{mp01} implies that $w:=\max\{u, v\}\in \cE(M, \xi, \omega^T)$. Now observe that $\max\{u_j, v_{j+1}\}=\max\{u, v, -j\}=w_j$. Since the cutoffs are bounded we have
\begin{equation}\label{measure31}
\chi_{\{u_j>v_{j+1}\}}\omega^n_{w_j}\wedge \eta=\chi_{\{u_j>v_{j+1}\}}\omega^n_{u_j}\wedge \eta
\end{equation}
By \ref{measure21}, we know that $\chi_{u>v}\omega^n_{u_j}\wedge \eta\rightarrow \chi_{u>v}\omega^n_{u}\wedge \eta$ and $\chi_{u>v}\omega^n_{w_j}\wedge \eta\rightarrow \chi_{u>v}\omega^n_{w}\wedge \eta$ as $j\rightarrow \infty$ (we also use the fact that $u, w\in \cE(M, \xi, \omega^T)$). Since 
\[
\{u>v\}\subset \{u_j>v_{j+1}\}\;\text{and}\; \{u_j>v_{j+1}\}\backslash \{u>v\}\subset \{u\leq -j\},
\]
it follows that
\[
0\leq (\chi_{\{u_j>v_{j+1}\}}-\chi_{\{u>v\}}) \omega^n_{u_j}\wedge \eta\leq \chi_{\{u\leq -j\}} \omega^n_{u_j}\wedge \eta\rightarrow 0.
\]
Similarly since 
\[
\{u_j>v_{j+1}\}\backslash \{u>v\}\subset \{w\leq -j\}
\]
we also obtain that 
\[
0\leq (\chi_{\{u_j>v_{j+1}\}}-\chi_{\{u>v\}}) \omega^n_{w_j}\wedge \eta\leq \chi_{\{w\leq -j\}} \omega^n_{w_j}\wedge \eta\rightarrow 0.
\]
By taking limit in \eqref{measure31} together with the limit facts above, we get the desired result. 
\end{proof}

Next we introduce finite energy class on Sasaki manifolds, following \cite{GZ}. 
By considering Young weights $\chi\in \cW^+_p$ (see \cite{D4}[Chapter 1] for a short introduction to Young wrights), one can introduce various finite energy subclasses of $\cE(M, \xi, \omega^T)$,
\[
\cE_\chi(M, \xi, \omega^T):=\{u\in \cE(M, \xi, \omega^T)\; \text{s. t.}\; E_\chi(u)<\infty\},
\]
where $E_\chi$ is the $\chi$-energy defined by
\[
E_\chi(u):=\int_M \chi(u)\omega^n_u\wedge \eta. 
\]
Of special importance are the weights $\chi^p(t)=|t|^p/p$ and the associated classes $\cE_p(M, \xi, \omega^T)$. For theses weights it is clear that $\cE_p(M, \xi, \omega^T) \subset \cE_1(M, \xi, \omega^T)$ for $p \geq 1$. We will need the following straightforward fact,
\begin{prop}
For any $u\in \cE_1(M, \xi, \omega^T)$, $u$ has Lelong number zero at every point. 
\end{prop}
\begin{proof}This is straightforward. We can assume $\sup u=0$.  For $u\in \cE_1(M, \xi, \omega^T)$, we have
\[
\int_M (-u) \omega^n_u\wedge \eta<\infty.
\]
We  consider locally  $(0, 0)\in W_\alpha=(-\delta, \delta)\times V_\alpha$ in a foliation chart. Then we have
\[
2\delta \int_{V_{\alpha}} (-u) \omega_u^n<\int_M (-u) \omega^n_u\wedge \eta<\infty.
\]
This implies that $u$ has Lelong number zero at $(0, 0)$. 
\end{proof}

The following result implies that to test membership in $\cE_\chi(M, \xi, \omega^T)$ it is enough to test the finiteness condition $E_\chi(u)<\infty$ on canonical cutoffs.

\begin{prop}\label{cte}Suppose $u\in \cE(M, \xi, \omega^T)$ with canonical cutoffs $\{u_k\}_{k\in \N}$. If $h: \R_+\rightarrow \R_+$ is continuous and increasing, then
\[
\int_M h(|u|)\omega_u^n\wedge \eta<\infty\Longleftrightarrow\limsup_{k\rightarrow \infty}\int_M h(|u_k|)\omega_{u_k}^n\wedge \eta<\infty.
\]
Moreover, if the above condition holds, then
\[
\int_M h(|u|)\omega_u^n\wedge \eta=\lim_{k\rightarrow \infty}\int_M h(|u_k|)\omega_{u_k}^n\wedge \eta
\]
\end{prop}
\begin{proof} Without loss of generality we can assume that $u \leq 0$.
If $\limsup_{k\rightarrow \infty}\int_M h(|u_k|)\omega_{u_k}^n\wedge \eta<\infty$, we obtain that
the sequence of Radon measures $h(|u_k|)\omega^n_{u_k} \wedge \eta$ is weakly compact. Hence there exists a subsequence $h(|u_{k_j}|)\omega_{u_{k_j}}^n \wedge \eta $ converging weakly to a Radon measure $\mu$. Recall that $h(|u_{k_j}|)$ is an increasing sequence of lower semicontinuous functions converging to $h(|u|)$ and $\omega_{u_{k_j}}^n \wedge \eta \xrightarrow{w} \omega_u^n \wedge \eta$, this yields that $h(|u|)\omega_u^n \wedge \eta \leq \mu$ as measure. In particular $\int_M \omega_u^n \wedge \eta \leq \mu(M)<\infty$.

Now assume $\int_M h(|u|)\omega_{u}^n \wedge \eta < \infty$. If $\lim\limits_{t \rightarrow +\infty}h(t)=+\infty$, we have
\begin{equation*}
\lim_{k \rightarrow \infty} \int_{\{u \leq -k\}} h(|u|)\omega_u^n \wedge \eta=\lim_{l \rightarrow +\infty} \int_{\{h(|u|) >l\}} h(|u|) \omega_u^n \wedge \eta=0
\end{equation*}
It follows from Proposition \ref{volume} and the Generalized Bedford-Taylor identity \ref{GBTI} that
\begin{equation*}
\int_{\{u \leq -k\}} \omega_{u_k}^n \wedge \eta =\int_{\{u \leq -k\}} \omega_u^n \wedge \eta
\end{equation*}
Then we have
\begin{equation*}
\begin{split}
|\int_M h(|u_k|)\omega_{u_k}^n \wedge \eta-\int_M h(|u|)\omega_u^n \wedge \eta| & \leq \int_{\{u \leq -k\}}h(k)\omega_{u_k}^n \wedge \eta+\int_{\{u \leq -k\}} h(|u|)\omega_u^n \wedge \eta     \\
 &=h(k) \int_{\{u \leq -k\}} \omega_u^n \wedge \eta +\int_{\{u \leq -k\}} h(|u|) \omega_u^n \wedge \eta   \\
 &\leq 2 \int_{\{u \leq -k\}} h(|u|) \omega_u^n \wedge \eta
 \end{split}
\end{equation*}
It follows that $\int_M h(|u_k|)\omega_{u_k}^n\wedge \eta$ is bounded and $\int_M h(|u|)\omega_u^n\wedge \eta=\lim_{k\rightarrow \infty}\int_M h(|u_k|)\omega_{u_k}^n\wedge \eta$.

If $\lim\limits_{t \rightarrow +\infty} h(t)=L<\infty$, it follows from Proposition \ref{volume} that $\int_M h(|u_k|)\omega_{u_k}^n\wedge \eta$ is bounded. Moreover for any $\epsilon>0$ there exists $N>0$ such that $0<L-h(t) <\epsilon$ for all $t>N$. Then for $k >N$ we have
\begin{equation*}
\begin{split}
|\int_M h(|u_k|)\omega_{u_k}^n \wedge \eta-\int_M h(|u|)\omega_u^n \wedge \eta| &=|\int_M (L-h(|u_k|))\omega_{u_k}^n \wedge \eta-\int_M (L-h(|u|))\omega_u^n \wedge \eta| \\
                                                                                                                                   &=|\int_{\{u \leq -k\}} (L-h(|u_k|))\omega_{u_k}^n \wedge \eta-\int_{\{u\leq-k\}}(L-h(|u|))\omega_u^n \wedge \eta|  \\
                                                                                                                                   &\leq 2\epsilon
                                                                                                                                   \end{split}
\end{equation*}
That is $\int_M h(|u|)\omega_u^n\wedge \eta=\lim_{k\rightarrow \infty}\int_M h(|u_k|)\omega_{u_k}^n\wedge \eta$.
\end{proof}
With the proposition above, we can then prove the so-called \emph{fundamental estimate}

\begin{prop}[Fundamental estimate]\label{fe}
Suppose $\chi\in \cW^+_p$ and $u, v\in \cE_\chi(M, \xi, \omega^T)$ such that $u\leq v\leq 0$. Then 
\begin{equation}\label{fe01}
E_\chi(v)\leq (p+1)^n E_\chi(u)
\end{equation}
\end{prop}
\begin{proof}
First of all we assume that $u, v \in \text{PSH}(M,\xi,\omega^T) \cap L^{\infty}$. For $0 \leq j \leq n-1$ we have
\begin{equation*}
\int_M \chi(u)\omega_v^{j+1} \wedge \omega_u^{n-j-1}\wedge \eta =\int_M \chi(u)\omega^T \wedge \omega_v^j \wedge \omega_u^{n-j-1} \wedge \eta+\int_M \chi(u) i\partial_B\overline{\partial}_B v \wedge \omega_v^j \wedge \omega_u^{n-j-1} \wedge \eta                                                                                                        
\end{equation*}
Recall that $\chi'(l) \leq 0$ for $l<0$. Using integration by parts we have
\begin{equation*}
\begin{split}
\int_M \chi(u)\omega^T \wedge \omega_v^j \wedge \omega_u^{n-j-1} \wedge \eta &=\int_M \chi(u) \wedge \omega_v^j \wedge \omega_u^{n-j} \wedge \eta-\int_M \sqrt{-1}\chi(u)\partial_B\overline{\partial}_B u \wedge \omega_v^j \wedge \omega_u^{n-j-1} \wedge \eta  \\
                                                                                                                                  &=\int_M \chi(u) \wedge \omega_v^j \wedge \omega_u^{n-j} \wedge \eta+\int_M \sqrt{-1} \chi'(u)\partial_B u \wedge \overline{\partial}_B u \wedge \omega_v^j \wedge \omega_u^{n-j-1} \wedge \eta\\
                                                                                                                                  &\leq  \int_M \chi(u) \wedge \omega_v^j \wedge \omega_u^{n-j} \wedge \eta
 \end{split}
\end{equation*}
Recall that $\chi'(l) \leq 0$ for $l<0$ and $l\chi'(l) \leq p \chi(l)$ for $l \geq 0$. Using the integration by parts repeatedly we have
\begin{equation*}
\begin{split}
&\int_M \chi(u) i\partial_B\overline{\partial}_B v \wedge \omega_v^j \wedge \omega_u^{n-j-1} \wedge \eta    \\
&=\int_M \sqrt{-1} v\chi^{''}(u) \partial_B u \wedge \overline{\partial}_B u \wedge \omega_v^j \wedge \omega_u^{n-j-1} \wedge \eta +\int_M \sqrt{-1}v\chi'(u)\partial_B \overline{\partial}_B u \wedge \omega_v^j \wedge \omega_u^{n-j-1} \wedge \eta \\
&\leq \int_M \sqrt{-1}v\chi'(u)\partial_B \overline{\partial}_B u \wedge \omega_v^j \wedge \omega_u^{n-j-1} \wedge \eta  \\
&\leq \int_M v\chi'(u) \omega_v^j \wedge \omega_u^{n-j} \wedge \eta =\int_M |v|\chi'(|u|) \omega_v^j\wedge \omega_u^{n-j} \wedge \eta \\
&\leq \int_M  |u| \chi'(|u|) \omega_v^j \wedge \omega_u^{n-j} \wedge \eta \leq p \int_M \chi(|u|) \omega_v^j \wedge \omega_u^{n-j} \wedge \eta
\end{split}
\end{equation*}
Combine the inequalities above we obtain
\begin{equation*}
\int_M \chi(u)\omega_v^{j+1} \wedge \omega_u^{n-j-1}\wedge \eta \leq (p+1) \int_M \chi(u) \omega_v^j \wedge \omega_u^{n-j} \wedge \eta
\end{equation*}
It follows that
\begin{equation*}
E_{\chi}(v) \leq \int_M \chi(u) \omega_v^n \wedge \eta \leq (p+1)^n E_{\chi}(u)
\end{equation*}
In the general case $u, v \in \cE_{\chi}(M,\xi,\omega^T)$, we have $E_{\chi}(v_k) \leq E_{\chi}(u_k)$ for the canonical cutoffs $u_k, v_k$. It follows from  Proposition \ref{cte} that $E_{\chi}(v) \leq (p+1)^n E_{\chi}(u)$.
\end{proof}

As a direct consequence, we obtain the \emph{monotonicity property} for $\cE_\chi(M, \xi, \omega^T)$

\begin{prop}\label{oder}Suppose $u\in \cE_\chi(M, \xi, \omega^T)$ and $v\in \text{PSH}(M, \xi, \omega^T)$. If $u\leq v$, then $v\in \cE_\chi(M, \xi, \omega^T)$
\end{prop}

\begin{proof}
Without loss of generality we can assume that $u \leq v \leq 0$.The monotonicity property implies that $v \in \cE(M,\xi,\omega^T)$. We have $u \leq v_k$ for the canonical cutoffs of $v$,then $E_{\chi}(v_k) \leq (p+1)^nE_{\chi}(u)$ according to the Proposition \ref{fe}. It follows from Proposition \ref{cte} that $E_{\chi}(v) (p+1)^n\leq E_{\chi}(u)$ and $v \in \cE_{\chi}(M,\xi,\eta)$.
\end{proof}

We also have the following,

\begin{prop}\label{mixedenergy}Suppose $u, v\in \cE_\chi(M, \xi, \omega^T)$ for $\chi\in \cW^+_p$. If $u, v\leq 0$,  then
\[
\int_M \chi(u)\omega_v^n\wedge \eta\leq p 2^p(E_\chi(u)+E_\chi(v))
\]  
\end{prop}

\begin{proof}
Take $\tilde{\chi}(t)=\chi(t)+\delta |t| \in \cW^+_p$. Assume that $t>0$, It is obvious $\tilde{\chi}(t),\tilde{\chi}'(t)>0$. Recall that  $\epsilon^p\tilde{\chi}(t) \leq \tilde{\chi}(\epsilon t)$  and $t\tilde{\chi}'(t) \leq p\tilde{\chi}(t)$ for $\tilde{\chi} \in \cW_p^+$ and  $0 < \epsilon <1$, hence we have $\tilde{\chi}(2t) \leq 2^p \tilde{\chi}(t)$. It follows from the convexity of the function $\tilde{\chi}(t)$ that $\frac{\tilde{\chi}(t)}{t} \leq \tilde{\chi}'(t)$. Then
\begin{equation*}
\tilde{\chi}'(2t)= \frac{1}{2}\frac{2t \tilde{\chi}'(2t)}{\tilde{\chi}(2t)} \frac{\tilde{\chi}(2t)}{\tilde{\chi}(t)} \frac{\tilde{\chi}(t)}{t} \leq p2^{p-1} \tilde{\chi}'(t)
\end{equation*}
Then $\delta \rightarrow 0$ implies that $\chi'(2t) \leq p 2^{p-1}\chi'(t)$ for $t>0$.

Consider the generalized Bedford-Taylor identity and $\{ |u|>2t\} \subset \{ u <v-t\} \cup \{ v < -t\}$, we have
\begin{equation*}
\begin{split}
\int_M \chi(u) \omega_v^n \wedge \eta &=\int_0^{\infty} \chi'(t) \omega_v^n \wedge \eta\{|u|>t\}dt \\
                                                              &\leq p2^p \int_0^{\infty} \chi'(t) \omega_v^n \wedge \eta\{|u|>2t\}dt  \\
                                                              &\leq p2^p(\int_0^{\infty} \chi'(t) \omega_v^n \wedge \eta \{u<v-t\}dt +\int_0^{\infty} \chi'(t) \omega_v^n \wedge \eta \{v<-t\} dt) \\
                                                              &\leq p2^p(\int_0^{\infty} \chi'(t) \omega_u^n \wedge \eta \{u<v-t\}dt +E_{\chi}(v)) \\
                                                              & \leq p2^p(\int_0^{\infty} \chi'(t) \omega_u^n \wedge \eta \{u<-t\}dt +E_{\chi}(v))  \\
                                                              &=p2^p(E_{\chi}(u)+E_{\chi}(v))
 \end{split}
\end{equation*}
\end{proof}

\begin{prop}\label{weightchange}Suppose $u\in \cE_\chi(M, \xi, \omega^T), \chi\in \cW^+_p$. Then there exists $\tilde \chi\in \cW^{+}_{2p+1}$ such that $\chi(t)\leq \tilde \chi(t), \chi(t)/\tilde \chi(t)\rightarrow 0$ as $t\rightarrow \infty$ and $u\in \cE_{\tilde \chi}(M, \xi, \omega^T)$
\end{prop}
\begin{proof} 
Take $\chi_0=\chi$, recall that $\lim\limits_{t \rightarrow \infty} \chi_0(t)=\infty$ and $u \in \cE_{\chi}(M,\xi,\omega^T)$, we have
\begin{equation*}
\lim_{t \rightarrow \infty} \int_{\{|u| > t\}} \chi(|u|)\omega_u^n\wedge\eta =\lim_{s \rightarrow \infty} \int_{\{\chi(u) > s\}} \chi(|u|) \omega_u^n\wedge \eta=0
\end{equation*}
Then one can choose $t_1>0$ such that $\int_{|u|>t_1} \chi(|u|)\omega_u^n\wedge\eta<\frac{1}{2^2}$. We define $\chi_1:\mathbb{R}^+ \rightarrow \mathbb{R}^+$ by the formula:
\begin{equation*}
\chi_1(t)=
\begin{cases}
\chi_0(t)   & \text{if} \quad t\leq t_1  \\
\chi_0(t_1)+2(\chi_0(t)-\chi_0(t_1)) &\text{if} \quad t>t_1.
\end{cases}
\end{equation*}
Then it is easy to check that 
\begin{enumerate}
\item $\chi_0(t) \leq \chi_1(t)$;
\item $\lim\limits_{t \rightarrow \infty} \frac{\chi_1(t)}{\chi_0(t)}=2$;
\item $E_{\chi_1}(u) \leq E_{\chi_0}(u)+\frac{1}{2}$;
\item $\sup\limits_{t>0} \frac{|t\chi'_1(t)|}{|\chi_1(t)|} \leq \sup\limits_{t>0}\frac{2|t\chi'_0(t)|}{|\chi_0(t)|}<2p+1$;
\item $\lim\limits_{t \rightarrow \infty} \frac{t\chi'_1(t)}{\chi_1(t)} \leq p$
\end{enumerate}
These properties imply that for $t_2>t_1$ big enough, the function $\chi_2:\mathbb{R}^+ \rightarrow \mathbb{R}^+$
\begin{equation*}
\chi_2(t)=
\begin{cases}
\chi_1(t)   & \text{if} \quad t\leq t_2  \\
\chi_1(t_2)+2(\chi_1(t)-\chi_1(t_2)) &\text{if} \quad t>t_2.
\end{cases}
\end{equation*}
satisfies
\begin{enumerate}
\item $\chi_1(t) \leq \chi_1(t)$;
\item $\lim\limits_{t\rightarrow \infty} \frac{\chi_2(t)}{\chi_1(t)}=2$;
\item $E_{\chi_2}(u) \leq E_{\chi_1}(u)+\frac{1}{2^2}$;
\item $\sup\limits_{t>0} \frac{|t\chi'_2(t)|}{|\chi_2(t)|} <2p+1$;
\item $\lim\limits_{t \rightarrow \infty} \frac{t\chi'_2(t)}{\chi_2(t)} \leq p$
\end{enumerate}
Continuing the above construction we can obtain an increasing sequence $\{\chi_k\}_k$ and the limit weight $\tilde{\chi}(t)=\lim\limits_{k \rightarrow \infty} \chi_k(t)$ will satisfy the requirements of the Proposition.
\end{proof}

\begin{prop}\label{weakcon3}
Assume that $\{\psi_k\}_{k\in \N}, \{\phi_k\}_{k\in \N}, \{v_k\}_{k\in \N}\subset \cE_\chi(M, \xi, \omega^T)$ decrease (increase a. e) to $\phi, \psi, v\in \cE_\chi(M, \xi, \omega^T)$ respectively. Suppose 
\begin{enumerate}
\item $\psi_k\leq \phi_k$ and $\psi_k\leq v_k$.
\item $h: \R\rightarrow \R$ is continuous with $\limsup_{|l|\rightarrow \infty}|h(l)|/\chi(l)\leq C$ for some $C\geq 0$.
\end{enumerate}
Then we have the weak convergence of
\[
h(\phi_k-\psi_k)\omega_{v_k}^n\wedge \eta\rightarrow h(\phi-\psi)\omega^n_v\wedge \eta.
\]
\end{prop}
\begin{proof}Without loss of generality one can assume all the functions $\phi_k,\phi,\psi_k,\psi, v, v_k$ are negative. We will only prove the Proposition for decreasing sequences, the case of increasing sequences can be proved similarly.

First of all we suppose that the functions involved are uniformly bounded below, that is, there exists $L>1$ such that $-L \leq \phi_k,\phi,\psi_k,\psi, v_k, v \leq 0$. Given $\epsilon>0$, it follows from Theorem \ref{quasicontinuity} that there exists an open subset $O_{\epsilon} \subset M$ such that $\text{cap}(O_{\epsilon})<\epsilon$ and $\phi_k,\phi,\psi_k,\psi, v_k, v$ are continuous on $M-O_{\epsilon}$. Then $\phi_k \rightarrow \phi$ and $\psi_k \rightarrow \psi$ uniformly on $M-O_{\epsilon}$. Hence there exists $N$ such that for $k>N$ we have $|h(\phi_k-\psi_k)-h(\phi-\psi)|<\epsilon$ on $M-O_{\epsilon}$ and the term
\begin{equation*}
\int_M h(\phi_k-\psi_k)\omega_{v_k}^n \wedge \eta-\int_M h(\phi-\psi)\omega_{v_k}^n \wedge \eta=(\int_{O_{\epsilon}}+\int_{M-O_{\epsilon}}) [h(\phi_k-\psi_k)-h(\phi-\psi)] \omega_{v_k}^n \wedge \eta
\end{equation*}
is bounded by $2\epsilon L^n \max\limits_{0 \leq l \leq L}|h(l)| +\epsilon$. Hence
\begin{equation}\label{weakcon1(1)}
\int_M h(\phi_k-\psi_k) \omega_{v_k}^n \wedge \eta-\int_M h(\phi-\psi)\omega_{v_k}^n \wedge \eta \rightarrow 0
\end{equation}

Given $\epsilon>0$, it follows from Theorem \ref{quasicontinuity} that there exists an open subset $\tilde{O}_{\epsilon}$ such that $\text{cap}(\tilde{O}_{\epsilon}) <\epsilon$ and $\phi,\psi$ are continuous on $M-\tilde{O}_{\epsilon}$. By the Tietze's extension theorem the function $h(\phi-\psi)|_{M-\tilde{O}_{\epsilon}}$ can be extended  to a continuous function $\alpha$ on $M$ bounded by $\max\limits_{0\leq l\leq L}|h(l)|$. By Proposition \ref{weakcon0} we have $\omega_{v_k}^n\wedge\eta \rightarrow \omega_v^n\wedge\eta$ weakly. Then there exists a constant $N$ such that for $k>N$ we have $|\int_M \alpha\omega_{v_k}^n\wedge\eta-\int_M\alpha\omega_v^n\wedge\eta|<\epsilon$ and the term
\begin{equation*}
\begin{split}
&\quad \int_M h(\phi-\psi)\omega_{v_k}^n\wedge\eta-\int_M h(\phi-\psi)\omega_v^n\wedge\eta \\
&= \int_{O_{\epsilon}} (h(\phi-\psi)-\alpha)\omega_{v_k}^n\wedge\eta-\int_{O_{\epsilon}}(h(\phi-\psi)-\alpha)\omega_v^n\wedge\eta+(\int_M \alpha \omega_{v_k}^n\wedge\eta-\int_M\alpha\omega_v^n\wedge\eta)
\end{split}
\end{equation*}
is bounded by $4\epsilon L^n \max\limits_{0\leq l\leq L}|h(l)|+\epsilon$. Hence
\begin{equation}\label{weakcon1(2)}
\int_M h(\phi-\psi)\omega_{v_k}^n\wedge \eta-\int_Mh(\phi-\psi)\omega_v^n \wedge \eta \rightarrow 0
\end{equation}
It follows from \ref{weakcon1(1)} and \ref{weakcon1(2)} that $h(\phi_k-\psi_k)\omega_{v_k}^n\wedge\eta \rightarrow h(\phi-\psi)\omega_v^n\wedge\eta$.

Now consider the general case when $\phi_k,\phi,\psi_k,\psi, v_k, v$ are unbounded. Let $\phi_k^l,\phi^l,\psi_k^l,\psi^l, v_k^l, v^l$ be the canonical cutoffs of the corresponding potentials, then we only have to show that
\begin{equation}\label{weakcon1(3)}
\int_M h(\phi_k-\psi_k)\omega_{v_k}^n\wedge\eta-\int_Mh(\phi_k^l-\psi_k^l)\omega_{v_k^l}^n\wedge\eta \rightarrow 0
\end{equation}
and
\begin{equation}\label{weakcon1(4)}
\int_Mh(\phi-\psi)\omega_v^n\wedge\eta-\int_Mh(\phi^l-\psi^l)\omega_{v^l}^n\wedge\eta \rightarrow 0
\end{equation}
as $l\rightarrow \infty$ uniformly with respect to $k$.

By Proposition \ref{weightchange} there exists $\tilde{\chi}  \in \cW_{2p+1}^+$ such that $\chi \leq \tilde{\chi},\lim\limits_{t \rightarrow \infty} \frac{\chi(t)}{\tilde{\chi}(t)}=0$ and $\psi \in \cE_{\tilde{\chi}}(M,\xi,\omega^T)$. Then $\psi_k,\phi_k,\phi, v_k, v \in \cE_{\tilde{\chi}}(M,\xi,\omega^T)$ according to Proposition \ref{oder}.

Recall that there exists $L>0$ such that $\chi(L) \geq 1$ and $|h(t)| \leq (C+1)\chi(t)$ for $|t|>L$. Take $\tilde{C}= \max\{C+1,\frac{\max\limits_{0 \leq l \leq L}|h(l)|}{\chi(L)}\}$, then we have
\begin{equation*}
|h(l_1-l_2)| \leq \tilde{C} \chi(l_2)
\end{equation*}
for $l_2 \leq -L$ and $l_2  \leq l_1 \leq 0$.
Using the Generalized Bedford-Taylor identity, the fundamental estimate and Proposition \ref{mixedenergy} we have
\begin{equation*}
\begin{split}
    &\quad |\int_M h(\phi_k-\psi_k)\omega_{v_k}^n\wedge\eta-\int_Mh(\phi_k^l-\psi_k^l)\omega_{v_k^l}^n\wedge\eta|  \\
&=|\int_{\{\psi_k \leq -l\}} h(\phi_k-\psi_k)\omega_{v_k}^n\wedge\eta-\int_{\{\psi_k \leq -l\}}h(\phi_k^l-\psi_k^l)\omega_{v_k^l}^n\wedge\eta| \\
&\leq \int_{\{\psi_k \leq -l\}} |h(\phi_k-\psi_k)|\omega_{v_k}^n\wedge\eta+\int_{\{\psi_k \leq -l\}}|h(\phi_k^l-\psi_k^l)|\omega_{v_k^l}^n\wedge\eta \\
 &\leq \tilde{C}(\int_{\{\psi_k \leq -l\}} \chi(\psi_k)\omega_{v_k}^n\wedge\eta+\int_{\{\psi_k \leq -l\}} \chi(\psi_k^l)\omega_{v_k^l}^n\wedge\eta)  \\
 &\leq \tilde{C}\sup_{s\leq-l}\frac{\chi(s)}{\tilde{\chi}(s)} (\int_{\{\psi_k \leq -l\}} \tilde{\chi}(\psi_k)\omega_{v_k}^n\wedge\eta+\int_{\{\psi_k \leq -l\}} \tilde{\chi}(\psi_k^l)\omega_{v_k^l}^n\wedge\eta)  \\
 &\leq \tilde{C}\sup_{s \leq -l}\frac{\chi(s)}{\tilde{\chi}(s)} (\int_M \tilde{\chi}(\psi_k)\omega_{v_k}^n\wedge\eta+\int_M \tilde{\chi}(\psi_k^l)\omega_{v_k^l}^n\wedge\eta)  \\ 
 &\leq (2p+1)2^{2p+1}\tilde{C}\sup_{s\leq-l}\frac{\chi(s)}{\tilde{\chi}(s)}(E_{\tilde{\chi}}(\psi_k)+E_{\tilde{\chi}}(v_k)+E_{\tilde{\chi}}(\psi_k^l)+E_{\tilde{\chi}}(v_k^l))  \\
 &\leq 4(2p+1)(2p+2)^n2^{2p+1}\tilde{C}E_{\tilde{\chi}}(\psi)\sup_{s\leq-l}\frac{\chi(s)}{\tilde{\chi}(s)}
  \end{split}
\end{equation*}
for $l>L$ and the statement \ref{weakcon1(3)} follows.
We also have
\begin{equation*}
\begin{split}
&\quad |\int_Mh(\phi-\psi)\omega_v^n\wedge\eta-\int_Mh(\phi^l-\psi^l)\omega_{v^l}^n\wedge\eta| \\
&=|\int_{\{\psi\leq-l\}}h(\phi-\psi)\omega_v^n\wedge\eta-\int_{\{\psi\leq-l\}}h(\phi^l-\psi^l)\omega_{v^l}^n\wedge\eta| \\
&\leq \int_{\{\psi\leq-l\}}|h(\phi-\psi)|\omega_v^n\wedge\eta+\int_{\{\psi\leq-l\}}|h(\phi^l-\psi^l)|\omega_{v^l}^n\wedge\eta \\
&\leq \tilde{C} (\int_{\{\psi\leq-l\}}\chi(\psi)\omega_v^n\wedge\eta+\int_{\{\psi\leq-l\}}\chi(\psi^l)\omega_{v^l}^n\wedge\eta) \\
&\leq \tilde{C} \sup_{s\leq-l} \frac{\chi(s)}{\tilde{\chi}(s)}(\int_{\{\psi\leq-l\}}\tilde{\chi}(\psi)\omega_v^n\wedge\eta+\int_{\{\psi\leq-l\}}\tilde{\chi}(\psi^l)\omega_{v^l}^n\wedge\eta) \\
&\leq \tilde{C} \sup_{s\leq-l} \frac{\chi(s)}{\tilde{\chi}(s)}(\int_M\tilde{\chi}(\psi)\omega_v^n\wedge\eta+\int_M\tilde{\chi}(\psi^l)\omega_{v^l}^n\wedge\eta) \\
&\leq (2p+1)2^{2p+1}\tilde{C}\sup_{s\leq-l}\frac{\chi(s)}{\tilde{\chi}(s)}(E_{\tilde{\chi}}(\psi)+E_{\tilde{\chi}}(v)+E_{\tilde{\chi}}(\psi^l)+E_{\tilde{\chi}}(v^l))  \\
&\leq  4(2p+1)(2p+2)^n2^{2p+1}\tilde{C}E_{\tilde{\chi}}(\psi)\sup_{s\leq-l}\frac{\chi(s)}{\tilde{\chi}(s)}
\end{split}
\end{equation*}
for $l>L$ and the statement \ref{weakcon1(4)} follows. This completes the proof.
\end{proof}

\begin{prop}\label{boundedenergy}Suppose $\chi\in \cW^+_p$ and $\{u_k\}_{k\in \N}\subset \cE_\chi(M, \xi, \omega^T)$ is a decreasing sequence converging to $u\in \text{PSH}(M, \xi, \omega^T)$. If $\sup_k E_\chi(u_k)<\infty$ then $u\in \cE_\chi(M, \xi, \omega^T)$ and 
\[
E_\chi(u)=\lim_{k\rightarrow \infty} E_\chi(u_k).
\]
\end{prop}

\begin{proof}
Without loss of generality we assume that $u_1\leq0$. The canonical cutoffs $u_k^l=\max\{u_k,-l\}$ decreases to the canonical cutoff $u^l=\max\{u,-l\}$. As $-l \leq u^l \leq u_k^l \leq 0$,  Proposition \ref{weakcon3} and the fundamental estimate imply that 
\begin{equation*}
E_{\chi}(u^l) =\lim_{k \rightarrow \infty} E_{\chi}(u_k^l) \leq (p+1)^n \sup_k E_{\chi}(u_k)
\end{equation*}
By Proposition \ref{cte}, $u \in \cE_{\chi}(M,\xi,\omega^T)$. Applying the previous Proposition in the case $\psi_k=v_k=u_k,\phi_k=0$ gives that $E_{\chi}(u)=\lim\limits_{k \rightarrow \infty} E_{\chi}(u_k)$.
\end{proof}

A very important notion in pluripotential theory is the \emph{envelop construction}, which we shall describe below. In our setting on a compact Sasaki manifold, given a usc function $f\in M\rightarrow [-\infty, \infty)$ such that $f$ is invariant under the Reeb flow, we consider the envelop
\begin{equation}\label{envelop01}
P(f):=\sup\{u\in \text{PSH}(M, \xi, \omega^T)\; \text{such that}\; u\leq f\}.
\end{equation}
As in K\"ahler setting, we have the following
\begin{prop}The envelop construction $P(f)\in \text{PSH}(M, \xi, \omega^T)$. 
\end{prop}
\begin{proof}This statement is local in nature, hence we only need to argue in foliations charts $W_\alpha=(-\delta, \delta)\times V_\alpha$, where $V_\alpha\subset \C^n$ give a transverse holomorphic charts. Since $P(f)$ is invariant under the Reeb flow, its usc regularization $P(f)^{*}$ is invariant under the Reeb flow. Hence by $P(f)^*$ is $\omega^T_\alpha$-psh on each $V_\alpha$, see \cite{BN}[Theorem 1.2.3 (viii)]. Since $f$ is usc, hence $P(f)^*\leq f^*=f$. Hence $P(f)^*$ is a candidate in the definition of $P(f)$, gives that $P(f)^*\leq P(f)$. This implies that $P(f)=P(f)^*$ and $P(f)\in \text{PSH}(M, \xi, \omega^T)$. 
\end{proof}

We also introduce the notion \emph{rooftop envelop}, for usc functions $\{f_1, \cdots, f_n\}$ which are invariant under the Reeb flow,
\[
P(f_1, \cdots, f_n):=P(\min\{f_1, \cdots, f_n\}). 
\]
We have the following, 
\begin{thm}\label{rooftop101}Given $f\in C^\infty_B$, then 
we have the following estimate
\[
\|P(f)\|_{C^{1, \bar 1}}\leq C(M, \omega^T, g, \|f\|_{C^{1, \bar 1}}).
\]
Moreover, if 
 $u_1, \cdots, u_k\in \cH_\Delta$, where we use the notation
\[
\cH_\Delta=\{u\in \text{PSH}(M, \xi, \omega^T): \|u\|_{C^{1, \bar 1}}<\infty\}
\]
then $P(u_1, \cdots, u_k)\in \cH_\Delta$.
\end{thm}

We shall prove Theorem \ref{rooftop101} in Appendix. 
The following result would be very essential for the rooftop envelop $P(u_0, u_1)$: that is, on the non-contact set $\Gamma:=\{P(u_0, u_1)<\min (u_0, u_1)\}, \omega_{P(u_0, u_1)}^n\wedge \eta=0$.  

\begin{lemma}For $u_0, u_1\in \cH_\Delta$, then on $\Gamma$, 
\begin{equation}\omega_{P(u_0, u_1)}^n\wedge \eta=0
\end{equation}
\end{lemma}
\begin{proof}
First we assume $\xi$ is regular or quasiregular, then the proof follows similarly as in K\"ahler setting. We sketch the proof briefly. We consider the quotient K\"ahler manifold (orbifold) $(Z=M/\cF_\xi, \omega_Z)$ such that $\omega^T=\pi^*\omega_Z$, where $\pi: M\rightarrow Z$ is the natural quotient map. Since $u_0, u_1$ and $P(u_0, u_1)$ are all basic functions, and they descend to $Z$ to define the functions on $Z$, which we still denote as $u_0, u_1$ and $P(u_0, u_1)$. We only need to show that $(\omega_Z+\sqrt{-1}\p\bar \p P(u_0, u_1))^n=0$ on $\Gamma_Z:=\{z\in Z: P(u_0, u_1)<\min (u_0, u_1)\}$. Note that $\Gamma_Z=\pi(\Gamma)$. This simply follows from \cite{BT1}[Corollary 9.2]. 

Now we deal with the case when $\xi$ is irregular. We need to use a Type-I deformation to approximate $(M, \xi, \eta, g, \Phi)$, as in Theorem \ref{type101}. Denote $T^k$ to be the torus in $\text{Aut}(\xi, \eta, g)$ with the Lie algebra $\mathfrak{t}$. Take $\rho_i\in \mathfrak{t}$ such that $\rho_i\rightarrow 0$ (convergence is smooth with respect to a fixed metric $g$).  We can take $\rho_i$ such that $\xi_i=\xi+\rho_i$ is quasiregular.  Consider the Type-I deformation $(M, \xi_i, \eta_i, g_i, \Phi_i)$ as in Definition \ref{type-01}. Given $u_0, u_1\in \cH_\Delta$ and we know that $P(u_0, u_1)\in \cH_\Delta$ (see Theorem \ref{rooftop101}), by Lemme \ref{type1}, there exists $\epsilon_i\rightarrow 0$ such that $(1-\epsilon_i) u_0, (1-\epsilon_i)u_1, (1-\epsilon_i)P(u_0, u_1)\in \text{PSH}(M, \xi_i, \omega^T_i)$.
Define
\begin{equation}\label{rooftop102}
P_i=P_i((1-\epsilon_i)u_0, (1-\epsilon_i) u_1)=\sup\{v\in \text{PSH}(M, \xi_i, \omega^T_i), v\leq (1-\epsilon_i)u_0, (1-\epsilon_i) u_1\}. 
\end{equation}
Since $(1-\epsilon_i)P(u_0, u_1)\in \text{PSH}(M, \xi_i, \omega^T_i)$ and $(1-\epsilon_i)P(u_0, u_1)\leq (1-\epsilon_i)u_0, (1-\epsilon_i) u_1$, hence $(1-\epsilon_i)P(u_0, u_1)\leq P_i$. On the other hand, we apply Lemma \ref{type1} and we know there exists $\varepsilon_i\rightarrow 0$, such that $(1-\varepsilon_i)P_i\in \text{PSH}(M, \xi, \omega^T)$. It follows that
\[
(1-\varepsilon_i) P_i\leq P(u_0, u_1)\leq P_i (1-\epsilon_i)^{-1}
\]
By Theorem \ref{rooftop101}, we know that $|d\Phi d P_i|$ is uniformly bounded and hence $P_i\rightarrow P(u_0, u_1)$ in $C^{1, \alpha}$. For any compact subset  $ K\subset \Gamma=\{P(u_0, u_1)<\min (u_0, u_1)\}$, we can choose $i$ sufficiently large, such that $P_i<\min\{(1-\epsilon_i)u_0, (1-\epsilon_i) u_1\}$. Since $\xi_i$ is quasiregular, by \eqref{rooftop102}, we can then get that 
\[(\omega_i^T+\frac{1}{2}d\Phi_i d P_i)^n \wedge \eta_i=0, \;\text{on}\; K.\]
Taking $i\rightarrow \infty$, by Lemma \ref{measure100}, we get that
\[(\omega^T+\frac{1}{2}d\Phi d P(u_0, u_1))^n \wedge \eta=0, \;\text{on}\; K.\]
This completes the proof. 
\end{proof}

As a consequence, we get a volume partition formula for $\omega^n_{P(u_0, u_1)}\wedge\eta$ as follows,
\begin{lemma}\label{decomposition}
For $u_0, u_1\in \cH_\Delta$, denote $\Lambda_{u_0}=\{P(u_0, u_1)=u_0\}$ and $\Lambda_{u_1}=\{P(u_0, u_1)=u_1\}$. Then we have the following
\begin{equation}
\omega^n_{P(u_0, u_1)}\wedge \eta=\chi_{\Lambda_{u_0}} \omega^n_{u_0}\wedge \eta+\chi_{\Lambda_{u_1}\backslash \Lambda_{u_0}} \omega^n_{u_1}\wedge \eta. 
\end{equation}
\end{lemma}

\begin{proof}
The previous Lemma implies that the measure $\omega^n_{P(u_0,u_1)}\wedge \eta$ is supported on the set $\Lambda_{u_0} \cup \Lambda_{u_1}$. It follows from Theorem \ref{rooftop101} that $P(u_0,u_1)$ has bounded Laplacian, hence all second partial derivatives of $P(u_0,u_1)$ are in $L^p(M)$ for all $p>1$. Then all the second order partial derivatives of $P(u_0,u_1)$ and $u_0$ coincide on $\Lambda_{u_0}$ a.e., all the second order partial derivatives of $P(u_0,u_1)$ and $u_1$ coincide on $\Lambda_{u_1}$ a.e..Recall the definition of Monge-Ampere operators on functions belong to $W^{2,n}$,we can write:
\begin{equation*}
\omega^n_{P(u_0, u_1)}\wedge \eta=\chi_{\Lambda_{u_0}} \omega^n_{u_0}\wedge \eta+\chi_{\Lambda_{u_1}\backslash \Lambda_{u_0}} \omega^n_{u_1}\wedge \eta. 
\end{equation*}
\end{proof}

\begin{lemma}\label{}
Suppose $\chi\in \cW^+_p$ and $u_0, u_1\in \cE_\chi(M, \xi, \omega^T)$. Then $P(u_0, u_1)\in \cE_\chi(M, \xi, \omega^T)$. If $u_0, u_1\leq 0$, then the following estimates hold
\begin{equation}
E_\chi(P(u_0, u_1))\leq (p+1)^n(E_\chi(u_0)+E_\chi(u_1)). 
\end{equation}
\end{lemma}

\begin{proof}
Without loss of generality we can assume $u_0,u_1<0$. It follows from Lemma \ref{BK} that there exist negative transverse K\"ahler potentials $u_0^k, u_1^k \in \mathcal{H}$ deceasing to $u_0,u_1$ respectively. By Theorem \ref{rooftop101}, the rooftop envelopes $P(u_0^k, u_1^k) \in \mathcal{H}_{\triangle}$ decreases to $P(u_0,u_1)$. And we have the following inequality by Lemma \ref{decomposition}:
\begin{equation*}
\omega^n_{P(u_0^k,u_1^k)}\wedge \eta \leq \chi_{\Lambda_{u_0}} \omega^n_{u_0} \wedge \eta +\chi_{\Lambda_{u_1}} \omega^n_{u_1} \wedge \eta
\end{equation*}
Then
\begin{equation*}
\begin{split}
E_{\chi}(P(u^k_0,u^k_1)) &=\int_M \chi(P(u^k_0,u^k_1)) \omega^n_{P(u^k_0,u^k_1)} \wedge \eta  \\
                                 &\leq \int_{P(u^k_0,u^k_1)=u^k_0} \chi(u_0^k)\omega^n_{u_0^k} \wedge \eta+\int_{P(u_0^k,u_1^k)=u_1^k} \chi(u_1^k) \omega^n_{u_1^k} \wedge \eta \\
                                 &\leq E_{\chi}(u_0^k)+E_{\chi}(u_1^k) \\
                                 &\leq (p+1)^n(E_{\chi}(u_0)+E_{\chi}(u_1))
\end{split}
\end{equation*}
By Proposition \ref{boundedenergy} we have $P((u_0,u_1)) \in \cE_{\chi}(M,\xi,\omega^T)$ and the required inequality holds.\end{proof}

As a corollary we know that $\cE_\chi(M, \xi, \omega^T)$ is convex,
\begin{cor}\label{convex1}
If $u_0, u_1\in \cE_\chi(M, \xi, \omega^T)$, then $tu_0+(1-t)u_1\in  \cE_\chi(M, \xi, \omega^T)$ for any $t\in [0, 1]$.
\end{cor}
\begin{proof}
By the previous Lemma we have $P(u_0,u_1) \in \cE_\chi(M, \xi, \omega^T)$. Notice that $P(u_0,u_1) \leq tu_0+(1-t)u_1$ for $t \in [0,1]$, then the monotonicity property of $\cE_\chi(M, \xi, \omega^T)$ implies that $tu_0+(1-t)u_1 \in \cE_\chi(M, \xi, \omega^T)$.\end{proof}

\begin{lemma}
Let $U \subset M$ be a Borel set with $(\omega^T)^n\wedge\eta(U)>0$. and $u \in \cE_1(M, \xi, \omega^T)$. There there exists $\varphi \in \cE_1(M,\xi, \omega^T)$ with $\varphi \leq u$ and $\omega_{\varphi}^n\wedge\eta(U) >0$.
\end{lemma}
\begin{proof}
Without loss of generality we can assume that $u<0$. Then we can choose a sequence $u_k \in \cH$ decreasing to $u$ with $u_k <0$. For a constant $\tau>0$, we have $\{P(u_k+\tau, 0)=u_k+\tau\} \subset \{ u_k \leq -\tau\}$. It follows from Proposition \ref{decomposition} that
\[
\omega_{P(u_k+\tau,0)}^n\wedge\eta \leq \chi_{\{u_k \leq -\tau\}} \omega_{u_k}^n\wedge\eta+ (\omega^T)^n\wedge\eta \leq -\frac{u_k}{\tau}\omega_{u_k}^n\wedge\eta+(\omega^T)^n\wedge\eta
\]
The sequence $P(u_k+\tau, 0) \in \cE_1(M,\xi,\omega^T)$ decreases to $P(u+\tau, 0) \in \cE_1(M, \xi, \omega^T)$. It follows from  Proposition \ref{weakcon3} that
\[
\omega_{P(u+\tau,0)}^n\wedge\eta \leq -\frac{u}{\tau}\omega_u^n\wedge\eta+(\omega^T)^n\wedge\eta
\]
Hence we have
\[
\omega_{P(u+\tau,0)}^n\wedge\eta(M-U) \leq \frac{1}{\tau} \int_{M-U}|u|\omega_u^n\wedge\eta+(\omega^T)^n\wedge\eta(M-U) \leq \frac{1}{\tau}\int_M|u|\omega_u^n\wedge\eta+(\omega^T)^n\wedge\eta(M-U)
\]
It follows from $\omega_{P(u+\tau,0)}^n\wedge\eta(M)=(\omega^T)^n\wedge\eta(M)$ that
\[
\omega_{P(u+\tau,0)}^n\wedge\eta(U) \geq (\omega^T)^n\wedge\eta(U)- \frac{1}{\tau}\int_M|u|\omega_u^n\wedge\eta
\]
and $\omega_{P(u+\tau,0)}^n\wedge\eta(U)>0$ for $\tau$ big enough. Then $\varphi=P(u+\tau, 0) -\tau$ is the potential required.
\end{proof}

\begin{lemma}\label{domination2}(The domination principle)
If $u, v\in \cE_1(M,\xi,\omega^T)$ and $u \leq v$ almost everywhere with respect to the measure $\omega_v^n\wedge\eta$. Then $u\leq v$.
\end{lemma}
\begin{proof} 
We only have to prove $u \leq v$ almost everywhere with respect to $(\omega^T)^n\wedge\eta$ for $u, v<0$.

Suppose that $(\omega^T)^n\wedge\eta(\{u>v\}) >0$. The previous Lemma implies that there exists  $\varphi \in \cE_1(M, \xi, \omega^T)$ with $\varphi \leq u$ and $\omega_{\varphi}^n\wedge\eta(\{u>v\}) >0$. It follows from Corollary \ref{convex1} that $t \varphi +(1-t) u \in \cE_1(M, \xi, \omega^T)$ for $t \in [0,1]$. Using the fact $\omega_{t\varphi+(1-t)u}^n\wedge\eta \geq t^n \omega_{\varphi}^n\wedge\eta$ , the Comparison principle $(\ref{cp01})$ and $\{v < t\varphi+(1-t)u\} \subset \{v < u\}$ we have
\begin{equation*}
\begin{split}
t^n\int_{\{v < t\varphi +(1-t)u\}} \omega_{\varphi}^n\wedge\eta &\leq \int_{\{v<t\varphi+(1-t)u\}} \omega_{t\varphi+(1-t)u}^n\wedge\eta \\
                                                                                                  &\leq \int_{\{ v < t\varphi+(1-t)u\}} \omega_v^n\wedge\eta \\
                                                                                                  & \leq \int_{\{v < u\}} \omega_v^n\wedge\eta \\
                                                                                                  &=0
\end{split}
\end{equation*}
and $\omega_{\varphi}^n\wedge\eta(\{v < t\varphi+(1-t)u\})=0$ for $t \in (0,1]$. Then
\[
\omega_{\varphi}^n\wedge\eta(\{v<u\})=\lim_{k\rightarrow\infty}\omega_{\varphi}^n\wedge\eta(\{v<\frac{1}{k}\varphi+(1-\frac{1}{k})u\})=0
\]
This leads to a contradiction.
\end{proof}

\subsection{The space of transverse K\"ahler potentials and $(\cH, d_2)$}
The Riemannian structure on $\cH$ has been studied extensively, notably by Guan-Zhang \cite{GZ1}. Guan-Zhang proved that for any two points $\phi_1, \phi_2\in \cH$, there exists a unique $C^{1, \bar 1}_B$ geodesic which realizes the distance of $(\cH, d_2)$ and $(\cH, d_2)$ is a metric space.  The Riemannian structure would play a very central role, as in Chen's result \cite{chen01} in K\"ahler setting. 

We shall recall these results. 
For $\psi_1,\psi_2 \in T_{\phi} \mathcal{H}=C_B^{\infty}(M)$, define a $L^2$ inner product on this tangent space
\begin{equation*}
(\psi_1,\psi_2)_{\phi}=\int_{M} \psi_1\psi_2 d\mu_{\phi}
\end{equation*}
and the length $||\psi||_{\phi}$ of a vector $\psi \in T_{\phi} \cH$ is 
\begin{equation*}
||\psi||_{2, \phi}=(\int_{M} \psi_1\psi_2 d\mu_{\phi})^{\frac{1}{2}},
\end{equation*}
where we use the notation 
\begin{equation}
d\mu_\phi=\omega_\phi^n\wedge \eta_\phi=\omega^n_\phi \wedge \eta. 
\end{equation}
For a smooth path $\phi_t \in \cH$, the length of the path is defined to be
\begin{equation*}
l(\phi_t)=\int_0^1 ||\dot{\phi}_t||_{2, \phi_t}dt
\end{equation*}

This is a direct adaption of Mabuchi's metric \cite{M1} on the space of K\"ahler potentials to Sasaki setting. 
The Levi-Civita connection $\nabla$ is torsion free and satisfies
\begin{equation*}
\frac{d}{dt}(u_t, v_t)_{\phi_t}=(\nabla_{\dot{\phi}_t} u_t, v_t)_{\phi_t}+(u_t,\nabla_{\dot{\phi}_t} v_t)_{\phi_t}
\end{equation*}
for any smooth vector fields $u_t, v_t$ along the path $\phi_t$ in $\mathcal{H}_{\omega^T}$.
Let $u_t\in C^\infty_B(M)$ be smooth vector fields along a smooth curve $\phi_t$ in $\mathcal{H}$,then
\begin{equation}
\nabla_{\dot{\phi}_t} u_t=\dot{u}_t-\frac{1}{4}<\nabla \dot{\phi}_t,\nabla u_t>_{{\phi_t}}
\end{equation}
The geodesic equation can be written as
\begin{equation}\label{cma01}
\nabla_{\dot \phi_t}(\dot\phi_t)=\ddot{\phi}_t-\frac{1}{4}|\nabla \dot{\phi}_t|^2_{{\phi_t}}=0
\end{equation}
Given $\phi_0, \phi_1\in \cH$, to solve the geodesic equation, Guan-Zhang \cite{GZ1} introduced the following perturbation equation, for a path $\phi_t: M\times [0, 1]\rightarrow \R$,
\begin{equation}\label{cma02}
\left\{\begin{matrix}
\left(\ddot{\phi}_t-\frac{1}{4}|\nabla \dot{\phi}_t|^2_{\omega_{\phi_t}}\right)\omega_\phi^n \wedge \eta=\epsilon (\omega^T)^n\wedge \eta, M\times (0, 1)\\
\phi|_{t=0}=\phi_0\\
\phi|_{t=1}=\phi_1
\end{matrix}
\right.
\end{equation}
Define a function $\psi$ on $M\times [1, 3/2]$, as a subset of the cone $X$,
 \[
 \psi(\cdot, r)=\phi_{t}(\cdot)+4\log r,\;\;t=2r-2
 \] 
Set a $(1, 1)$-form by,
\[
\Omega_\psi=\omega_X+\frac{r^2}{2}\sqrt{-1}\left(\p\bar\p \psi-\frac{\p \psi}{\p r}\p\bar\p r\right)
\]
 Guan-Zhang wrote  an equivalent form of \eqref{cma02} in terms of a complex Monge-Ampere equation on $\psi$ of  the following form (with $f=r^2, \epsilon\in(0, 1] $), 
 \begin{equation}\label{cma03}
 \begin{split}
 &(\Omega_\psi)^{n+1}=\epsilon f (\omega_X)^{n+1}, M\times \left(1, \frac{3}{2}\right)\\
 &\psi|_{M\times \{r=1\}}=\phi_0, \psi|_{M\times \{r=3/2\}}=\psi_1+4\log(3/2)
 \end{split}
 \end{equation}
 
 Guan-Zhang proved the following results reagrding \eqref{cma03},
 \begin{thm}[Guan-Zhang]Fix a Sasaki structure $(M, \xi, \eta, g)$ on a compact manifold $M$. For any positive basic function $f$ and any two points $\phi_0, \phi_1\in \cH$, there exists a unique smooth solution of $\psi$ to \eqref{cma03}, satisfying the following estimates: $\psi$ is basic and there exists a constant $C>0$, depending only on $\|f^{\frac{1}{n}}\|_{C^{2}(M\times [1, \frac{3}{2}])}, \|\phi_0\|_{C^{2, 1}}, \|\phi_1\|_{C^{2, 1}}$ such that
 \begin{equation}\label{cma05}
\|\psi\|_{C^{2}_w}:= \|\psi\|_{C^1}+\sup |\Delta \psi| \leq C. 
 \end{equation}
 Denote the corresponding solution of \eqref{cma02} by $\phi_t^\epsilon$, then $\phi^\epsilon_t$ is called a $\epsilon$-geodesic (smooth) connecting $\phi_0, \phi_1$ satisfying 
 \begin{equation}\label{cma04}
 \|\phi^\epsilon_t\|_{C^1}+\sup(\ddot\phi^\epsilon+|\nabla \dot\phi^\epsilon_t|_g+\Delta_g \phi^\epsilon_t)\leq C
 \end{equation}
 When $\epsilon \rightarrow 0$, there exists a unique (weak $C^2_w$) limit $\phi_t$ of $\phi^\epsilon_t: M\times [0, 1]\rightarrow \R$ connecting $\phi_0, \phi_1$ such that $\Omega_{\phi^\epsilon+4log r}$ is positive. The later is equivalent to \[\omega_{\phi^\epsilon_t}>0, \ddot{\phi^\epsilon_t}-\frac{1}{4}|\nabla \dot{\phi^\epsilon_t}|^2_{\omega_{\phi^\epsilon_t}}>0.\]
 As a consequence, $(\cH, d_2)$ is a metric space. 
 \end{thm}

\begin{rmk}The constant $1/4$ appears in the geodesic equation 
\[
\ddot{\phi}_t-\frac{1}{4}|\nabla \dot{\phi}_t|^2_{\omega_{\phi_t}}=0
\]
This constant is insignificant. In K\"ahler setting, some authors write the constant as $1/2$ and some write as $1$, depending on the gradient $\nabla$ is interpreted as  real or complex; they differ by a constant $2$. The constant $1/4$ appears in Sasaki setting in \cite{GZ1} since the authors use the real gradient and use the space of Sasaki potentials (transverse K\"ahler potentials) defined as
\[
\{\phi: d\eta+\sqrt{-1}\p_B\bar\p_B \phi>0. \}
\]
In the following, we shall write the geodesic equation as
\[
\ddot{\phi}_t-|\nabla \dot{\phi}_t|^2_{\omega_{\phi_t}}=0,
\]
where we use complex gradient, and our choice space of transverse K\"ahler potentials is as
 \[
 \cH=\{\phi\in C^\infty_B(M): \omega^T+\sqrt{-1}\p_B\bar\p_B\phi>0\}. 
 \]
\end{rmk}

To prove $(\cH, d_2)$ is a metric space, Guan-Zhang \cite{GZ1}[Lemma 14, proof of Theorem 2] proved the following triangle inequality,

\begin{lemma}[Guan-Zhang]\label{triangle01}Let $\psi(s): [0, 1]\rightarrow \cH$ be a smooth curve, $\phi\in \cH\backslash \psi([0, 1])$. Fix $\epsilon\in (0, 1]$. Let $u^{\epsilon}\in C^\infty_B([0, 1]\times [0, 1]\times M)$ be the function such that $u^\epsilon_t(\cdot, s)$ is the $\epsilon$-geodesic connecting $\phi$ and $\psi_s$, for $t\in [0, 1]$. Then the following estimate holds,
\begin{equation}
l(u^{\epsilon}_t(\cdot, 0))\leq l(\psi)+l(u^\epsilon_t(\cdot, 1))+\epsilon C,
\end{equation}
where $C=C(\phi, \psi, g)$ is a uniform constant, independent of $\epsilon$. 
\end{lemma}

There are several estimates which are not explicitly stated or not proved in \cite{GZ1}. We include these estimates below since we shall need them below. 
Regarding \eqref{cma02}, first we have the following comparison principle, 
\begin{lemma}
Suppose we have two solutions $\varphi, \phi$ with boundary datum $\varphi_0, \varphi_1$ and $\phi_0, \phi_1$ respectively, 
\begin{equation} \left(\ddot{\phi}_t-|\nabla \dot{\phi}_t|^2_{\omega_{\phi_t}}\right)\omega_\phi^n \wedge \eta=\epsilon (\omega^T)^n\wedge \eta=\left(\ddot{\varphi}_t-|\nabla \dot{\varphi}_t|^2_{\omega_{\varphi_t}}\right)\omega_\varphi^n \wedge \eta,
\end{equation}
then we have the following
\begin{equation}\label{comparison01}
\max |\phi-\varphi|\leq \max |\phi_0-\varphi_0|+\max |\phi_1-\varphi_1|. 
\end{equation}
\end{lemma}
\begin{proof}This is a standard comparison principle. We sketch the proof for completeness. Denote the operator
\[
F(D^2\phi)=\log \det \begin{pmatrix} \ddot \phi & \nabla \dot \phi\\
(\nabla \dot \phi)^t& g_{i\bar j}^T+\phi_{i\bar j}
\end{pmatrix}-\log \det(g^T_{i\bar j})=\log \left(\ddot{\phi}_t-|\nabla \dot{\phi}_t|^2_{\omega_{\phi_t}}\right)+\log\frac{ \det(g_{i\bar j}^T+\phi_{i\bar j})}{\det(g_{i\bar j}^T)}
\]
The $\epsilon$-geodesic equation can be written as $F(D^2\phi)=\epsilon$. Now suppose $F(D^2\phi)=F(D^2\varphi)=\epsilon>0$, then \eqref{comparison01} holds. Otherwise suppose at some interior point 
\[\phi-\varphi>\max |\phi_0-\varphi_0|+\max |\phi_1-\varphi_1|. \]
Hence $\phi-\varphi+at(1-t)$ obtains its maximum at an interior point $p$ for some $a>0$. Denote $v=\phi+at(t-1)$. Then on one hand, 
\[
F(D^2v)>F(D^2\phi)=\epsilon
\]
On the other hand at $p$, $D^2 v\leq D^2 \varphi$. It follows from the concavity of $F$, we have at $p$, 
\[
F(D^2v)-F(D^2\varphi)\leq \cL_{F} (v-\varphi)\leq 0,
\]
where $\cL_{F_v}$ is the linearized operator of $F$ at $v$. Contradiction. 
\end{proof}

One can actually be more precise about the estimate \eqref{cma04} (and \eqref{cma05}). For simplicity, we state the result for \eqref{cma02}, 

\begin{lemma}\label{cma10}The $\epsilon$ geodesic $\phi^\epsilon_t$ connecting $\phi_0, \phi_1\in \cH$ satisfies the following estimate,
\begin{equation}\label{cma06}
\max |\dot\phi^\epsilon_t|\leq \max |\phi_1-\phi_0|+C \max |\nabla(\phi_1-\phi_0)|^2_g+\epsilon,
\end{equation}
where $C$ depends only on $\phi_0, \phi_1$. Moreover, we have
\begin{equation}\label{cma07}
|\nabla \phi^\epsilon_t|_g+\sup \Delta_g\phi^\epsilon\leq C(\|\phi_0\|_{C^1}, \|\phi_1\|_{C^1}, \sup \Delta_g\phi_0, \Delta_g \phi_1, g)
\end{equation}
\end{lemma}

\begin{proof}The first estimate follows from $\ddot \phi^\epsilon_t>0$ and the following $C^0$ estimate \eqref{cma08}, which can be proved similarly using the concavity of $F$. 
First there exists $a>0$ such that
\begin{equation}\label{cma08}
at(t-1)+(1-t)\phi_0+t\phi_1
\leq \phi_t^\epsilon\leq (1-t)\phi_0+t\phi_1
\end{equation}
The righthand side is a direction consequence of $\ddot \phi^\epsilon_t>0$, while the lefthand side can be argued as follows. Denote $U^a=at(t-1)+(1-t)\phi_0+t\phi_1$; we know $\phi^\epsilon_t$ agrees with $U^a$ on the boundary. Hence if $\phi^\epsilon_t<U^a$, then $\phi^\epsilon_t-U^a$ takes its minimum at some interior point $p$. At $p$, we know $D^2 \phi^\epsilon\geq D^2 U^a$. By concavity of $F$, we get (at $p$)
\[
0\leq \cL_{F_a} (\phi^\epsilon_t-U^a)\leq F(D^2\phi^\epsilon_t)-F(D^2 U^a)
\]
That is $F(D^2 U^a)\leq \log \epsilon$. This is a contradiction when $a>0$ is sufficiently large. Indeed, a direct computation shows that if $a\geq C\max |\nabla(\phi_1-\phi_0)|^2+\epsilon$, then $F(D^2 U^a)>\log \epsilon$. Hence for such choice of $a$, \eqref{cma08} holds. By convexity in $t$ direction, we know that
\[
\dot\phi^\epsilon_t (\cdot, 0)\leq \dot \phi^\epsilon_t\leq \dot \phi^\epsilon_t(\cdot, 1)
\]
It is evident to show that
\[-a+ \phi_1-\phi_0 \leq \dot\phi^\epsilon_t (\cdot, 0) \leq \phi_1-\phi_0\leq \dot \phi^\epsilon_t(\cdot, 1)\leq  a+\phi_1-\phi_0\]
Hence \eqref{cma06} follows. The gradient estimate $|\nabla \phi^\epsilon_t|$ is given by \cite{GZ1}[Proposition 2].  The estimate on $\Delta _g \phi^\epsilon_t$, depending only on $\phi_0, \phi_1$ up to second order derivative, was proved for K\"ahler setting by the first named author \cite{he12}[Theorem 1.1] (for $\epsilon=0$, it was proved earlier in \cite{BD} using pluripotential theory). The method in \cite{he12} is to deal with the equation \eqref{cma02} directly, and it can be carried over to prove the interior estimate of $\Delta_g \phi^\epsilon$ word by word (since in Sasaki setting, this estimate only involves transverse K\"ahler structure and basic functions). We skip the details. 
\end{proof}

By taking $\epsilon \rightarrow 0$, we have the following,

\begin{lemma}\label{zero01}Suppose $\phi$ is the weak geodesic connecting $\phi_0, \phi_1\in \cH$, then for some positive constant $C=C(M, g, \|\phi_0\|_{C^2}, \|\phi_1\|_{C^2})$, we have
\[
|\dot \phi|\leq \max |\phi_1-\phi_0|+C \max |\nabla\phi_1-\nabla \phi_0|_g^2
\]
As a consequence, when $\phi_0\rightarrow \phi_1$ in $\cH$, then $d_2(\phi_0, \phi_1)\rightarrow 0$. 
\end{lemma}

\begin{rmk}One can get a much sharper estimate, 
\[
|\dot \phi|\leq \max |\phi_1-\phi_0|
\]
using the uniqueness and comparison for the generalized solutions of complex Monge-Ampere in the sense of Bedford-Taylor, see \cite{D4}[Lemma 3.5] for K\"ahler setting. We shall prove this sharper version below.   
\end{rmk}
Using Lemma \ref{triangle01} and Lemma \ref{zero01}, it follows that the distance function $d_2(\phi_0, \phi_1)$ is realized by the weak geodesic $\phi$ connecting $\phi_0, \phi_1$. In particular, 
\begin{lemma}\label{distance}
Given $\phi_0, \phi_1\in \cH$, we have,
\begin{equation}\label{distance01}
d_2(\phi_0, \phi_1)=\|\dot\phi\|_{2, \phi_t}, \forall t\in [0, 1]
\end{equation}
\end{lemma}

\begin{proof}Let $\phi^\epsilon_t$ be the $\epsilon$ geodesic connecting $\phi_0, \phi_1$. Then we compute
\begin{equation}\label{dis03}
\begin{split}
\frac{d}{dt}\int_M |\dot\phi^\epsilon_t|^2 (\omega_{\phi^\epsilon_t})^n\wedge \eta=&2\int_M \dot \phi^\epsilon_t (\ddot \phi^\epsilon_t-|\nabla \dot\phi^\epsilon_t|_{\phi_t^\epsilon}) (\omega_{\phi^\epsilon_t})^n\wedge \eta\\
=&2\epsilon \int_M \dot\phi^\epsilon_t (\omega^T)^n\wedge \eta
\end{split}
\end{equation}
Since $|\dot \phi^\epsilon_t|$ is uniformly bounded, letting $\epsilon\rightarrow 0$, we get that
\[
\frac{d}{dt}\int_M |\dot\phi_t|^2 (\omega_{\phi_t})^n\wedge \eta=0. 
\]
This proves \eqref{distance01}. In particular if $\phi_0\neq \phi_1$, $\dot\phi_t$ is not identically zero for any $t$. Moreover, if $\epsilon$ is small enough, depending on $\phi_0\neq \phi_1$, then $\dot\phi^\epsilon_t$ is not identically zero for any $t\in [0, 1]$. This follows from \eqref{dis03} and it is easy to see that $\int_M |\dot\phi^\epsilon_t|^2 (\omega_{\phi^\epsilon_t})^n\wedge \eta$ has a positive lower bound for any $t$ (say $l(\phi^\epsilon_t)/2$),  if $\epsilon$ is sufficiently small. 
\end{proof}

We also have the following
\begin{thm}[Guan-Zhang, Theorem 2]\label{triangle100}For $u, v, w\in \cH$, 
\[
d_2(u, w)\leq d_2(u, v)+d_2(v, w). 
\]
\end{thm}

\subsection{The Orlicz-Finsler geometry on Sasaki manifolds}
The Orlicz-Finsler geometry on the space of K\"ahler potentials was introduced by T. Darvas \cite{D2} and it has played an important role in problems regarding csck and Calabi's extremal metric in K\"ahler geometry. In particular the Finsler metric $d_1$ will play an important role and it is used to define the properness of $\cK$-energy. 
In this section we discuss the Orlicz-Finsler geometry on Sasaki geometry.
We prove the following theorem, which is the counterpart of Darvas's  \cite{D2}[Theorem 1] in Sasaki setting.

\begin{thm}\label{darvas1}If $\chi\in \cW_p^{+}, p\geq 1$, then $(\cH, d_\chi)$ is a metric space and for any $u_0, u_1\in \cH$, the $C^{1, \bar 1}_B$ geodesic $t\rightarrow u_t$ connecting $u_0, u_1$ satisfies 
\begin{equation}
d_\chi(u_0, u_1)=\|\dot u_t\|_{\chi, u_t}, t\in [0, 1]. 
\end{equation}
\end{thm}
 
Theorem \ref{darvas1} is the generalization  for $d_2$ to general smooth Young weights. 
 This important result in T. Darvas's theory says that, the same $C^{1, \bar 1}_B$ geodesic (with respect to $d_2$) is ``length minimizing" for all $d_\chi$ metric structures and this holds in Sasaki setting.
The proof of Theorem \ref{darvas1} pretty much follows Darvas's proof \cite{D4}[Theorem 3.4], with minor modifications adapted to Sasaki setting. The main point is that only transverse K\"ahler structure is involved, and hence this is essentially the same as in K\"ahler setting. We include the details for completeness. 
 
Following T. Darvas (see \cite{D4}[Chapter 3]), we define the Orlicz-Finsler length of $v\in T_u\cH=C^\infty_B(M)$ for any weight $\chi\in \cW^+_p$:
\begin{equation}
\|v\|_{\chi, u}=\inf\left\{r>0: \frac{1}{\text{Vol}(M)}\int_M \chi\left(\frac{v}{r}\right)\omega_u^n\wedge d\eta\leq \chi(1)\right\}
\end{equation}
For simplicity, we shall assume $\text{Vol}(M)=1$ in this section. 
Given a smooth curve $\gamma: t\in [0, 1]\rightarrow \cH$, its length is computed by the formula
\begin{equation}
l_\chi(\gamma_t)=\int_0^1\|\dot\gamma_t\|_{\chi, \gamma_t}dt
\end{equation}
Furthermore, the distance $d_\chi(u_0. u_1)$ between $u_0, u_1\in \cH$ is the infimum of the $l_\chi$-length of smooth curves joining $u_0$ and $u_1$:
\begin{equation}\label{dis10}
d_{\chi}(u_0, u_1)=\inf\{l_\chi(\gamma_t): \gamma_t\;\text{is a smooth curve with}\; \gamma_0=u_0, \gamma_1=u_1\}. 
\end{equation}
First we have the following,
\begin{prop}\label{derivativeoflength}
Suppose $\chi \in \cW_p^+ \cap C^{\infty}(\mathbb{R})$. For a smooth curve $u_t (t \in [0,1])$ in $\cH$ and a vector field $f_t \in C_B^{\infty}(X)$ along this curve with $f_t \not\equiv 0$,we have
\begin{equation}
\frac{d}{dt} ||f_t||_{\chi, u_t}=\frac{\int_M \chi'(\frac{f_t}{||f_t||_{\chi,\phi_t}})\nabla_{\dot{u}_t} f_t d\mu_{u_t}}{\int_M\chi'(\frac{f_t}{||f_t||_{\chi, u_t}}) \frac{f_t}{||f_t||_{\chi, u_t}}d\mu_{u_t}}
\end{equation}
\end{prop}

\begin{proof}This works as in \cite{D2}[Proposition 3.1] word by word. We skip the details. 
\end{proof}

\begin{lemma}\label{kai length estimate}
Suppose $\chi \in \mathcal{W}_p^+ \cap C^{\infty}(\mathbb{R})$ and $u_0,u_1 \in \mathcal{H},u_0 \neq u_1$.Then the $\epsilon$-geodesics $[0,1] \owns t \rightarrow u_{t}^{\epsilon} \in \mathcal{H} $  connecting $u_0,u_1$ satisfies the following estimate:
\begin{equation}
\int_M \chi(\dot{u}_t^{\epsilon})\omega_{u_t^{\epsilon}}^n \wedge \eta \geq \max(\int_M \chi(\min(u_1-u_0,0)) \omega_{u_0}^n \wedge \eta, \int_M \chi(\min(u_0-u_1,0)) \omega_{u_1}^n \wedge \eta)-\epsilon C
\end{equation}
for all $t \in [0,1]$, where $C:=C(\chi,||u_0||_{C^2(M)},||u_1||_{C^2(M)})$
\end{lemma}

\begin{proof}This follows exactly as in K\"ahler setting \cite{D4}[Lemma 3.8],  by a direct computation and the convexity of $\chi$. 
\end{proof}

\begin{lemma}\label{derivativeofsmoothgeodesics}
Suppose $\chi \in \mathcal{W}_p^+ \cap C^{\infty}(\mathbb{R})$ and $u_0,u_1 \in \mathcal{H},u_0 \neq u_1$.Then there exists a constant $\epsilon_0$ depends on $u_0, u_1$ such that for all $\epsilon \in (0, \epsilon_0]$ the $\epsilon$-geodesic $[0,1] \owns t \rightarrow u_{t}^{\epsilon} \in \mathcal{H} $ connecting $u_0,u_1$ satisfies:
\begin{equation}
\frac{d}{dt} ||\dot{u}_t^{\epsilon}||_{\chi, u_t^{\epsilon}}=\epsilon \frac{\int_M \chi'(\frac{\dot{u}_t^{\epsilon}}{||\dot{u}_t^{\epsilon}||_{\chi,\dot{u}_t^{\epsilon}}})(\omega^T)^n \wedge \eta}{\int_M \frac{\dot{u}_t^{\epsilon}}{||\dot{u}_t^{\epsilon}||_{\chi,\dot{u}_t^{\epsilon}}} \chi'(\frac{\dot{u}_t^{\epsilon}}{||\dot{u}_t^{\epsilon}||_{\chi,\dot{u}_t^{\epsilon}}}) \omega_{u_t^{\epsilon}} \wedge \eta_{u_t^{\epsilon}} },\quad t \in [0,1].
\end{equation}
\end{lemma}

\begin{proof}
If we choose $\epsilon_0>0$ sufficiently small, then $\dot u^\epsilon_t$ is not identically zero for any $t\in [0, 1]$, if $\epsilon \in (0, \epsilon_0]$, given $u_0\neq u_1$, see Lemma \ref{distance01}. Then the results follows from Proposition \ref{derivativeoflength}.
\end{proof}

We have the following, similar to Lemma \ref{distance01} (for $d_2$), 
\begin{prop}\label{tangentestimate}
Suppose $\chi \in \mathcal{W}_p^+ \cap C^{\infty}(\mathbb{R})$ and $u_0,u_1 \in \mathcal{H},u_0 \neq u_1$.Then there exists $\epsilon_0>0$ such that for any $\epsilon \in (0,\epsilon_0]$ the $\epsilon$-geodesic $[0,1] \owns t \rightarrow u_{t}^{\epsilon} \in \mathcal{H} $  connecting $u_0,u_1$ satisfies
\begin{itemize}
\item[(i)] $||\dot{u}_t^{\epsilon}||_{\chi,\dot{u}_t^{\epsilon}}>R_0,t \in [0,1]$;
\item[(ii)] $|\frac{d}{dt} ||\dot{u}_t^{\epsilon}||_{\chi,\dot{u}_t^{\epsilon}}| \leq \epsilon R_1,t \in [0,1]$.
\end{itemize}
where $\epsilon_0,R_0,R_1$ depends on upper bounds for $||u_0||_{C^2(M)},||u_1||_{C^2(M)}$ and lower bounds for $||\chi(u_1-u_0)||_{L^1((\omega^T)^n \wedge \eta)},\frac{\omega_{u_0}^n \wedge \eta_{u_0}}{ (\omega^T)^n\wedge \eta}$ and $\frac{\omega_{u_1}^n \wedge \eta_{u_1}}{ (\omega^T)^n\wedge \eta}$.
\end{prop}

\begin{proof}
\begin{itemize}
\item[(i)] Recall the equation $(1.11)$ in \cite{D4}
\begin{equation*}
||f||_{\chi,\mu} \geq \min\{\frac{\int_{\Omega} \chi(f)d\mu}{\chi(1)} ,(\frac{\int_{\Omega} \chi(f)d\mu}{\chi(1)})^{\frac{1}{p}}\}
\end{equation*}
and Lemma \ref{kai length estimate}, the estimate in (i) follows immediately.
\item[(ii)] Choose $\epsilon_0$ small so that Lemma\ref{derivativeofsmoothgeodesics} applies. Recall the Young identity
\begin{equation*}
\chi(a)+\chi^*(\chi'(a))=a\chi'(a),a, b \in \mathbb{R},\chi'(a) \in \partial \chi(a)
\end{equation*}
Then we have
\begin{equation}
\begin{split}
|\frac{d}{dt} ||\dot{u}_t^{\epsilon}||_{\chi, u_t^{\epsilon}}| &=\epsilon \frac{|\int_M \chi'(\frac{\dot{u}_t^{\epsilon}}{||\dot{u}_t^{\epsilon}||_{\chi,\dot{u}_t^{\epsilon}}})(\omega^T)^n \wedge \eta|}{\int_M \frac{\dot{u}_t^{\epsilon}}{||\dot{u}_t^{\epsilon}||_{\chi,\dot{u}_t^{\epsilon}}} \chi'(\frac{\dot{u}_t^{\epsilon}}{||\dot{u}_t^{\epsilon}||_{\chi,\dot{u}_t^{\epsilon}}}) \omega_{u_t^{\epsilon}} \wedge \eta_{u_t^{\epsilon}} }   \\
&=\epsilon \frac{|\int_M \chi'(\frac{\dot{u}_t^{\epsilon}}{||\dot{u}_t^{\epsilon}||_{\chi,\dot{u}_t^{\epsilon}}})(\omega^T)^n \wedge \eta|}{\chi(1)+\int_M \chi^*(\chi'(\frac{\dot{u}_t^{\epsilon}}{||\dot{u}_t^{\epsilon}||_{\chi,\dot{u}_t^{\epsilon}}})) \omega_{u_t^{\epsilon}} \wedge \eta_{u_t^{\epsilon}} }  \\
&\leq \frac{\epsilon}{\chi(1)} |\int_M \chi'(\frac{\dot{u}_t^{\epsilon}}{||\dot{u}_t^{\epsilon}||_{\chi,\dot{u}_t^{\epsilon}}})(\omega^T)^n \wedge \eta|
\end{split}
\end{equation}
Then the estimates (ii) follows from (i) and the fact that $\dot{u}_t^{\epsilon}$ is uniformly bounded  in terms of $||u_0||_{C^2(M)},||u_1||_{C^2(M)}$.
\end{itemize}
\end{proof}
Next we are ready to prove the triangle inequality, as in Lemma \ref{triangle01} for $d_2$ and \cite{D2}[Proposition 3.4] in K\"ahler setting, 
\begin{prop}\label{triangleinequality}
Suppose $\chi \in \mathcal{W}_p^+ \cap C^{\infty}(\mathbb{R})$ ,$\psi_s \in \mathcal{H}$ is a smooth curve,$\phi \in \mathcal{H}\backslash{\psi([0,1])}$ and $\epsilon>0$.$u^{\epsilon} \in C^{\infty}([0,1] \times [0,1] \times M)$ is the smooth function for which $t \rightarrow u_t^{\epsilon}(.,s)=u^{\epsilon}(t,s,.)$ is the $\epsilon$-geodesic connecting $\phi$ and $\psi_s$. There exists $\epsilon_0(\phi,\psi)>0$ such that for any $\epsilon \in (0,\epsilon_0)$ the following holds:
\begin{equation*}
l_{\chi}(u_t^{\epsilon}(.,0)) \leq l_{\chi}(\psi_s)+l_{\chi}(u_t^{\epsilon}(.,1)) +\epsilon R
\end{equation*}
for some $R(\phi,\psi,\chi,\epsilon_0)>0$ independent of $\epsilon$.
\end{prop}

\begin{proof}
Fix $s \in [0,1]$. By Proposition \ref{derivativeoflength} and Proposition \ref{tangentestimate}, there exists a constant $\epsilon_0(\phi,\psi)>0$ such that for $\epsilon \in (0,\epsilon_0)$
\begin{equation*}
\begin{split}
\frac{d}{ds}l_{\chi}(u_t(.,s)) &=\int_0^1 \frac{d}{ds}||\dot{u}(t,s,.)||_{\chi,u(t,s,.)}dt  \\
                                           & =\int_0^1 \frac{\int_M \chi'(\frac{\dot{u}}{||\dot{u}||_{\chi,u}})\nabla_{\frac{du}{ds}} \dot{u} d\mu_{u_t}}{\int_M\chi'(\frac{\dot{u}}{||\dot{u}||_{\chi,u}}) \frac{\dot{u}}{||\dot{u}||_{\chi,u}}d\mu_{u_t}}dt  \\
                                           &=\int_0^1 \frac{\int_M \chi'(\frac{\dot{u}}{||\dot{u}||_{\chi,u}})\nabla_{\frac{du}{ds}} \dot{u} d\mu_{u_t}}{\chi(1)+\int_M\chi^*(\chi'(\frac{\dot{u}}{||\dot{u}||_{\chi,u}}) )d\mu_{u_t}}dt    \\
                                           &=\int_0^1 \frac{\frac{d}{dt}\int_M \chi'(\frac{\dot{u}}{||\dot{u}||_{\chi,u}})\frac{du}{ds} d\mu_{u_t}-\int_M\frac{du}{ds}\nabla_{\dot{u}}(\chi'(\frac{\dot{u}}{||\dot{u}||_{\chi,u}}))d\mu_{u_t}}{\chi(1)+\int_M\chi^*(\chi'(\frac{\dot{u}}{||\dot{u}||_{\chi,u}}) )d\mu_{u_t}}dt 
\end{split}    
\end{equation*}

Moreover we have
\begin{equation}
\nabla_{\dot{u}}(\chi'(\frac{\dot{u}}{||\dot{u}||_{\chi,u}}))d\mu_{u_t}=\chi''(\frac{\dot{u}}{||\dot{u}||_{\chi,u}})(\frac{\nabla_{\dot{u}}\dot{u}}{||\dot{u}||_{\chi,u}}-\frac{\dot{u}}{||\dot{u}||_{\chi,u}^2}\frac{d}{dt}||\dot{u}||_{\chi,u})d\mu_{u_t}
\end{equation}
It follows from Proposition \ref{tangentestimate} that  $||\dot{u}||_{\chi,t}$ is uniformly bounded away from zero and both $\nabla_{\dot{u}}\dot{u}d\mu_{u_t}$ and $\frac{d}{dt} ||\dot{u}||_{\chi, u}$ are uniformly bounded by the form $\epsilon R$, where $R$ is uniformly bounded.
Moreover $\dot{u},\frac{du}{ds}$ are uniformly bounded  independent of $\epsilon$ \cite{GZ1}[Lemma 14]. Hence
\begin{equation*}
\frac{d}{ds}l_{\chi}(u_t(.,s))=\int_0^1 \frac{\frac{d}{dt}\int_M \chi'(\frac{\dot{u}}{||\dot{u}||_{\chi,u}})\frac{du}{ds} d\mu_{u_t}}{\chi(1)+\int_M\chi^*(\chi'(\frac{\dot{u}}{||\dot{u}||_{\chi,u}}) )d\mu_{u_t}}dt +\epsilon R
\end{equation*}
where $R$ is  uniform bounded independent of $\epsilon$.

Recall that $\chi*'(\chi'(l))=l$ for $l \in \mathbb{R}$,the expression
\begin{equation*}
\frac{d}{dt} (\chi(1)+\int_M\chi^*(\chi'(\frac{\dot{u}}{||\dot{u}||_{\chi,u}}) )d\mu_{u_t})=\int_M \frac{\dot{u}}{||\dot{u}||_{\chi,u}} \chi''(\frac{\dot{u}}{||\dot{u}||_{\chi,u}}) \nabla_{\dot{u}}(\frac{\dot{u}}{||\dot{u}||_{\chi, u}})d\mu_{u_t}
\end{equation*}
is a term of type $\epsilon R$.Hence we can write

\begin{equation*}
\begin{split}
\frac{d}{ds}l_{\chi}(u_t(.,s)) &=\int_0^1 \frac{d}{dt}\frac{\int_M \chi'(\frac{\dot{u}}{||\dot{u}||_{\chi,u}})\frac{du}{ds} d\mu_{u_t}}{\chi(1)+\int_M\chi^*(\chi'(\frac{\dot{u}}{||\dot{u}||_{\chi,u}}) )d\mu_{u_t}}dt +\epsilon R   \\
                                          &= \frac{\int_M \chi'(\frac{\dot{u}(1,s)}{||\dot{u}(1,s)||_{\chi,\psi}})\frac{d\psi}{ds} d\mu_{\psi}}{\chi(1)+\int_M\chi^*(\chi'(\frac{\dot{u}(1,s)}{||\dot{u}(1,s)||_{\chi,\psi}}) )d\mu_{\psi}}          \\
                                          &\geq -||\frac{d\psi}{ds}||_{\chi,\psi}+\epsilon R                                     
\end{split}
\end{equation*}
where the last line follows from the Young inequality
\begin{equation*}
\chi(a)+\chi^*(b) \geq ab ,a,b \in \mathbb{R} 
\end{equation*}
The integration of the above inequality with respect to $s \in [0,1]$ yields the desired inequality.

\end{proof}

Now we are ready to prove Theorem \ref{darvas1}. Certainly the proof follows closely Darvas's result in K\"ahler setting \cite{D2}[Section 3].  
\begin{proof}
First we show that for $u_0,u_1 \in \mathcal{H}$ and the weak $C^{1,\overline{1}}$-geodesic $u_t$ connecting $u_0,u_1$
\begin{equation}
d_{\chi}(u_0,u_1)=l_{\chi}(u_t)
\end{equation}

We assume $u_0 \neq u_1$. We first assume $\chi \in C^{\infty}(\mathbb{R})$. Recall that, by Guan-Zhang \cite{GZ1}, $\epsilon$-geodesics $u_t^{\epsilon}$ connecting $u_0,u_1$ converge to the weak $C^{1,\overline{1}}_B$ geodesic $u_t$ in $C^{1,\alpha}$. Hence $\dot{u}_t^{\epsilon}$ converges uniformly to $\dot{u}_t$.
\begin{clm}
$||\dot{u}_t^{\epsilon}||_{\chi, u_t^{\epsilon}} \rightarrow ||\dot{u}_t^{\epsilon}||_{\chi, u_t}$ as $\epsilon \rightarrow 0$.
\end{clm} 
Recall that $\dot{u}_t^{\epsilon}$ is uniformly bounded  in terms of $||u_0||_{C^2(M)},||u_1||_{C^2(M)}$ and the estimate (i) in Proposition \ref{tangentestimate}, there exist constants $0<C_1<C_2$ such that for sufficiently small $\epsilon>0$
\begin{equation*}
C_1 \leq ||\dot{u}_t^{\epsilon}||_{\chi, u_t^{\epsilon}} \leq C_2
\end{equation*}
Then the claim follows immediately if we can prove the only cluster point of $\{||\dot{u}_t^{\epsilon}||_{\chi, u_t^{\epsilon}}\}_{\epsilon}$ is $||\dot{u}_t||_{\chi, u_t}$.Take a cluster point  $N$, after taking a subsequence,  we can assume that $||\dot{u}_t^{\epsilon}||_{\chi, u_t^{\epsilon}}\rightarrow N$ as $\epsilon \rightarrow 0$.Then $\frac{\dot{u}_t^{\epsilon}}{||\dot{u}_t^{\epsilon}||_{\chi, u_t^{\epsilon}}}$ converges to $\frac{\dot{u}_t}{N}$ uniformly.
Moreover, we have $\omega_{u_t^{\epsilon}}^n \wedge \eta_{u_t^{\epsilon}}$ converges to $\omega_{u_t}^n \wedge \eta_{u_t}$ weakly. Hence
 \begin{equation*}
 \chi(1)=\int_M \chi(\frac{\dot{u}_t^{\epsilon}}{||\dot{u}_t^{\epsilon}||_{\chi, u_t^{\epsilon}}}) \omega_{u_t^{\epsilon}}^n \wedge \eta_{u_t^{\epsilon}} \rightarrow \int_M \chi(\frac{\dot{u}_t}{N}) \omega_{u_t}^n \wedge \eta_{u_t}
  \end{equation*}
Recall $||f||_{\chi,\mu} =\alpha>0$ if and only if $\int_{\Omega} \chi(\frac{f}{\alpha})d\mu=\chi(1)$. Hence $N=||\dot{u}_t||_{\chi, u_t}$.

Then it follows from the dominated convergence theorem that
\begin{equation}\label{smoothapproximation}
\lim\limits_{\epsilon \rightarrow 0} l_{\chi}(u_t^{\epsilon}) =l_{\chi}(u_t)
\end{equation}
and $d_{\chi}(u_0,u_1) \leq l_{\chi}(u_t)$.
Next we show that
\begin{equation}
l_{\chi}(\phi_t) \geq l_{\chi}(u_t)
\end{equation}
for all smooth curves $\phi_t$ in $\mathcal{H}$ connecting $u_0,u_1$.
We can assume that $u_1 \notin \phi([0,1))$ and take $h \in [0,1)$. Applying Proposition \ref{triangleinequality} to the case $\phi=u_1$ and $\psi_s=\phi|_{[0,h]}$, letting  $\epsilon \rightarrow 0$, we can obtain
\begin{equation*}
l_{\chi}(u_{t}) \leq l_{\chi}(\phi_t|_{[0,h]})+l_{\chi}(w_t^h)
\end{equation*}
where $u_{t}$ is the $C^{1,\bar 1}$ geodesic connecting $u_1, u_0$ and $w_t^h$ is the $C^{1,\bar 1}$ geodesic connecting $u_1,\phi_h$.
By Lemma \ref{cma10}, $l_{\chi}(w_t^h) \rightarrow 0$ as $h \rightarrow 1$. Hence $l_{\chi}(\phi_t) \geq l_{\chi}(u_t)$

For the general weight $\chi \in \mathcal{W}_p^+$,
we need to do approximation as in \cite{D2}[Proposition 2.4]. There exists sequence $ \chi_k \in \mathcal{W}_{p_k}^+ \cap C^{\infty}(\mathbb{R})$ such that $\chi_k$ converges to $\chi$ uniformly on compact subsets. Then we have
\begin{equation*}
\int_0^1 ||\dot{\phi}_t||_{\chi_k,\phi_t} dt=l_{\chi_k}(\phi_t) \geq l_{\chi_k}(u_t)=\int_0^1 ||\dot{u}_t||_{\chi_k,u_t} dt
\end{equation*}
and $||\dot{\phi}_t||_{\chi_k,\phi_t} \rightarrow ||\dot{\phi}_t||_{\chi,\phi_t},||\dot{u}_t||_{\chi_k,u_t} \rightarrow ||\dot{u}_t||_{\chi, u_t}$. Moreover,$\dot{u}_t,\dot{\phi}_t$ are uniformly bounded. By the dominated convergence theorem, $l_{\chi}(\phi_t) \geq l_{\chi}(u_t)$.

Recall $l_{\chi}(u_t)=\int_0^1 ||\dot{u}_t||_{\chi,u_t} dt$ and by Lemma \ref{arcparameter} below, we have
\begin{equation*}
d_\chi(u_0, u_1)=\|\dot u_t\|_{\chi, u_t}, t\in [0, 1]\end{equation*}

Suppose $u_0 \neq u_1 \in \mathcal{H}$, take $\epsilon \rightarrow 0$ in the estimate Lemma \ref{kai length estimate} we obtain $\dot{u}_0 \not\equiv 0$ and $d_{\chi}(u_0,u_1)=||\dot{u}_0||_{\chi,u_0} >0$. This implies that $(\mathcal{H},d_{\chi})$ is a metric space.
\end{proof}

\begin{lemma}\label{arcparameter}
Let $u_t$ be the weak $C^{1, \bar 1}_B$-geodesic connecting $u_0,u_1$. Then for any $\chi \in \mathcal{W}_p^+$ and $t_0, t_1 \in [0,1]$ the following hold
\begin{equation}
d_\chi(u_0, u_1)=||\dot{u}_{t_0}||_{\chi, u_{t_0}} =||\dot{u}_{t_1}||_{\chi, u_{t_1}}
\end{equation}
\end{lemma}
\begin{proof}
It had been shown that for $\epsilon$-geodesics $u^{\epsilon}_t$ joining $u_0,u_1$, we have
\begin{equation*}
||\dot{u}_{t_0}^{\epsilon}||_{\chi, u_{t_0}^{\epsilon}} \rightarrow ||\dot{u}_{t_0}||_{\chi, u_{t_0}}, ||\dot{u}_{t_1}^{\epsilon}||_{\chi, u_{t_1}^{\epsilon}} \rightarrow ||\dot{u}_{t_1}||_{\chi, u_{t_1}}\end{equation*} 
as $\epsilon \rightarrow 0$.
Proposition \ref{tangentestimate} implies that
\begin{equation*}
| ||\dot{u}_{t_0}^{\epsilon}||_{\chi, u_{t_0}^{\epsilon}}- ||\dot{u}_{t_1}^{\epsilon}||_{\chi, u_{t_1}^{\epsilon}}| \leq |t_0-t_1|\epsilon R_1
\end{equation*}
Then taking $\epsilon \rightarrow 0$  we have $||\dot{u}_{t_0}||_{\chi, u_{t_0}} =||\dot{u}_{t_1}||_{\chi, u_{t_1}}$.
\end{proof}

Finally, we have the following triangle inequality,

\begin{lemma}\label{triangle200}For $u, v, w\in \cH$, $\chi\in \cW^+_p, p\geq 1$,
\[
d_\chi(u, w)\leq d_\chi(u, v)+d_\chi(v, w). 
\]
\end{lemma}

\section{The metric space $(\cE_p(M, \xi, \omega^T), d_p)$}

In this section we prove Theorem \ref{pluri01}. We shall follow the K\"ahler setting closely as in \cite{D2}[Section 4], but we shall only consider $d_p$ distance. Given $u_0, u_1\in \cE_p(M, \xi, \omega^T), p\geq 1$, by Lemma \ref{BK} there exists decreasing sequences $u_0^k, u_1^k\in \cH$ such that $u_0^k\searrow u_0$ and $u_1^k\searrow u_1$. We shall prove that  the following formula for distance $d_p$ is well-defined,
\begin{equation}\label{distance10}
d_p(u_0, u_1)=\lim_{k\rightarrow \infty} d_p(u_0^k, u_1^k)
\end{equation}
and the definition in \eqref{distance10} coincides with \eqref{dis10} (we only consider $\chi(l)=|l|^p/p$).
We have the following
\begin{thm}\label{distance11}$(\cE_p, d_p)$ is a geodesic metric space extending $(\cH, d_p)$. 
\end{thm}

We start with the notion of generalized solution of complex Monge-Ampere in the sense of Bedford-Taylor in Sasaki setting, which was considered by van Coevering in \cite{VC}, by adapting  the complex Monge-Ampere operator for basic functions in  $\text{PSH}(M, \xi, \omega^T)\cap L^\infty$ to Sasaki setting. van Coevering discussed in particular weak solution in $\text{PSH}(M, \xi, \omega^T)\cap C^0(M)$ \cite{VC}[Section 2.4]. 
Let $S=[0, 1]\times S^1$ be the cylinder and $N=M\times S$. Then $N$ is a manifold of dimension $2n+3$ with boundary and $N$ has a transverse holomorphic structure, simply the product structure of transverse holomorphic structure on $M$ and holomorphic structure on $S$. A path $\phi: [0, 1]\rightarrow C^\infty_B(M)$ corresponds to an $S^1$-invariant function  $\Phi_w$ on $N$. 
If $\phi_t$ is a smooth path in $\cH$ then a direct computation gives,
\begin{equation}\label{vc01}
(\pi^*\omega^T+\sqrt{-1}\p_B\bar\p_B \Phi)^{n+1}=c_m(\ddot \phi-|\nabla\dot\phi|^2_{\omega^T_{\phi_t}}) (\omega^T_{\phi_t})^n\wedge{dw\wedge d\bar w}
\end{equation}
Note that this choice of complexification (see van Coevering \eqref{vc01}) is different with the choice of Guan-Zhang \eqref{cma03}. It seems that \eqref{vc01} would be more natural to discuss weak solutions. By \eqref{vc01}, a smooth geodesic then corresponds to a solution of homogeneous complex Monge-Ampere for basic function $\Phi: N\rightarrow \R$,
\[
(\pi^*\omega^T+\sqrt{-1}\p_B\bar\p_B \Phi)^{n+1}\wedge \eta=0.
\]
We define a \emph{weak geodesic} between $u_0, u_1\in \text{PSH}(M, \xi, \omega^T)\cap L^\infty$ as follows, for $\Phi(\cdot, w)=\Phi(\cdot, t)\in \text{PSH}({N}^{\circ}, \xi, \pi^*\omega^T)\cap L^\infty$, $(t=\text{Re}(w))$, it satisfies
\begin{equation}\label{bounded01}
\begin{cases}
(\pi^*\omega^T+\sqrt{-1}\p_B\bar\p_B \Phi)^{n+1}\wedge \eta=0\\
\lim_{t\rightarrow 0}\Phi(\cdot, t)=u_0,\; \lim_{t\rightarrow 1}\Phi(\cdot, t)=u_1
\end{cases}
\end{equation}

We have the following \emph{strong maximum principle},  see \cite{VC}[Theorem 2.5.3], \cite{Blocki12}[Theorem 21] and \cite{D4}[Theorem 3.2].
\begin{lemma}\label{strong01}
Let $u, v\in  \text{PSH}({N}^{\circ}, \xi, \pi^*\omega^T)\cap L^\infty(N)$. Suppose that
\[
(\pi^*\omega^T+\sqrt{-1}\p_B\bar\p_B u)^{n+1}\wedge \eta\leq (\pi^*\omega^T+\sqrt{-1}\p_B\bar\p_B v)^{n+1}\wedge \eta
\]
and $\lim_{x\rightarrow \p N} (u-v)(x)\geq 0$, then $u\geq v$ on $N$. 
\end{lemma}

\begin{proof} Our proof is similar to K\"ahler case, see \cite{D4}[Theorem 3.2].  Fix $\epsilon>0$ and $v_\epsilon:=\max\{u, v-\epsilon\}\in \text{PSH}(N^\circ, \xi, \omega^T)\cap L^\infty$. Then $v_\epsilon=u$ near the boundary $\p N=M\times S^1\times \{t=0\}\cup M\times S^1\times \{t=1\}$. Hence it is enough to show that $u=v_\epsilon$ on $N$.

We write $N=M\times S$ and $\omega_u=\pi^*\omega^T+dd^c_B u$ etc. 
Note that on each foliation chart $W_\alpha=(-\delta, \delta)\times V_\alpha$ of $M$, we have the following inequality on $V_\alpha\times S$ for complex Monge-Ampere measure \cite{BN}[Theorem 2.2.10]
\[
\omega_{v_\epsilon}^{n+1}\geq \chi_{\{u\geq v-\epsilon\}\cap V_\alpha} \omega_u^{n+1}+ \chi_{\{u< v-\epsilon\}\cap V_\alpha} \omega_v^{n+1}\geq \omega_u^{n+1}
\]
It follows that on $N$, we have
\[
\omega_{v_\epsilon}^{n+1}\wedge \eta\geq \omega_u^{n+1}\wedge \eta
\]
Then we have the following, 
\begin{equation}\label{i001}
0\leq \int_N (v_\epsilon-u)(\omega_{v_\epsilon}^{n+1}-\omega_u^{n+1})\wedge \eta
\end{equation}
Using integration by parts, we obtain that
\[
\int_N d(u-v_\epsilon)\wedge d^c_B(u-v_\epsilon)\wedge \omega_u^k\wedge \omega_{v_\epsilon}^{n-k}\wedge \eta=0, 0\leq k\leq n.
\]
By an induction argument as in \cite{D4}[Theorem 3.2], we can prove that
\[
\int_N d(u-v_\epsilon)\wedge d^c_B(u-v_\epsilon)\wedge \omega_u^k\wedge (\pi^*\omega^T)^{n-k}\wedge \eta=0, 0\leq k\leq n.
\]
For $k=n$, this shows that 
\[
\int_{M\times S} d(u-v_\epsilon)\wedge d^c_B(u-v_\epsilon)\wedge (\pi^*\omega^T)^n\wedge \eta=0.
\]
Writing $\rho=u-v_\epsilon$, this reads
\[\int_{M\times S} |\p_t\rho|^2 dt\wedge ds \wedge (\pi^*\omega^T)^n\wedge \eta=0\]
Hence $\p_t\rho=0$. Since $\rho=0$ near the boundary $\p N=M\times S^1\times \{t=0\}\cup M\times S^1\times \{t=1\}$, this shows that $\rho=0$. It completes the proof. 
\end{proof}

\begin{rmk}One can certainly formulate a general version of comparison principle as in \cite{D4}[Theorem 3.2]. But one would need certainly a (transverse) K\"ahler form. Note that $\pi^*\omega^T$ is not transverse K\"ahler (it is zero along $S$-direction). Here we use the product structure of $N$. 
\end{rmk}

With this maximum principle for bounded TPSH, we have the following,
\begin{lemma}\label{strong02}Given $u_0, u_1\in \cH$, let $u_t: [0, 1]\rightarrow \cH$ be the unique $C^{1, \bar 1}_B$ geodesic connecting $u_0, u_1$. Then we have the following,
\[
\|\dot u_t\|_{C^0}\leq \|u_0-u_1\|_{C^0}, \forall t\in [0, 1].
\]
\end{lemma}
\begin{proof} Note that this gives a much sharper estimate than Lemma \ref{zero01}. The proof follows the K\"ahler setting \cite{D4}[Lemma 3.5]. Denote $C=\max|u_0- u_1|$. By the convexity of $u$ in $t$-variable, we know that
\[
\dot u_0\leq \dot u_t\leq \dot u_1.
\]
Note that $v_t=u_0-Ct$ is a smooth geodesic connecting $u_0$ and $u_0-C$. Hence its complexification gives a solution to \eqref{bounded01}. By Lemma \ref{strong01}, we know that $v_t\leq u_t$, for $t\in [0, 1]$, since $u_0-C\leq u_1$. It follows that $-C\leq \dot u_0$. Similarly one can prove that $\dot u_1\leq C$, by considering $\tilde v_t=u_0+Ct$. 
\end{proof}

\begin{rmk}The upper envelop construction was used to construct bounded weak geodesic segment in K\"ahler setting by Berndtsson \cite{Berndtsson}, where he proved that Lemma \ref{strong02} holds for $u_0, u_1\in \text{PSH}(M, \omega)$ (when $(M, \omega)$ is K\"ahler). A direct adaption to Sasaki setting using Lemma \ref{strong01} would lead to an extension of Berndtsson's result to Sasaki setting.
\end{rmk}

In general, $\Phi(\cdot, w)\in \text{PSH}(N^\circ, \xi, \pi^*\omega^T)$ will be called \emph{weak subgeodesic}, if $\Phi(\cdot, )=\Phi(\cdot, \text{Re}(w))$, $(t=\text{Re}(w))$. For $u_0, u_1\in \text{PSH}(M, \xi, \omega^T)$, we define
\begin{equation}\label{weakgeodesic}
u=\sup\{ \Phi: \Phi(\cdot, t)\in  \text{PSH}(N^\circ, \xi, \pi^*\omega^T), \lim_{t\rightarrow 0, 1}\Phi(\cdot, t)\leq u_{0, 1}\}
\end{equation}
We have the following,
\begin{prop}  $u\in \text{PSH}(N^\circ, \xi, \pi^*\omega^T)$. Denote $u_t=u(\cdot, t)$. We refer $t\rightarrow u_t$ to the weak geodesic segment connecting $u_0, u_1$. 
\end{prop}

\begin{proof}Note that usc $u^*$ is basic, and $u^*\in \text{PSH}(N^\circ, \xi, \pi^*\omega^T)$. Since $\Phi$ is convex in $t$ direction, it follows that $\Phi(\cdot, t)\leq (1-t)u_0+tu_1$. Hence $u_t\leq (1-t)u_0+t u_1$. It follows that
\[
u^*\leq (1-t)u_0+tu_1
\]
In other words, $u^*\leq u$ by definition. It follows that $u^*=u$. 
\end{proof}

\begin{prop}\label{coincide4}
If $u_0, u_1 \in \text{PSH}(M, \xi, \omega^T)\cap L^{\infty}(M)$, $u$ is defined by $(\ref{weakgeodesic})$ and $u_t=u(\cdot, t)$ is the weak geodesic. Let $C$ be a constant $\geq ||u_1-u_0||_{L^{\infty}(M)}$.
\begin{enumerate}
\item We have
\begin{equation}\label{convex4}
\max(u_0 -Ct, u_1-C(1-t)) \leq u_t \leq (1-t) u_0+tu_1
\end{equation}
\item $u_t \in \text{PSH}(M, \xi, \omega^T) \cap L^{\infty}(M)$ and $u$ is the unique solution of (\ref{bounded01}).
\item $u_t$ is uniformly Lipschitz continuous with respect to $t$:
\[
|u_t-u_s| \leq C|s-t|.
\]
for $s, t \in [0,1]$.
\item The derivatives $\dot{u}_0, \dot{u}_1$ exists and 
\[
|\dot{u}_0|\leq C, \quad |\dot{u}_1| \leq C.
\]
\end{enumerate}
\end{prop}

\begin{proof}
\begin{enumerate}
\item It is obvious that $u_0-Ct, u_1-C(1-t)$ are weak subgeodesics. It follows from the definition of $u_t$ $(\ref{weakgeodesic})$ that
\[
\max(u_0-Ct, u_1-C(1-t)) \leq u_t
\]
The other half of the inequality comes from the convexity of $u_t$ with respect to $t$. 

\item By the inequality $(\ref{convex4})$ we have $u_t \in \text{PSH}(M, \xi, \omega^T) \cap L^{\infty}(M)$ and $\lim\limits_{t \rightarrow 0,1} u_t=u_{0,1}$. Then $u \in \text{PSH}({N}^{\circ}, \xi, \pi^*\omega^T)\cap L^\infty$. Using the classical Perron-Bremmerman  argument we have $(\pi^*\omega^T+\sqrt{-1}\partial_B\overline{\partial}_B u)^{n+1}\wedge\eta=0$. Hence $u$ is a solution of $(\ref{bounded01})$. The uniqueness of the solution of $(\ref{bounded01})$ follows from the strong maximum principle.

\item If one of $s, t$ equals to $0$ or $1$, the required inequality is a direct consequence of $(\ref{convex4})$. If $0<s<t<1$, by the convexity of $u_t$ with respect to $t$ we have
\[
\frac{t-s}{s}(u_s-u_0) \leq u_t-u_s \leq \frac{t-s}{1-s}(u_1-u_s)
\]
and the inequality follows from the case $t=0,1$ we have proved.

\item By the convexity of $u_t$ we have 
\[
\frac{u_{t_1}-u_0}{t_1} \leq \frac{u_{t_2}-u_0}{t_2}
\]
for $0<t_1 <t_2$. These quantities are uniformly bounded by $C$. Hence $\dot{u}_0$ exists and $|\dot{u}_0| \leq C$. The case of $\dot{u}_1$ follows by a similar argument.
\end{enumerate}
\end{proof}

\begin{rmk}
If $u_0, u_1 \in \cH_{\triangle}$, the weak geodesic $u_t$ coincides with the $C_B^{1,\bar1}$ geodesic.
\end{rmk}

\begin{prop}\label{homweakgeodesic}
Let $u_0^k,u_1^k\in\text{PSH}(M,\xi,\omega^T)$ be sequences decreasing to $u_0,u_1\in \text{PSH}(M,\xi,\omega^T)$ respectively. Suppose that $u_t^k,u_t \in \text{PSH}(M,\xi,\omega^T)$ be the weak geodesic connecting $u_0^k,u_1^k$ and $u_0,u_1$ respectively. Then
\begin{enumerate}
\item $u_t^k$ decreases to $u_t$ for $t \in [0,1]$;
\item For any $t_1,t_2\in[0,1]$, $[0,1] \ni t \rightarrow u_{(1-t)t_1+tt_2} \in \text{PSH}(M, \xi, \omega^T)$ is the weak geodesic connecting $u_{t_1}$ and $u_{t_2}$.
\end{enumerate}
\end{prop}
\begin{proof}
\begin{enumerate}
\item By the definition of $u^k_t$ (\ref{weakgeodesic}) it is obvious that $\{u_t^k\}_{k\in\mathbb{N}}$ is decreasing and $v_t := \lim\limits_{k\rightarrow\infty} u_t^k \in \text{PSH}(M, \xi, \omega^T)$. Again by the definition of $u_t^k, u_t$ (\ref{weakgeodesic}) we have $u^k_t \geq u_t$, hence $v_t \geq u_t$.

Recall that $u_t^k$ is convex with respect to $t$. Then $u_t^k \leq (1-t)u_0^k+t u_1^k$ and $v_t \leq (1-t)u_0+t u_1$. It follows from the definition of $u_t$ (\ref{weakgeodesic}) that $v_t \leq u_t$.

Consequently the sequence $u_t^k$ decreases to $u_t$ for $t \in [0,1]$.

\item Recall that $u_0, u_1$ are the decreasing limits of their canonical cutoffs, it follows from part (1) that we only have to prove the proposition for $u_0, u_1$ in $L^{\infty}(M)$.  $v_t:=u_{(1-t)t_1+tt_2}$ be a path connecting $u_{t_1}, u_{t_2}$. By Proposition \ref{coincide4} we have $\lim\limits_{t\rightarrow0,1} v_t=u_{t_1,t_2}$ and $\Phi(\cdot,t)=v_t$ is a solution of the equation $(\ref{bounded01})$ with initial data $u_{t_1},u_{t_2}$. Then it follows from Proposition \ref{coincide4}(2) that $v_t=u_{(1-t)t_1+tt_2}$ is the weak geodesic connecting $u_{t_1}, u_{t_2}$.
\end{enumerate}
\end{proof}

\begin{lemma}[Rooftop formula]\label{rf}Suppose $u_0, u_1\in \text{PSH}(M, \xi, \omega^T)$ and $t\rightarrow u_t$ is the weak geodesic segment connecting $u_0, u_1$. Then
\begin{equation}
\inf_{t\in (0, 1)} (u_t-t\tau)=P(u_0, u_1-\tau), \tau\in \R
\end{equation}
Moreover, for any $\tau\in \R$, we have
\begin{equation}
\{\dot u_0\geq \tau\}=\{P(u_0, u_1-\tau)=u_0\}. 
\end{equation}
If $u_0, u_1\in \cE_p(M, \xi, \omega^T)$, then $u_t\in \cE_p(M, \xi, \omega^T)$. 
\end{lemma}

\begin{proof}First note that $t\rightarrow v_t=u_t-\tau t$ is the weak geodesic connecting $u_0, u_1-\tau$, hence the proof can be reduced to the particular case $\tau=0$. 
By definition $P(u_0, u_1)\leq u_0, u_1$. As a result, the constant weak subgeodesic $t\rightarrow h_t:=P(u_0, u_1)$ is a candidate for definition of $u_t$, hence $h_t\leq u_t, t\in [0, 1]$. It follows that $P(u_0, u_1)\leq \inf_{t\in [0, 1]} u_{t}$.

For the other direction, we use Kiselman minimum principle \cite{Demailly}[Chapter I, Theorem 7.5], which asserts that $w:=\inf_{t\in [0, 1]} u_t\in \text{PSH}(M, \xi, \omega^T)$ (note that $u_t$ is a genuine plurisubharmonic function on foliation charts, for each $t$ and $u_t$ is convex in $t$-variable; hence Kiselman minimum principle applies, as in K\"ahler setting). Note that $u_t\leq (1-t)u_0+t u_1$, it follows that $w$ is a candidate for $P(u_0, u_1)$ and hence $w\leq P(u_0, u_1)$. This completes the proof. 
\end{proof}

Now we prove Theorem \ref{distance11}, through a series of propositions and lemmas, following \cite{D2}[Section 4] (and in particular \cite{D4}[Section 3]). 
\begin{lemma}\label{dis100}Suppose $u, v\in \cH$ with $u\leq v$. We have
\begin{equation}\label{distance12}
\max\left\{\frac{1}{2^{n+p}}\int_M |u-v|^p\omega_u^n\wedge \eta, \int_M |u-v|^p\omega_v^n\wedge\eta\right\}\leq d_p(u, v)^p\leq \int_M |u-v|^p\omega_u^n\wedge\eta
\end{equation}
\end{lemma}

\begin{proof}Let $w_t: [0, 1]\rightarrow \cH$ be the $C^{1, \bar 1}_B$ geodesic connecting $u$ and $v$. By Theorem \ref{darvas1}, we have
\begin{equation}\label{distance13}
d_p(u, v)^p=\int_M |\dot w_0|^p\omega_u^n\wedge \eta=\int_M |\dot w_1|^p\omega_v^n\wedge \eta
\end{equation}
By Lemma \ref{strong01}, we have $u\leq w_t$ given $u\leq v$. Since $w_t$ is convex in $t$, it follows that 
\begin{equation}\label{distance14}
0\leq \dot w_0\leq v-u\leq \dot w_1.
\end{equation}
It then follows that, by \eqref{distance13} and \eqref{distance14}, 
\begin{equation}\label{distance15}
 \int_M |u-v|^p\omega_v^n\wedge\eta\leq d_p(u, v)^p\leq \int_M |v-u|^p \omega_u^n\wedge \eta. 
\end{equation}
Next we use $\omega_u^n\wedge \eta\leq 2^n \omega_{(\frac{u+v}{2})}^n\wedge \eta$ to obtain that
\begin{equation*}
2^{-n}\int_M |u-v|^p\omega_u^n\wedge \eta\leq \int_M |u-v|^p \omega_{(\frac{u+v}{2})}^n\wedge\eta
\end{equation*}
We write the righthand side above as follows and apply \eqref{distance15} for $u\leq (u+v)/2$ to obtain,
\[
2^{-p}\int_M |u-v|^p \omega_{(\frac{u+v}{2})}^n\wedge\eta= \int_M \left|u-\frac{u+v}{2}\right|^p \omega_{(\frac{u+v}{2})}^n\wedge\eta\leq d_p\left(u, \frac{u+v}{2}\right)^p
\]
The lemma below implies that $d_p(u, (u+v)/2)\leq d_p(u, v)$, completing the proof. 
\end{proof}

\begin{lemma}\label{order1}Suppose $u, v, w\in \cH$ and $u\leq v\leq w$. Then we have, \[d_p(u, v)\leq d_p(u, w), d_p(v, w)\leq d_p(u, w)\]
\end{lemma}
\begin{proof}Let $\alpha_t, \beta_t$ be the $C^{1, \bar 1}_B$ geodesic segments connecting $u, v$ and $u, w$ respectively. Since $u\leq v\leq w$, by Lemma \ref{strong01} we have
$u\leq \alpha_t\leq v$ and $u\leq \beta_t\leq w$; moreover, $\alpha_t\leq \beta_t$.  Since $\alpha_0=\beta_0$, this gives that $0\leq \dot \alpha_0\leq \dot \beta_0$. Theorem \ref{darvas1} then implies that $d_p(u, v)\leq d_p(u, w)$. Similarly we can prove $d_p(v, w)\leq d_p(u, w)$. 
\end{proof}

Next we prove that the distance formula \eqref{distance10} is well-defined and agrees with the original definition \eqref{dis10}.

\begin{lemma}Given $u_0, u_1\in \cE_p(M, \xi, \omega^T)$, the limit \eqref{distance10} is finite and independent of the approximating sequences $u_0^k, u_1^k\in \cH$. 
\end{lemma}

\begin{proof}
First we show that given $u\in \cE_p(M, \xi, \omega^T)$ and a sequence $\{u_k\}_{k\in \N}\subset \cH$ is a decreasing sequence converging to $u$. Then as $l, k\rightarrow \infty$,
$d_p(u_l, u_k)\rightarrow 0.$ We can assume that $l\leq k$ and hence $u_k\leq u_l$. Lemma \ref{dis100} then implies that 
\[
d_p(u_l, u_k)^p\leq \int_M |u_l-u_k|^p\omega_{u_k}^n\wedge \eta.
\]
Clearly we have $u-u_l\leq u_k-u_l\leq 0$ and $u-u_l, u_k-u_l\in \cE_p(M, \xi, \omega_{u_l})$. Hence applying Proposition \ref{fe} for the class $\cE_p(M, \xi, \omega_{u_l})$, we obtain that
\begin{equation}\label{dis102}
d_p(u_l, u_k)^p\leq \int_M |u_l-u_k|^p\omega_{u_k}^n\wedge \eta\leq (p+1)^n\int_M |u-u_l|^p\omega_{u_l}^n\wedge \eta. 
\end{equation}
As $u_l$ decreases to $u\in \cE_p(M, \xi, \omega^T)$, the monotone convergence theorem implies that the righthand side above converges to zero as $l\rightarrow\infty$, hence $d_p(u_l, u_k)\rightarrow 0$ as $l, k\rightarrow\infty$. 
Now by Lemma \ref{triangle200}, we know that
\[
|d_p(u_0^l, u_1^l)-d_p(u_0^k, u_1^k)|\leq d_p(u_0^l, u_0^k)+d_p(u_1^l, u_1^k)\rightarrow 0, l, k\rightarrow \infty.
\]
Hence this proved that the limit \eqref{distance10} is convergent and finite. 

Next we show that the limit is independent of the choice of approximating sequences. Let $v_0^l, v_1^l$ be other approximating sequences. Certainly we can assume the sequences are strictly decreasing, by adding small constants if necessary.  Fix $k$ and consider the sequence $\{\max \{u_0^{k+1}, v_0^j\}_{j\in \N}\}$ decreases pointwise to $u_0^{k+1}$. By Dini's lemma, the convergence is uniform (for fixed $k$) and hence we can choose $j_k$ sufficiently large such that $v^j_0<u^k_0$, $j\geq j_k$. Repeating the argument we can assume $v_1^{j}<u^k_1$, for $j\geq j_k$. By triangle inequality again, we have
\[
|d_p(v_0^j, v_1^j)-d_p(u_0^k, u_1^k)|\leq d_p(v_0^j, u_0^k)+d_p(v_1^j, u_1^k), j\geq j_k
\]
By \eqref{dis102} we know that if $k$ is sufficiently large, $d_p(v_0^j, u_0^k)+d_p(v_1^j, u_1^k)$ is sufficiently small. Hence the distance $d_p(u_0, u_1)$ is independent of the choice of approximating sequence. 
\end{proof}

We choose a decreasing sequence $\{u_0^k\}_{k\in \N}, \{u_1^k\}_{k\in \N}\in \cH$ such that $u_0^k\searrow u_0, u^k_0\searrow u_1$.  We connect $u_0^k, u_1^k$ by the unique $C^{1, \bar 1}$ geodesic segment $u^k_t$. By Lemma \ref{strong01}, it follows that $u^k_t$ decreases in $k$. Hence the  limit $\lim_{k\rightarrow \infty} u^k_t
$ exists. Using Dini's lemma as above, one can show that the limit does not depends on the choice of approximating sequence. Indeed, the limit coincides with the weak geodesic segment defined above,
\[
u_t=\lim_{k\rightarrow \infty} u^k_t
\]

\begin{lemma}We have $t\rightarrow u_t$ is a $d_p$-geodesic in the sense that \[d_p(u_{t_1}, u_{t_2})=|t_1-t_2|d_p(u_0, u_1), s, t\in [0, 1].\]
\end{lemma}

\begin{proof}
Let $\{u_0^k\}_k,\{u_1^k\}_k \in \cH$ be sequences strictly decreasing to $u_0,u_1$ respectively and $u_t^k \in \cH_{\triangle}$ the $C^{1,\bar 1}$ geodesic connecting $u_0^k,u_1^k$. By Theorem \ref{darvas1} we have
\begin{equation*}
d_p(u_0,u_1)^p=\lim_{k\rightarrow\infty}d_p(u_0^k,u_1^k)^p=\lim_{k\rightarrow\infty}\int_M |\dot{u}_0^k|^p\omega_{u_0^k}^n\wedge\eta
\end{equation*}
For $l \in (0,1)$ the strong maximum principle Lemma \ref{strong01} implies that  $u_l^k$ strictly decreases to $u_l$. Then one can choose a sequence $\{w_l^k\}_k \in \cH$  such that
\begin{enumerate}
\item $u_l^k \leq w_l^k \leq u_l^{k+1}$;
\item For the $C^{1,\bar 1}$ geodesic $v_t^k$ connecting $u_0^k$ and $w_l^k$ with $v_0^k=u_0^k, v_1^k=w_l^k$ we have
\begin{equation*}
|\int_M|\dot{v}_0^k|^p \omega_{u_0^k}^n\wedge\eta-l^p\int_M|\dot{u}_0^k|^p\omega_{u_0^k}\wedge\eta|<\frac{1}{k}
\end{equation*}
\end{enumerate}
In fact there exists a sequence $\varphi^j \in \cH$ decreasing to $u_l^k$. By Dini's Lemma $\varphi^j$ converges to $u_k^l$ uniformly. It follows from Lemma \ref{tangentapp04} that for $j$ big enough, $w_l^k=\varphi^j$ will satisfy our requirements.
Then we have
\begin{equation*}
d_p(u_0,u_l)^p=\lim_{k\rightarrow\infty}d_p(u_0^k, w_l^k)^p=\lim_{k\rightarrow\infty}\int_M||\dot{v}_0^k||\omega_{u_0^k}^n\wedge\eta=l^pd_p(u_0,u_1)^p
\end{equation*} 
Hence $d_p(u_0,u_l)=ld_p(u_0,u_1)$ for $l \in [0,1]$.

Without loss of generality we assume that $0 \leq t_1\leq t_2\leq 1$. By the Proposition \ref{homweakgeodesic} $h_t=u_{(1-t)t_2}$ is the weak geodesic connecting $u_{t_2}$ and $u_0$. It follows from the results above we have
\[
d_p(u_{t_2},u_{t_1})=(1-\frac{t_1}{t_2})d_p(u_{t_2},u_0)=(t_2-t_1)d_p(u_1,u_0)
\]
This completes the proof.
\end{proof}

\begin{lemma}\label{tangentapp04}
$u_0,u_1\in \text{PSH}(M,\xi,\omega^T)\cap L^{\infty}$. Let $\{u_1^k\}_{k\in \mathbb{N}} \in \text{PSH}(M,\xi,\omega^T)\cap L^{\infty}$ be a sequence decreasing to $u_1$ and  $u_t, u_t^k \in \text{PSH}(M,\xi,\omega^T)\cap L^{\infty}$ the weak geodesic connecting $u_0,u_1$ and $u_0, u_1^k$ respectively.  Then
\[
\lim_{k \rightarrow \infty} \int_M |\dot{u}_0^k|^p\omega_{u_0}^n\wedge\eta=\int_M |\dot{u}_0|^p\omega_{u_0}^n\wedge\eta
\]
\end{lemma}
\begin{proof}
Denote by $C=\max(||u_1^1-u_0||_{L^{\infty}}, ||u_1-u_0||_{L^{\infty}})$. It follows Proposition \ref{coincide4} that $||\dot{u}_0||_{L^{\infty}} \leq C, ||\dot{u}_0^k||_{L^{\infty}} \leq C$. By Proposition \ref{homweakgeodesic} the sequence $\{u_t^k\}_{k \in \mathbb{N}}$ decreases to $u_t$ hence the sequence $\{\dot{u}_0^k\}_{k\in\mathbb{N}}$ is decreasing  with $\dot{u}_0^k \geq \dot{u}_0$. 

Moreover we have $\dot{u}_0^k$ decreases to $\dot{u}_0$. If this is not true, we can find $x_0 \in M, a \in \mathbb{R}$ such that $\dot{u}_0^k > a > \dot{u}_0$. Then there exists $0<t_0<1$ such that $u_t^k(x_0) > u_0+at > u_t(x_0)$ for $t \in [0, t_0]$. It contradicts with the fact that $u_t^k$ decreases to $u_t$.

Then the Lemma follows from Lebesgue's dominated convergence theorem.\end{proof}

The following Pythagorean formula  plays an essential role in Darvas's results \cite{D1, D2} and we have the same,

\begin{thm}[Pythagorean formula]\label{Pythagorean}
Given $u_0, u_1\in \cE_p(M, \xi, \omega^T)$, we have $P(u_0, u_1)\in \cE_p(M, \xi, \omega^T)$ and 
\begin{equation}
d_p(u_0, u_1)^p=d(u_0, P(u_0, u_1))^p+d_p(u_1, P(u_0, u_1))^p. 
\end{equation}
\end{thm}
\begin{proof}
First we assume that $u_0,u_1\in\cH$. It follows from Theorem \ref{rooftop101} that $P(u_0,u_1) \in \cH_{\triangle}$. Let $u_t$ be the $C_B^{1,\bar1}$ geodesic connecting $u_0,u_1$.
Let $v_t$ be the weak geodesic connecting $P(u_0,u_1),u_1$. It follows from the strong maximum principle that $P(u_0,u_1) \leq v_t$ for $t \in [0,1]$. Hence we have $\dot{v}_0 \geq 0$. By Lemma \ref{distancef2}, Lemma \ref{rf}, the definition of rooftop and Lemma \ref{decomposition} we have
\begin{equation*}
\begin{split}
d_p(P(u_0,u_1),u_1)^p &=\int_M|\dot{v}_0|^p\omega_{P(u_0,u_1)}^p\wedge\eta \\
                           &=\int_{\{\dot{v}_0>0\}} |\dot{v}_0|^p\omega_{P(u_0,u_1)}^n\wedge\eta \\
                           &=p\int_0^{\infty}s^{p-1}\omega_{P(u_0,u_1)}^n\wedge\eta\{\dot{v}_0 \geq s\} ds\\
                           &=p\int_0^{\infty}s^{p-1}\omega_{P(u_0,u_1)}^n\wedge\eta\{P(P(u_0,u_1),u_1-s)=P(u_0,u_1)\}ds\\
                           &=p\int_0^{\infty}s^{p-1}\omega_{P(u_0,u_1)}^n\wedge\eta\{P(u_0,u_1-s)=P(u_0,u_1)\}ds \\
                           &=p\int_0^{\infty}s^{p-1}\omega_{u_0}^n\wedge\eta\{P(u_0,u_1-s)=P(u_0,u_1)=u_0\}ds\\
                           &=p\int_0^{\infty}s^{p-1}\omega_{u_0}^n\wedge\eta\{P(u_0,u_1-s)=u_0\}ds \\
                           &=p\int_0^{\infty}s^{p-1}\omega_{u_0}^n\wedge\eta\{\dot{u}_0 \geq s\}ds \\
                           &=\int_{\{\dot{u}_0>0\}} |\dot{u}_0|^p\omega_{u_0}^n\wedge\eta
 \end{split}
\end{equation*}
By a similar argument we also have
\begin{equation*}
d_p(u_0,P(u_0,u_1))^p =\int_{\{\dot{u}_0<0\}}|\dot{u}_0|^p\omega_{u_0}^n\wedge\eta
\end{equation*}
Now using Theorem \ref{darvas1} we have
\begin{equation*}
\begin{split}
d_p(u_0,u_1)^p &=\int_M |\dot{u}_0|^p\omega_{u_0}^n\wedge\eta\\
                          &=\int_{\{\dot{u}_0<0\}}|\dot{u}_0|^p\omega_{u_0}^n\wedge\eta+\int_{\{\dot{u}_0>0\}}|\dot{u}_0|^p\omega_{u_0}^n\wedge\eta\\
                          &=d_p(u_0,P(u_0,u_1))^p+d_p(P(u_0,u_1),u_1)^p
\end{split}
\end{equation*}
and the Pythagorean formula holds for smooth potentials $u_0,u_1\in \cH$. 

For the general case we can choose sequences $\{u_0^k\}_{k\in \mathbb{N}}, \{u_1^k\}_{k\in\mathbb{N}} \in \cH$ decreases to $u_0,u_1$ respectively. Then the sequence $P(u_0^k, u_1^k) \in \cH_{\triangle}$ decreases to $P(u_0,u_1)$ and the Pythagorean formula follows from Lemma \ref{approximation4}.
\end{proof}

\begin{lemma}\label{distancef2}
Let $u_t$ be the weak geodesic connecting $u_0,u_1 \in \cH_{\triangle}$.Then the following holds:
\[
d_p(u_0,u_1)^p=\int_M |\dot{u}_0|^p\omega_{u_0}^p\wedge\eta=\int_M|\dot{u}_1|^p\omega_{u_1}^n\wedge\eta
\]
\end{lemma}
\begin{proof}
$v_t=u_{1-t}$ is the weak geodesic connecting $u_1,u_0$. By Lemma \ref{rf} we have
\begin{equation*}
\begin{split}
\{P(u_0+s,u_1)<u_1\} &=M-\{P(u_0+s,u_1)=u_1\} \\
                                   &=M-\{\dot{v}_0 \geq -s\} \\
                                   &=\{\dot{u}_1 >s\}
\end{split}
\end{equation*}
Recall that $\omega_{u_1}^n\wedge\eta$ has total finite measure $\text{Vol}(M)$, hence except for a countably many $s \in \mathbb{R}$ we have $\omega_{u_1}^n\wedge\eta(\{u_0=u_1-s\})=0$ and $\omega_{u_1}^n\wedge\eta(\{\dot{u}_1 \geq s\})=\omega_{u_1}^n\wedge\eta(\{\dot{u}_1>s\})$.  For such real number $s$, it follows from Lemma \ref{decomposition} that
\begin{equation*}
\omega_{P(u_0,u_1-s)}^n\wedge\eta=\chi_{\{P(u_0,u_1-s)=u_0\}} \omega_{u_0}^n\wedge\eta+\chi_{\{P(u_0,u_1-s)=u_1-s\}} \omega_{u_1}^n\wedge\eta
\end{equation*}
and
\begin{equation*}
\text{Vol}(M)=\omega_{u_0}^n\wedge\eta(\{P(u_0,u_1-s)=u_0\})+\omega_{u_1}^n\wedge\eta(\{P(u_0,u_1-s)=u_1-s\})
\end{equation*}
It follows from Lemma \ref{rf}, the definition of rooftop envelope that
\begin{equation*}
\begin{split}
\int_{\{\dot{u}_0>0\}} |\dot{u}_0|^p\omega_{u_0}^n\wedge\eta &= p\int_0^{\infty} s^{p-1} \omega_{u_0}^n\wedge\eta(\{\dot{u}_0 \geq s\})ds \\
                                                                                                  &= p\int_0^{\infty} s^{p-1} \omega_{u_0}^n\wedge\eta(\{P(u_0,u_1-s)=u_0\})ds \\
                                                                                                  &=p\int_0^{\infty} s^{p-1}(\text{Vol}(M)-\omega_{u_1}^n\wedge\eta(\{P(u_0,u_1-s)=u_1-s\}))ds \\
                                                                                                  &=p\int_0^{\infty} s^{p-1}\omega_{u_1}^n\wedge\eta(\{P(u_0,u_1-s) <u_1-s\})ds \\
                                                                                                  &=p\int_0^{\infty} s^{p-1}\omega_{u_1}^n\wedge\eta(\{P(u_0+s,u_1) < u_1\})ds \\
                                                                                                  &=p\int_0^{\infty} s^{p-1}\omega_{u_1}^n\wedge\eta(\{\dot{u}_1> s\}) ds \\
                                                                                                  &= p\int_0^{\infty} s^{p-1}\omega_{u_1}^n\wedge\eta(\{\dot{u}_1 \geq s\}) ds \\
                                                                                                  &=\int_{\{\dot{u}_1>0\}} |\dot{u}_1|^p\omega_{u_1}^n\wedge\eta
\end{split}
\end{equation*}
A similar arguments gives that
\begin{equation*}
\int_{\{\dot{u}_0<0\}} |\dot{u}_0|^p\omega_{u_0}^n\wedge\eta=\int_{\{\dot{u}_1<0\}} |\dot{u}_1|^p\omega_{u_1}^n\wedge\eta
\end{equation*}
It follows that
\begin{equation*}
\int_M |\dot{u}_0|^p\omega_{u_0}^n\wedge\eta=\int_M |\dot{u}_1|^p\omega_{u_1}^n\wedge\eta
\end{equation*}
Now choose sequence $\{u_0^k\}_{k\in\mathbb{N}}, \{u_1^k\}_{k\in\mathbb{N}} \subset \cH$ decreasing to $u_0,u_1$ respectively. Let $u_t^{kl},u_t$ be the $C_B^{1,\bar1}$ geodesic connecting $u_0^k, u_1^l$ and $u_0,u_1$ respectively. Let $u_t^k$ be the $C_B^{1,\bar1}$ geodesic connecting $u_0^k, u_1$.
It follows from Lemma \ref{approximation4}, Lemma \ref{tangentapp04} and the above results that
\begin{equation*}
d_p(u_0^k,u_1)^p=\lim_{l\rightarrow\infty}d_p(u_0^k,u_1^l)^p=\lim_{l\rightarrow\infty}\int_M |\dot{u}_0^{kl}|^p\omega_{u_0^k}^n\wedge\eta=\int_M|\dot{u}_0^k|^p\omega_{u_0^k}^n\wedge\eta=\int_M |\dot{u}_1^k|^p\omega_{u_1}^n\wedge\eta
\end{equation*}
Then use Lemma \ref{approximation4} ,Proposition \ref{homweakgeodesic} and Lemma \ref{tangentapp04}, we have
\[
d_p(u_0,u_1)^p=\lim_{k\rightarrow\infty}d_p(u_0^k,u_1)^p=\lim_{k\rightarrow\infty}\int_M|\dot{u}_1^k|^p\omega_{u_1}^n\wedge\eta=\int_M|\dot{u}_1|^p\omega_{u_1}^n\wedge\eta
\]
This completes the proof.
\end{proof}

\begin{lemma}\label{order4}
Assume that $u, v\in \cE_p(M,\xi,\omega^T)$ with $u\leq v$.Then we have
\begin{equation*}
\max(\frac{1}{2^{n+p}}\int_M |v-u|^p\omega_u^n\wedge\eta,\int_M|u-v|^p\omega_v^n\wedge\eta) \leq d_p(u, v)^p \leq \int_M |v-u|^p \omega_u^p\wedge\eta
\end{equation*}
\end{lemma}
\begin{proof}
First we can choose $u_k, w_k \in \cH$ strictly decreasing to $u, v$ respectively. Then $\max(u_k, w_k) \in \text{PSH}(M,\xi,\omega^T)$ are continuous and strictly decreases to $v$. By Dini's Lemma there exists $v_k \in \cH$ such that $\max(u_{k-1},v_{k-1}) \geq v_k \geq \max(u_k, v_k)$. Then $v_k$ decreases to $v$ and $u_k \leq v_k$. It follows from Lemma \ref{dis100} that
\begin{equation*}
\max(\frac{1}{2^{n+p}}\int_M |v_k-u_k|^p\omega_{u_k}^n\wedge\eta,\int_M|u_k-v_k|^p\omega_{v_k}^n\wedge\eta) \leq d_p(u_k, v_k)^p \leq \int_M |v_k-u_k|^p \omega_{u_k}^p\wedge\eta
\end{equation*}
By the Proposition \ref{weakcon3} the required inequality follows as $k \rightarrow \infty$.
\end{proof}

\begin{lemma}\label{approximation4}
If the sequence $\{u_k\}_{k\in \mathbb{N}} , \{v_k\}_{k \in \mathbb{N}}\in \cE_p(M,\xi,\omega^T)$ decreases (increases) to $u ,v\in \cE_p(M,\xi,\omega^T)$ respectively, then $d_p(u_k, v_k) \rightarrow d_p(u, v)$ as $k \rightarrow \infty$. In particular, $d_p(u_k,u) \rightarrow 0$.
\end{lemma}
\begin{proof}
If the sequence $\{u_k\}_{k\in \mathbb{N}}$ is decreasing, using the triangle inequality and Lemma \ref{order4} we have
\begin{equation*}
\begin{split}
|d_p(u_k, v_k)-d_p(u, v)| &\leq d_p(u_k, u)+d_p(v, v_k) \\
                                     &\leq  (\int_M|u_k-u|^p\omega_u^n\wedge\eta)^{\frac{1}{p}}+(\int_M|v_k-v|^p\omega_v^n\wedge\eta)^{\frac{1}{p}}
\end{split}
\end{equation*}
and the Lemma follows from Lemma \ref{weakcon3}.

If the sequence $\{u_k\}_{k\in\mathbb{N}}$ is increasing, using the triangle inequality and Lemma \ref{order4} we have
\begin{equation*}
\begin{split}
|d_p(u_k, v_k)-d_p(u, v)| &\leq d_p(u_k,u)+d_p(v, v_k) \\
                                     &\leq  (\int_M|u_k-u|^p\omega_{u_k}^n\wedge\eta)^{\frac{1}{p}}+(\int_M|v_k-v|^p\omega_{v_k}^n\wedge\eta)^{\frac{1}{p}}
\end{split}
\end{equation*}
and the Lemma follows from Lemma \ref{weakcon3}.
\end{proof}

\begin{lemma}Suppose $u_0, u_1\in \cE_p(M, \xi, \omega^T)$. Then we have
\[
d_p\left(u_0, \frac{u_0+u_1}{2}\right)^p\leq Cd_p(u_0, u_1)^p
\]
\end{lemma}
\begin{proof}
It is obvious that $ P(u_0,u_1) \leq P(u_0,\frac{u_0+u_1}{2}) \leq u_0$ and $P(u_0,u_1) \leq P(u_0,\frac{u_0+u_1}{2}) \leq \frac{u_0+u_1}{2}$. By The Pythagorean Theorem \ref{Pythagorean} ,Lemma \ref{order1} and Lemma \ref{order4} we have
\begin{equation*}
\begin{split}
d_p(u_0,\frac{u_0+u_1}{2})^p &=d_p(u_0,P(u_0,\frac{u_0+u_1}{2}))^p+d_p(\frac{u_0+u_1}{2},P(u_0,\frac{u_0+u_1}{2}))^p \\
                                                &\leq d_p(u_0,P(u_0,u_1))^p+d_p(\frac{u_0+u_1}{2},P(u_0,u_1))^p  \\
                                                &\leq \int_M |u_0-P(u_0,u_1)|^p\omega_{P(u_0,u_1)}^n\wedge\eta+\int_M|\frac{u_0+u_1}{2}-P(u_0,u_1)|^p\omega_{P(u_0,u_1)}^n\wedge\eta \\
                                                &\leq 2(\int_M|u_0-P(u_0,u_1)|\omega_{P(u_0,u_1)}^n\wedge\eta+\int_M|u_1-P(u_0,u_1)|\omega_{P(u_0,u_1)}^n\wedge\eta) \\
                                                &\leq 2^{n+p+1}(d_p(u_0,P(u_0,u_1))^p+d_p(u_1,P(u_0,u_1))^p) \\
                                                &=2^{n+p+1}d_p(u_0,u_1)^p
\end{split}
\end{equation*}
This completes the proof.
\end{proof}

\begin{thm}\label{comparison04}For any $u_0, u_1\in \cE_p(M, \xi, \omega^T)$ we have
\begin{equation}
C^{-1}d_p(u_0, u_1)^p\leq \int_M |u_0-u_1|^p(\omega^n_{u_0}\wedge \eta+\omega^n_{u_1}\wedge \eta)\leq Cd_p(u_0, u_1)^p.
\end{equation}
\end{thm}
\begin{proof}
Using the triangle inequality, arithmetic-geometric mean inequality, and Lemma \ref{order4} we have:
\begin{equation*}
\begin{split}
d_p(u_0,u_1)^p & \leq (d_p(u_0,\max(u_0,u_1))+d_p(u_1,\max(u_0,u_1)))^p   \\
                          &\leq   2^{p-1}(d_p(u_0,\max(u_0,u_1))^p+d_p(u_1,\max(u_0,u_1))^p) \\
                          & \leq  2^{p-1}(\int_M |u_0-\max(u_0,u_1)|^p\omega_{u_0}^n\wedge\eta+\int_M|u_1-\max(u_0,u_1)|^p\omega_{u_1}^n\wedge\eta) \\
                          & = 2^{p-1}(\int_{\{u_0 <u_1\}} |u_0-u_1|^p\omega_{u_0}^n\wedge\eta+\int_{\{u_1<u_0\}} |u_1-u_0|^p\omega_{u_1}^n\wedge\eta)  \\
                          &\leq 2^{p-1}\int_M |u_0-u_1|^p(\omega_{u_0}^n\wedge\eta+\omega_{u_1}^n\wedge\eta)
\end{split}
\end{equation*}

By the previous Lemma , the Pythagorean formula and Lemma \ref{order4}, there exists a constant $C$ such that
\begin{equation*}
\begin{split}
Cd_p(u_0,u_1)^p & \geq d_p(u_0,\frac{u_0+u_1}{2})^p \\
                             & \geq  d_p(u_0,P(u_0,\frac{u_0+u_1}{2}))^p \\
                             & \geq   \int_M|u_0-P(u_0,\frac{u_0+u_1}{2})|\omega_{u_0}^n\wedge\eta
\end{split}
\end{equation*}
Similarly we also have:
\begin{equation*}
\begin{split}
Cd_p(u_0,u_1)^p &\geq d_p(u_0,\frac{u_0+u_1}{2})^p \\
                            & \geq d_p(\frac{u_0+u_1}{2}, P(u_0,\frac{u_0+u_1}{2}))^p \\
                            & \geq \int_M |\frac{u_0+u_1}{2}-P(u_0,\frac{u_0+u_1}{2})|^p\omega_{\frac{u_0+u_1}{2}}^n\wedge\eta \\
                            & \geq  \frac{1}{2^n} \int_M |\frac{u_0+u_1}{2}-P(u_0,\frac{u_0+u_1}{2})|^p\omega_{u_0}^n\wedge\eta
\end{split}
\end{equation*}
Hence by the Holder inequality we have:
\begin{equation*}
\begin{split}
(2^n+1)Cd_p(u_0,u_1)^p &\geq  \int_M(|u_0-P(u_0,\frac{u_0+u_1}{2})|^p+|\frac{u_0+u_1}{2}-P(u_0,\frac{u_0+u_1}{2})|^p)\omega^n_{u_0}\wedge\eta \\
                                         & \geq \frac{1}{2^p} \int_M|u_0-u_1|^p\omega_{u_0}^n\wedge\eta
\end{split}
\end{equation*}
By symmetry of $u_0,u_1$ we also have:
\begin{equation*}
(2^n+1)Cd_p(u_0,u_1)^p \geq \frac{1}{2^p} \int_M |u_0-u_1|\omega_{u_1}^n\wedge\eta
\end{equation*}
Adding the last two inequalities we obtain:
\begin{equation*}
2^p(2^n+1)C d_p(u_0,u_1)^p \geq \int_M |u_0-u_1|^p(\omega_{u_0}^n\wedge\eta+\omega_{u_1}^p\wedge\eta)
\end{equation*}
This completes the proof.
\end{proof}

\begin{lemma}\label{close4}
Let $\{u_k\}_{k\in\mathbb{N}} \subset \cE_p(M,\xi,\omega^T)$ be a $d_p$-bounded sequence decreasing (increasing) to $u$. Then $u \in \cE(M,\xi,\omega^T)$ and $d_p(u_k,u)\rightarrow 0$.
\end{lemma}
\begin{proof}
If $\{u_k\}_{k\in\mathbb{N}}$ is decreasing, we can assume that $u_k <0$. It follows from Lemma \ref{order4} that
\[
\max(\frac{1}{2^{n+p}}\int_M|u_k|^p\omega_{u_k}^n\wedge\eta, \int_M |u_k|^p(\omega^T)^n\wedge\eta) \leq d_p(u_k,0)^p
\]
are uniformly bounded. $\int_M |u_k|^p(\omega^T)^n\wedge\eta$ is uniformly bounded, the monotone convergence theorem and the dominated convergence theorem imply that $u_k \rightarrow u $ in $L_{loc}^1$ and $u \in \text{PSH}(M, \xi, \omega^T)$. $E_p(u_k)= \int_M |u_k|^p\omega_{u_k}^n\wedge\eta$ is uniformly bounded, it follows from Proposition \ref{boundedenergy} and Lemma \ref{approximation4} that $u \in \cE_p(M, \xi, \omega^T)$ and $d_p(u_k, u) \rightarrow 0$.

If $\{u_k\}_{k\in \mathbb{N}}$ is increasing, it follows from Theorem \ref{comparison04} that there exists a constant $C$ such that
\[
\int_M|u_k|^p(\omega_{u_k}^n\wedge\eta+(\omega^T)^n\wedge\eta) \leq Cd_p(u_k, 0)
\]
is uniformly bounded. By Proposition  \ref{compactness001} we have $u_k \rightarrow u$ in $L^1$ for some $u \in \text{PSH}(M,\xi,\omega^T)$.  By Proposition \ref{boundedenergy} and Lemma \ref{approximation4}  we have $u \in \cE_p(M, \xi, \omega^T)$ and $d_p(u_k, u) \rightarrow 0$.
\end{proof}

\begin{prop}\label{rooftop04}Given $u_0, u_1, v\in \cE_p(M, \xi, \omega^T)$, 
\[
d_p(P(u_0, v), P(u_1, v))\leq d_p(u_0, u_1)
\]
\end{prop}
\begin{proof}
By Theorem \ref{rooftop101} and Lemma \ref{approximation4} we only have to prove the inequality for $u_0, u_1, v \in \cH_{\triangle}$. In this case $P(u_0, v), P(u_1, v) \in \cH_{\triangle}$ according to Theorem \ref{rooftop101}.

First we assume that $u_0 \leq u_1$. Let $u_t, v_t$ be the $C_B^{1,\bar1}$ geodesic connecting $u_0,u_1$ and $P(u_0,v), P(u_1, v)$ respectively. Then $P(u_0, v) \leq P(u_1,v) \leq v$  and the strong maximum principle implies that  $P(u_0, v) \leq v_t \leq v$. Hence for $x \in \{P(u_0,v)=v\}$, $v_t(x)$ is independent of $t$ and $\dot{v}_0(x)=0$. Then we have
\[
\int_{\{P(u_0, v)=v\}} |\dot{v}_0|^p\omega_v^n\wedge\eta=0.
\]
$P(u_0,v) \leq P(u_1,v), P(u_0,v)\leq u_0, P(u_1, v) \leq u_1$ and the strong maximum principle implies that $P(u_0, v) \leq v_t \leq u_t$ for $t \in [0,1]$ and $\dot{v}_0 \geq 0$. Moreover for $x \in \{P(u_0,v)=u_0\}$ we have
\[
\dot{v}_0(x) =\lim_{t\rightarrow 0+} \frac{v_t(x)-v_0(x)}{t} \leq \lim_{t\rightarrow 0+}\frac{u_t(x)-u_0(x)}{t}=\dot{u}_0(x).
\]
Then it follows from Lemma \ref{distancef2}, Lemma \ref{decomposition} that
\begin{equation*}
\begin{split}
d_p(P(u_0,v), P(u_1, v))^p &= \int_M |\dot{v}_0|\omega_{P(u_0,v)}^n\wedge\eta \\
                                           &\leq \int_{\{P(u_0,v)=u_0\}}|\dot{v}_0|^p\omega_{u_0}^n\wedge\eta+\int_{\{P(u_0, v)=v\}} |\dot{v}_0|^p\omega_v^n\wedge\eta \\
                                           &\leq \int_{\{P(u_0,v)=u_0\}} |\dot{u}_0|^p\omega_{u_0}^n\wedge\eta  \\
                                           &\leq \int_M |\dot{u}_0|^p\omega_{u_0}^n\wedge\eta \\
                                           &=d_p(u_0, u_1)^p.
\end{split}
\end{equation*}
For the general case, using the Pythagoreans formula we have
\begin{equation*}
\begin{split}
d_p(P(u_0,v), P(u_1, v))^p &=d_p(P(u_0,v),P(u_0,u_1,v))^p +d_p(P(u_1, v), P(u_0, u_1, v))^p \\
                                           &=d_p(P(u_0,v), P(P(u_0,u_1), v))^p+ d_p(P(u_1,v), P(P(u_0,u_1), v))^p \\
                                           &\leq d_p(u_0, P(u_0,u_1))^p+ d_p(u_1, P(u_0,u_1))^p    \\
                                           &=d_p(u_0, u_1)^p.
\end{split}
\end{equation*}
This completes the proof.
\end{proof}

\begin{prop}$(\cE_p(M, \xi, \omega^T), d_p)$ is a complete metric space. 
\end{prop}
\begin{proof}
First we show that $(\cE_p(M,\xi,\omega^T), d_p))$ is a metric space. The symmetry of $d_p$ is obvious and the triangle inequality follows from Lemma \ref{triangle200}. We only have to check the non-degeneracy of $d_p$. Suppose $w_1,w_2 \in \cE(M,\xi,\omega^T)$ and $d_p(w_1,w_2)=0$. It follows from the Pythagorean formula that $d_p(w_1,P(w_1,w_2))=0$ and $d_p(P(w_1,w_2),w_2)=0$. Then Lemma \ref{order4} implies that  $w_1=P(w_1,w_2)=w_2$ with respect to the measure $\omega_{P(w_1,w_2)}^n\wedge\eta$. Then the domination principle Lemma \ref{domination2} implies that $w_1=P(w_1,w_2)=w_2$. Hence $(\cE_p(M, \xi, \omega^T),d_p)$ is a metric space.

Then we show that the metric space $(\cE_p(M, \xi, \omega^T), d_p)$ is complete. Suppose $\{u_k\}_{k\in\mathbb{N}} \subset \cE_p(M, \xi, \omega^T)$ is a $d_p$ Cauchy sequence. We will prove that there exists $u \in \cE_p(M, \xi, \omega^T)$ such that $d_p(u_k, u) \rightarrow 0$.

Without loss of generality we can assume that
\begin{equation*}
d_p(u_k,u_{k+1}) \leq \frac{1}{2^k}
\end{equation*}
for $k \in \mathbb{N}$. Denote by $u_k^l=P(u_k, u_{k+1},..., u_{k+l})$ for $k, l \in \mathbb{N}$ and $u_k^0=u_k$. It follows from the definition of rooftop envelope and  Proposition \ref{rooftop04}  that
\begin{equation*}
d_p(u_k^l,u_k^{l+1})=d_p(P(u_k^l, u_{k+l}), P(u_k^l,u_{k+l+1})) \leq d_p(u_{k+l}, u_{k+l+1}) \leq \frac{1}{2^{k+l}}
\end{equation*}
and the sequence $\{u_k^l\}_{l \in \mathbb{N}} \subset \cE_p(M, \xi, \omega^T)$ is $d_p$ bounded and decreasing. According to Lemma \ref{close4} $\tilde{u}_k =\lim\limits_{l\rightarrow\infty} u_k^l \in \cE_p(M, \xi, \omega^T)$ and $d_p(u_k^l, \tilde{u}_k) \rightarrow 0$ as $l \rightarrow \infty$. Moreover $u_k^{l+1} \leq u_{k+1}^l$ implies that $\tilde{u}_k \leq \tilde{u}_{k+1}$ and $\{\tilde{u}_k\}_{k\in\mathbb{N}}$ is a increasing sequence in $\cE_p(M, \xi, \omega^T)$.

It follows from Lemma \ref{approximation4}, the definition of rooftop envelope and Proposition \ref{rooftop04} that
\begin{equation*}
\begin{split}
d_p(\tilde{u}_k, \tilde{u}_{k+1}) &=\lim_{l\rightarrow \infty} d_p(u_k^{l+1}, u_{k+1}^l) \\
                                                 &=\lim_{l\rightarrow\infty}   d_p(P(u_{k+1}^l, u_k), P(u_{k+1}^l, u_{k+1})) \\
                                                 &\leq \lim_{l\rightarrow\infty} d_p(u_k, u_{k+1}) \\
                                                 &\leq \frac{1}{2^k}
\end{split}
\end{equation*}
and the sequence $\{\tilde{u}_k\}_{k\in\mathbb{N}} \subset \cE_p(M, \xi, \omega^T)$ is $d_p$-bounded and increasing. By Lemma \ref{close4} $u=\lim\limits_{k \rightarrow\infty} \tilde{u}_k \in \cE_p(M, \xi, \omega^T)$ and $\lim\limits_{k\rightarrow\infty}d_p(\tilde{u}_k, u)=0$.
Moreover by Proposition \ref{rooftop04} we have
\begin{equation*}
\begin{split}
d_p(u_k^l, u_k) =d_p(P(u_k,u_{k+1}^{l-1}),P(u_k,u_k)) \leq d_p(u_{k+1}^{l-1}, u_k) \leq d_p(u_{k+1}^{l-1},u_{k+1}) +d_p(u_k,u_{k+1})
\end{split}
\end{equation*}
and
\begin{equation*}
d_p(u_k^l,u_k) \leq d_p(u_{k+l}^0,u_{k+l})+\sum_{j=1}^{l}d_p(u_{k+j-1},u_{k+j})=\sum_{j=1}^ld_p(u_{k+j-1},u_{k+j})
\end{equation*}
It follows from Lemma \ref{approximation4} that 
\begin{equation*}
d_p(\tilde{u}_k,u_k) \leq \sum_{j=1}^{\infty}\frac{1}{2^{k+j-1}}=\frac{1}{2^{k-1}}
\end{equation*}
By the triangle inequality
\begin{equation*}
d_p(u_k,u) \leq d_p(\tilde{u}_k,u_k)+d_p(\tilde{u}_k,u)
\end{equation*}
we have $d_p(u_k ,u) \rightarrow 0$. This completes the proof.
\end{proof}

\section{Sasaki-extremal metric}

We give a brief discussion of existence of Sasaki-extremal metric and properness of modified $\cK$-energy. Calabi's extremal metric was extended to Sasaki setting by Boyer-Galicki-Simanca \cite{BGS1}. A  Sasaki metric is called Sasaki-extremal if its transverse K\"ahler metric is extremal in the sense of Calabi \cite{Ca1}. As in K\"ahler setting, given a priori estimates \cite{he182} and the pluripotential theory developed in the paper, we have the following,

\begin{thm}\label{extremal}A compact Sasaki manifold $(M, \xi, \eta, g)$ admits a Sasaki-extremal metric in the transverse K\"ahler class $[\omega^T]$ if and only if the modified $\cK$-energy is reduced proper. 
\end{thm}

We recall some basic notions \cite{Fut, M2, Ca1, FM, BGS1}. 
We use the group $\text{Aut}_0(\xi, J)$ to denote the subgroup of diffeomorphism group of $M$ which preserves both $\xi$ and transverse holomorphic structure. Its Lie algebra is the Lie algebra of all \emph{Hamiltonian holomorphic vector fields} in the sense of \cite{FOW}[Definition 4.4].

First one can define Sasaki-Futaki invariant as follows, given $X\in \mathfrak{aut}$, the Lie algebra of $\text{Aut}_0(\xi, J)$, 
\begin{equation}\label{fut01}
\cF_X(\omega^T)=\int_M X(f) \omega_T^n\wedge \eta, 
\end{equation}
where $f$ is the potential of transverse scalar curvature, 
\[
\Delta f=R^T-\underline {R}. 
\]
The first step is certainly to verify that \eqref{fut01} does not depend on a particular choice of transverse K\"ahler form in $[\omega^T]$ (see \cite{BGS1}[Proposition 5.1]).  We are interested in the reduced part $\mathfrak{h}_0$ of $\mathfrak{aut}$, which consists of \emph{Hamiltonian holomorphic vector fields} such that $\eta(Y)$ has non empty zero. When $(M, \xi, \eta, g)$ is a Sasaki-extremal metric, then similar as in Calabi's decomposition, we have \cite{BGS1}[Theorem 4.8] the decomposition
\[
\mathfrak{h}=\mathfrak{a}\oplus \mathfrak{h}_0,
\]
where $\mathfrak{a}$ consists of parallel vector fields of the transverse K\"ahler metric $g^T$. Moreover the reduced part $\mathfrak{h}_0$ has the decomposition
\[
\mathfrak{h}_0=\mathfrak{z}_0\oplus J\mathfrak{z}_0\oplus(\oplus_{\l>0}\mathfrak{h}^\l),
\]
where $\mathfrak{z}_0=\text{aut}(\xi, \eta, g)/\{\xi\}$ and 
\[
\mathfrak{h}^\l=\{Y\in \mathfrak{h}: \cL_{X} Y=\l Y, X=(\bar\p R)^{\#},\}
\] 
where $X:=(\bar\p R)^{\#}$ is the dual vector and it is the extremal vector field in $\mathfrak{h}_0$. In general, we can define Futaki-Mabuchi bilinear form \cite{FM} on $\mathfrak{h}_0$ as in K\"ahler setting (in Sasaki setting this is well-defined on $\mathfrak{aut}$ since every Hamiltonian vector field has a potential, simply given by $\eta(Y)$; for example, $\xi$ has potential $1$). Given $Y, Z\in \mathfrak{aut}$, define
\begin{equation}\label{fm01}
B(Y, Z)=\int_M \eta(Y) \eta(Z) (\omega^T)^n\wedge \eta. 
\end{equation}
It is straightforward to check that \eqref{fm01} remains unchanged if $\eta\rightarrow \eta+d^c_B\phi$ for $\phi\in \cH$. If we restrict us on the \emph{real Hamiltonian holomorphic vector fields} such that $\eta(Y)$ is real, then there exists a unique vector field $V$ such that
\begin{equation}\label{fut03}
\cF_{\text{Re}(Y)}=B(\text{Re}(Y), V)
\end{equation}
We call such $V$ and its corresponding $X=V-\sqrt{-1}JV$ \emph{the extremal vector field.} As in K\"ahler setting, for $JV$-invariant metrics in $\cH$, we define the modified $\cK$-energy \cite{Guan, Simanca} as
\begin{equation}
\delta \cK_V=-\int_M \delta \phi (R_\phi-\underline{R}-\eta_\phi(V)) \omega^n_\phi\wedge \eta. 
\end{equation}
Let
$\text{Aut}_0(\xi, J, V)$  be the subgroup of $\text{Aut}_0(\xi, J)$ which commutes with the flow of $JV$. 
\begin{prop}The $\cK_V$ energy is invariant under the action of $\text{Aut}_0(\xi, J, V)$
\end{prop}

\begin{proof}The proof is similar to K\"ahler setting \cite{he18}[Lemma 2.1] and it follows in a tautologic way from Futaki-invariant and definition of extremal vector field through Futaki-Mabuchi bilinear form. We fix a background transverse K\"ahler structure $\omega^T$ such that it is $JV$ invariant. 
For $\sigma\in \text{Aut}_0(\xi, J, V)$, let $\sigma_t$ be one parameter subgroup generated by the flow of $Y_\R:=\text{Re}(Y)$ for some $Y\in \mathfrak{aut}$. Since $Y$ commutes with $V$, hence $\sigma_t^* \omega_0$ is invariant with respect to $JV$ if $\omega_0\in [\omega^T]$ is invariant. We compute
\[
\begin{split}
\frac{d}{dt}\cK(\sigma_t^*\omega_0)=&-\int_M\sigma_t^*(\eta_0(\text{Re}(Y)) (R_0-\underline R-\eta_0(V))\omega_0^n\wedge \eta_0)\\
=&-\int_M \eta_0(Y_\R) (R_0-\underline{R})\omega_0^n\wedge \eta_0+\int_M \eta_0(Y_\R) \eta_0 (V)\omega_0^n\wedge \eta_0
\end{split}
\]
The righthand side is zero by \eqref{fut03}. 
\end{proof}
We  define the distance $d_1$ modulo the group action $G_0:=\text{Aut}_0(\xi, J, V)$. Fix a compact subgroup $K$ of $G_0$ such that $K$ contains the flow of $JV$ (and $\xi$ of course). 
Denote \[\cH^K_0=\{\phi\in \cH_0, \phi\; \text{is invariant under the flow of}\; K\}\]
Note that $G_0$ acts on $\cH_0$ through $\omega_\phi\rightarrow \sigma^*\omega_\phi=\omega^T+\sqrt{-1}\p_B\bar\p_B \sigma[\phi]$. Given any $\phi, \psi\in \cH_0$, we can consider the distance modulo $G_0$ as follows \cite{CPZ}
\[
d_{1, G_0}(\phi, \psi)=\inf_{\sigma_1, \sigma_2\in G_0} d_1(\sigma_1[\phi], \sigma_2[\psi])=\inf_{\sigma\in G_0}d_1(\phi, \sigma[\psi]). 
\]
\begin{defn}We say $\cK_V$ is reduced proper for $K$-invariant metrics  with respect to  $d_{1, G_0}$ if the following conditions hold
\begin{enumerate}
\item $\cK_V$ is bounded below over $\cH^K$. 
\item There exists constant $C, D>0$ such that for $\phi\in \cH^K$
\[
\cK_V(\phi)\geq C d_{1, G_0}(0, \phi)-D.
\]
\end{enumerate}
\end{defn}

To prove Theorem \ref{extremal}, we proceed exactly as in \cite{he18}, to consider the modified Chen's continuity path \cite{chen15}, for a $K$-invariant transverse K\"ahler metric $\omega^T$,
\begin{equation}\label{se001}
t(R_\phi-\underline{R}-\eta_\phi(V))+(1-t)(\Lambda_{\omega_\phi}\omega^T-n)=0
\end{equation}
Given a priori  estimates as in \cite{he182} and the pluripotential theory on Sasaki manifolds developed in this paper, we can then follow \cite{he18, he182} to prove Theorem \ref{extremal}. Since the argument is almost identical, we only sketch the process and skip the details.

\begin{enumerate}
\item The openness of \eqref{se001} is proved similarly \cite{he18}[Theorem 3.4]; note that we assume transverse K\"ahler metrics and potentials are $K$-invariant.

\item For $0<t<1$, $\cK_V$ bounded below over $\cH^K$ implies that the distance $d(0, \phi_t)$ is uniformly bounded by a constant in the order $C((1-t)^{-1}+1)$, where $\phi_t$ is the solution of \eqref{se001} at $t$. This together with the fact that $\phi_t$ minimizes $t\cK_V+(1-t)\mathbb{J}$, gives the uniform upper bound of entropy of $H(\phi_t)$ (depending on $(1-t)^{-1}$). Hence estimates 
in \cite{he182}[Theorem 2] applies to get the solution for any $t<1$.

\item Choose an increasing sequence $t_i\rightarrow 1$, first using the properness assumption we can assume that there are $\sigma_i\in G$ such that $\psi_i:=\sigma_i[\phi_{t_i}]$ ($\omega_{\psi_i}=\sigma_i^{*}\omega_{\phi_{t_i}}$) satisfies that $d(0, \psi_i)$ is uniformly bounded above. Then $\psi_i$ satisfies a scalar curvature type equation 
\[
\begin{split}
&\omega_{\psi_i}^n=e^{F_i}(\omega^T)^n\\
&\Delta_{\psi_i} F_i=h_i+\text{tr}_{\psi_i}(Ric(\omega^T)-\frac{1-t_i}{t_i}\omega_i)
\end{split}
\]
where $h_i$ is uniformly bounded and $\omega_i=\sigma_i^{*}(\omega^T)$. One can use \cite{he182}[Theorem 3] and arguments as in \cite{he18}[Theorem 3.5] to conclude the convergence of $\psi_i, F_i$ to a smooth Sasaki-extremal structure. 
\end{enumerate}

\section{Appendix}
\subsection{Approximation through Type-I deformation and Regularity of rooftop envelop}
Using Type-I deformation, we can obtain
the following approximation of irregular Sasaki structure $(M, \xi, \eta, g)$, which would be important for us; see \cite{Ruk}  and in particular \cite{BG}[Theorem 7.1.10] for the approximation. 
Suppose $\xi$ is irregular, then the Reeb flow generates an isometry in $\text{Aut}(M, \xi, \eta, g)$. Let $T^k\subset \text{Aut}(M, \xi, \eta, g)$ ($k\geq 2$) be the torus generated by $\xi$ and denote $\mathfrak{t}$ to be its Lie algebra. We can then choose $\rho_i\rightarrow 0, \rho_i\in \mathfrak{t}$ such that $\xi_i=\xi+\rho_i$ is quasiregular. Define
\begin{equation}\label{approx}
\eta_i=\frac{\eta}{1+\eta(\rho_i)}, \Phi_i=\Phi-\frac{1}{1+\eta(\rho_i)}\Phi\rho_i\otimes \eta, \omega_i^T=\frac{1}{2}d\eta_i, g_i=\eta_i\otimes \eta_i+\omega_i^T(\mathbb{I}\otimes \Phi_i),
\end{equation}
where $\Phi$ is the $(1, 1)$ tensor field defined on the contact bundle $\cD=\text{Ker}(\eta)$.  
We recall the following,

\begin{thm}[Approximation of irregular Sasaki structure]\label{type101}Let $(M, \xi, \eta, g)$ be an irregular Sasaki structure on a compact manifold $M$. Then we can choose $\rho_i\rightarrow 0$ such that $\xi_i$ is quasiregular and \eqref{approx} defines a quasi-regular Sasaki structure which is invariant under the action of $T^k$, the torus generated by $\xi$ in $\text{Aut}(M, \xi, \eta, g)$.
\end{thm}

\begin{lemma}\label{type1}Let $(M, \xi, \eta, g)$ be a Sasaki structure on a compact manifold $M$. Consider a torus $T\subset \text{Aut}(M, \xi, \eta, g)$ and $\xi_i\in \mathfrak{t}$. 
Choose $\xi_i=\xi+\rho_i$ for $\rho_i$ sufficiently small. Consider two Sasaki structures $(\xi, \eta, \Phi, g)\leftrightarrow (\xi_i, \eta_i, \Phi_i, g_j)$ via Type-I deformation. Then we have the following.
Suppose $u$ is $T$ invariant and $u\in \text{PSH}(M, \xi, \omega^T)$ with $|d \Phi du|\leq C_0$. Then for $\rho_i$ sufficiently small, there exists positive constant $\epsilon_i\rightarrow 0$ (as $\rho_i\rightarrow 0$) such that,
\begin{equation}\label{approx100}
(1-\epsilon_i) u\in \text{PSH}(M, \xi_i, \omega_i^T)
\end{equation}
Similarly, suppose $|d \Phi du|\leq C_0$ and $u\in \text{PSH}(M, \xi_i, \omega_i^T)$, then there exists positive constant $\epsilon_i \rightarrow 0$ as $i\rightarrow \infty$, such that
\begin{equation}\label{approx101}
(1-\epsilon_i) u\in \text{PSH}(M, \xi, \omega^T)
\end{equation}
\end{lemma}

\begin{proof}Since $u$ is $T^k$-invariant, hence $u$ is a basic function with respect to both $\xi$ and $\xi_i$. 
We write
\[
\omega_i^T+\sqrt{-1}\p^i_B\bar\p^i_B u
=\omega^T_i+\frac{1}{2} d\Phi_i d u.\]
Using \eqref{approx},  we compute
\begin{equation}\label{approx1005}
\begin{split}
\omega^T_i+\frac{1}{2} d\Phi_i d u=&\frac{\omega^T}{1+\eta(\rho_i)}+\eta\wedge d\left(\frac{1-du(\Phi \rho_i)}{1+\eta(\rho_i)}\right)+\frac{1}{2}d\Phi du+2\omega^T \frac{du(\Phi \rho_i)}{1+\eta(\rho_i)}\\
=&\frac{1+2du(\Phi \rho_i)}{1+\eta(\rho_i)}\omega^T+\frac{1}{2}d\Phi du+\eta\wedge d\left(\frac{1-du(\Phi \rho_i)}{1+\eta(\rho_i)}\right)\\
=&\omega^T+\frac{1}{2}d\Phi du+\left(\frac{1+2du(\Phi \rho_i)}{1+\eta(\rho_i)}-1\right)\omega^T+\eta\wedge d\left(\frac{1-du(\Phi \rho_i)}{1+\eta(\rho_i)}\right)
\end{split}
\end{equation}
If  $|d\Phi d u|\leq C_0$, then \eqref{approx1005} implies that $|d\Phi_i d u|\leq C_1$ (vice versa). Moreover, when $\rho_i\rightarrow 0$, \[\frac{1+2du(\Phi \rho_i)}{1+\eta(\rho_i)}\rightarrow 1,\;\;\; d\left(\frac{1-du(\Phi \rho_i)}{1+\eta(\rho_i)}\right)\rightarrow 0.\] We can then choose $\epsilon_i\rightarrow 0$ as $\rho_i\rightarrow 0$, such that
\[
\omega^T_i+\frac{1}{2} d\Phi_i d (u(1-\epsilon_i))\geq 0. 
\]
This proves \eqref{approx100}. 
Note that given the relation of $\Phi$ and $\Phi_i$, then  $|d \Phi du|\leq C_0$ implies that  $|d\Phi_i d u|$ is uniformly bounded (we suppose $\rho_i$ is uniformly small in smooth topology). Interchanging $\xi$ and $\xi_i$, this proves \eqref{approx101}. 
\end{proof}

\begin{rmk}
Note that the complex structure on the cone remains unchanged under Type-I deformation \cite{HeSun}[Lemma 2.2]. The transverse holomorphic structure is changed since the foliation is changed, due to the change of Reeb vector foliation; on the other hand, the contact bundle $\cD$ remains unchanged. Note that $(\cD, \Phi)$ and $(\cD, \Phi_i)$ can be identified to transverse holomorphic tangent bundle $T^{1, 0}(\cF_\xi)$ and $T^{1, 0}(\cF_{\xi_i})$ (the foliations are different).  Since  the term $\eta\wedge d\left(\frac{1-du(\Phi \rho_i)}{1+\eta(\rho_i)}\right)$ vanishes on $\cD$ and $\left(\frac{1+2du(\Phi \rho_i)}{1+\eta(\rho_i)}-1\right)\omega^T$ involves with only $du$, hence the above statement holds if we only assume that $|du|$ is uniformly bounded. Since we shall not need this, we skip the argument. However, it seems that assumption like $|du|\leq C$ is necessary and we are not able to extend this to $\text{PSH}(M, \xi, \omega^T)$. 
\end{rmk}

As above we fix a torus $T\subset \text{Aut}(N, \xi, \eta, g)$ and consider $\rho_i\in \mathfrak{t}$ sufficiently small. Let $\xi_i=\xi+\rho_i$ and let $(\xi_i, \eta_i, g_i, \Phi_i)$ be the Type-I deformation of $(\xi, \eta, g, \Phi)$. 
\begin{lemma}\label{measure100}Let $\rho_i\rightarrow 0$. Suppose a sequence of $T$-invariant functions $u_i\in \text{PSH}(M, \xi_i, \omega_i^T)$ with $|d\Phi d u_i|_{\omega^T}\leq C_0$ converges to $u\in \text{PSH}(M, \xi, \omega^T)$. Then $|d\Phi du|_{\omega^T}\leq C_0$ and we have the following weak convergence of the measure
\[
(\omega_i^T+\frac{1}{2}d\Phi_i d u_i)^n\wedge \eta_i\rightarrow (\omega^T+\frac{1}{2}d\Phi d u)^n\wedge \eta
\]
\end{lemma}

\begin{proof}By \eqref{approx1005} and $|d\Phi d u_i|_{\omega^T}\leq C_0$,  $\omega_i^T+\frac{1}{2}d\Phi_i d u_i$ and  $\omega^T+\frac{1}{2}d\Phi d u_i$ differ by a term with small $L^\infty$ norm, hence we only need to prove that
\[(\omega^T+\frac{1}{2}d\Phi d u_i)^n\wedge \eta_i\rightarrow  (\omega^T+\frac{1}{2}d\Phi d u)^n\wedge \eta.\]
Note that $\eta_i=\eta/(1+\eta(\rho_i))$ converges smoothly to $\eta$, then the above follows from the weak convergence of $(\omega^T+\frac{1}{2}d\Phi d u_i)^n\wedge \eta$.
\end{proof}

Next we  give a proof of Theorem \ref{rooftop101} in Sasaki setting, regarding the regularity of envelop construction. 
\begin{thm}Given $f\in C^\infty_B(M)$, then 
we have the following estimate
\[
\|P(f)\|_{C^{1, \bar 1}}\leq C(M, \omega^T, g, \|f\|_{C^{1, \bar 1}}).
\]
Moreover, if 
 $u_1, \cdots, u_k\in \cH_\Delta$, where we use the notation
\[
\cH_\Delta=\{u\in \text{PSH}(M, \xi, \omega^T): \|u\|_{C^{1, \bar 1}}<\infty\}
\]
then $P(u_1, \cdots, u_k)\in \cH_\Delta$.
\end{thm}

\begin{proof}
The first result was proved by Berman-Demailly \cite{BD} in K\"ahler setting. 
For the first statement, we follow \cite{D4}[Theorem A.7] and it is a direct adaption to Sasaki setting.  Consider the following complex Monge-Ampere equation on Sasaki manifolds,
\[
\omega_{u_\beta}^n\wedge \eta= e^{\beta(u_\beta-f)}\omega_T^n\wedge\eta.
\]
Since all quantities are basic and only transverse K\"ahler structure is involved, then the argument as in K\"ahler setting has a direct adaption; see \cite{D4}[Theorem A.7] and we skip the details. For the second statement, first note that we only need to show that $u_0, u_1\in \cH_\Delta$, then $P(u_0, u_1)\in \cH_\Delta$. Let $u_t$ be the geodesic segment connecting $u_0, u_1$, then by Lemma \ref{cma10}, we know that $u_t\in \cH_\Delta$ (see \cite{BD} and \cite{he12} for K\"ahler setting). Now we have already known $P(u_0, u_1)=\inf_{t\in [0, 1]} u_t$, then by \cite{DR1}[Proposition 4.4] (applied to each foliation charts), $\Delta u_t$ is uniformly bounded. This shows that $P(u_0, u_1)\in \cH_\Delta$. 
\end{proof}

More generally, one can obtain results as in \cite{DR1} that $P(f_1, \cdots, f_n)\in C^{1, \bar 1}_B$ given $f_1, \cdots, f_n\in C^{1, \bar 1}_B$. The point is that given two functions $f_1, f_2$, $h=\min \{f_1, f_2\}$ satisfies $\Delta h\leq \max\{\Delta f_1, \Delta f_2\}$ in viscosity sense, writing $h=\frac{f_1+f_2}{2}-\frac{|f_1-f_2|}{2}$. The argument as in \cite{D4}[Theorem A.7] applies using the maximum principle in viscosity sense. Since we do not need this, we shall skip the details.

\subsection{Complex Monge-Ampere operator and intrinsic capacity on compact Sasaki manifolds}\label{CMA001}
We discuss briefly the Bedford-Taylor theory on Sasaki manifolds. For details for complex Monge-Ampere operator, see Bedford-Taylor \cite{BT1}.
We also extend intrinsic Monge-Ampere capacity to Sasaki setting, see  \cite{GZ01} for K\"ahler setting.

Given a Sasaki structure, there is a splitting of tangent bundle $TM=L\xi\otimes \cD$, where $\cD=\text{Ker}(\eta)$, with $\Phi: \cD\rightarrow \cD$ inducing a splitting $\cD\otimes \C=\cD^{1, 0}\oplus \cD^{0, 1}$. Hence 
the subbundle $\Lambda^{2p}(\cD^*)$ of $\Lambda^{2p}M$  is well-defined and $\Phi$ induces a splitting to give bidegree of forms in $\Lambda^{2p}(\cD^*)$. Note that we have the following, 
\[\Lambda^{2p}(\cD^*)=\{\theta: \theta\in \Lambda^{2p}M, \iota_\xi \theta=0\}.\]
We do not assume that $\theta\in \Lambda^{2p}(\cD^*)$ is basic. That is, the coefficients of $\theta$ might not be invariant under the Reeb flow. A simple observation shows that if $\theta\in \Lambda^{2p}(\cD^*)$, then $\theta$ is basic if it is closed, $d\theta=0$  (since $\iota_\xi \theta=0$). Hence a closed $2p$-form in $\Lambda^{2p}(\cD^*)$ is basic and can be regarded as a \emph{transverse} closed $2p$-form, defined as in \cite{VC}. In general $d\Lambda^{2p}(\cD^*)$ is not in $\Lambda^{2p+1}(\cD^*)$.

Next we give a very brief discussion of \emph{transverse} positive closed currents of bidegree of $(p, p)$ on $M$, $0\leq p\leq n$; see \cite{VC} for similar treatment. 
We  simply treat them as closed differential forms of bidegree $(p, p)$ in $\Lambda^{2p}(\cD^*)$ with measurable coefficients which are invariant under the Reeb flow. Its total variation is controlled by 
\[
\|T\|:=\int_M T\wedge (\omega^T)^{n-p}\wedge \eta. 
\]
Given $\phi\in \text{PSH}(M, \xi, \omega^T)$, we write $\phi\in L^1(T)$ if $\phi$ is integrable with respect to the measure $T\wedge (\omega^T)^{n-p}\wedge \eta$. In this case, the current $\phi T$ is well-defined and we write
\[
\begin{split}
&\omega_\phi\wedge T:=\omega^T\wedge T+dd^c_B(\phi T)\\
&\omega_\phi\wedge T\wedge (\omega^T)^{n-p-1}\wedge \eta= T\wedge (\omega^T)^{n-p}\wedge \eta+dd^c_B(\phi T)\wedge (\omega^T)^{n-p-1}\wedge \eta.
\end{split}
\]
The positivity is a local notion and we simply think $T$ as a positive closed $(p, p)$-form on each foliation chart. Hence $\omega_\phi\wedge T$ is also a transverse closed positive $(p+1, p+1)$ form. Note that we think transverse positive closed currents of bidegree of $(p, p)$-type as a linear functional on $\Lambda^{n-p, n-p}(\cD^*)$, hence the test forms are of bidegree $(n-p, n-p)$. A main point is that test forms are not restricted to basic forms. In other words, given such a current $T$ and $\gamma\in \Lambda^{n-p, n-p}(\cD^*)$, we have the following paring,
\[
\gamma\rightarrow \int_M \gamma\wedge T\wedge \eta. 
\]
When $\phi\in \text{PSH}(M, \xi, \omega^T)\cap L^\infty$, it follows that $\phi\in L^1(T)$ for any transverse positive closed current $T$ of bidegree $(p, p)$ and hence one can define inductively $\omega_\phi^k\wedge (\omega^T)^{n-k}$; in particular, this leads to the definition of transverse complex Monge-Ampere operator $\omega_\phi^n$ of bidegree $(n, n)$. 
Moreover, the cocycle condition on transverse holomorphic structure ensures that $\omega_\phi^k\wedge (\omega^T)^{n-k}$  is well-defined on $M$. 
In particular $\omega_\phi^n\wedge \eta$ defines a positive Borel measure on $M$.

It is more convenient to consider this construction locally in foliations charts $W_\alpha=(-\delta, \delta)\times V_\alpha$. 
By taking test forms $\gamma\in \Lambda^{n-p, n-p}(\cD^*)$ with compact support, we can consider $T\wedge \eta$ on a foliation chart for a transverse positive closed $(p, p)$ current $T$. In particular this give a local description of the complex Monge-Ampere measures $\omega_\phi^k\wedge (\omega^T)^{n-k}\wedge \eta$.
By taking test functions $f$ supported in a foliation chart, the measure $\omega_\phi^k\wedge (\omega^T)^{n-k}\wedge \eta$ for each $k$ is regarded as the product measure $\omega_\phi^k\wedge (\omega^T)^{n-k}\wedge dx$ on $W_\alpha$, where $\xi=\p_x$ is the Reeb direction. Note that $\omega_\phi^k\wedge (\omega^T)^{n-k}$ is defined on $V_\alpha$ as the usual way in K\"ahler setting, and the cocycle condition on transverse holomorphic structure ensures that $\omega_\phi^k\wedge (\omega^T)^{n-k}$ is well-defined as a transverse positive closed current of bidegree $(n, n)$. On each foliation chart,  we have $\omega_\phi^k\wedge (\omega^T)^{n-k}\wedge\eta=\omega_\phi^k\wedge (\omega^T)^{n-k}\wedge dx$ as a product measure. This coincides with the local description given by van Coevering \cite{VC}[Section 2]. 

Moreover, when $u, v\in \text{PSH}(M, \xi, \omega^T)\cap L^\infty$, $du\wedge d^c_B v\wedge T$ can also be defined, where $T$ is a transverse closed positive current of bidegree $(n-1, n-1)$. 
By the polarization formula we only need to define $du\wedge d^c_B u\wedge T$. By adding a positive constant if necessary, we assume $u\geq 0$. Then we define
\begin{equation}
du\wedge d^c_B u\wedge T:=\frac{1}{2} dd^c_B (u^2)\wedge T- udd^c_B u\wedge T.
\end{equation}
 In particular, $du\wedge d^c_B u\wedge T$ is positive if $T$ is a transverse closed positive current of bidegree $(n-1, n-1)$. We can then define $du\wedge d^c_B u\wedge T\wedge \eta$ as a positive 
Borel measure.  Using the polarization formula, we have the following Cauchy-Schwartz inequality, for $u, v\in \text{PSH}(M, \xi, \omega^T)\cap L^\infty$,
\begin{equation}\label{cs01}
|\int_M du\wedge d^c_B v\wedge T\wedge \eta|^2\leq \left(\int_M du\wedge d^c_B u \wedge T\wedge \eta\right)\left(\int_M dv\wedge d^c_B v \wedge T\wedge \eta\right)
\end{equation}
We also record the following Stokes' theorem in Sasaki setting, and its proof follows the Bedford-Taylor theory as in K\"ahler setting via approximation (Lemma \ref{BK}); see \cite{VC}[Theorem 2.3.1, Proposition 2.3.2]. 
\begin{lemma}Let $u, v, \phi\in \text{PSH}(M, \xi, \omega^T)\cap L^\infty$, then for each $0\leq k\leq n-1$, we have
\begin{equation}
\begin{split}
\int_M u dd^c_B v\wedge \omega_\phi^{k}\wedge (\omega^T)^{n-k-1}\wedge \eta=&\int_M v dd^c_B u \wedge \omega_\phi^{k}\wedge (\omega^T)^{n-k-1}\wedge \eta\\
=&-\int_M d u\wedge d^c_B v \wedge \omega_\phi^{k}\wedge (\omega^T)^{n-k-1}\wedge \eta
\end{split}
\end{equation}
\end{lemma}

We record a basic inequality in Sasaki setting, usually referred to Chern-Levine-Nirenberg inequality, 

\begin{prop}[Chern-Levine-Nirenberg inequalities]\label{CLN} Let $T$ be a positive closed current
of bidegree $(p, p)$ on M and $\phi\in \text{PSH}(M, \xi, \omega^T)\cap L^\infty$. Then $\|\omega_\phi\wedge T\|=\|T\|$. Moreover, if $\psi\in \text{PSH}(M, \xi, \omega^T)\cap L^1(T)$, then $\psi\in L^1(\omega_\phi\wedge T)$ and 
\begin{equation}\label{cln01}
\|\psi\|_{L^1(T\wedge \omega_\phi)}\leq \|\psi\|_{L^1(T)}+(2\max\{\sup \psi, 0\}+\sup \phi-\inf \phi) \|T\|.
\end{equation}
\end{prop}

\begin{proof}By Stokes' theorem, we have $\int_M dd^c_B (\phi T)\wedge (\omega^T)^{n-p-1}\wedge \eta=0$, hence
\[
\|\omega_\phi\wedge T\|=\int_M \omega^T\wedge T\wedge (\omega^{T})^{n-p-1}\wedge \eta=\|T\|. 
\]
To prove \eqref{cln01}, we first assume $\psi\leq 0, \phi\geq 0$. By assumption, $\psi\in L^1(T)$, then
\[
\|\psi\|_{L^1(T\wedge \omega_\phi)}:=\int_M -\psi T\wedge \omega_\phi \wedge (\omega^T)^{n-p-1}\wedge \eta =\|\psi\|_{L^1(T)}+\int_M -\psi dd^c_B(\phi T)\wedge (\omega^T)^{n-p-1}\wedge \eta
\]
By Stokes' theorem we compute
\[
\begin{split}
\int_M -\psi dd^c_B(\phi T)\wedge (\omega^T)^{n-p-1}\wedge \eta=&\int_M dd^c_B(-\psi) \wedge \phi T\wedge (\omega^T)^{n-p-1}\wedge \eta\\
\leq &\int_M   \phi T\wedge (\omega^T)^{n-p}\wedge \eta\\
\leq &\sup_M \phi \int_M T\wedge (\omega^T)^{n-p}\wedge \eta=(\sup_M\phi) \|T\|. 
\end{split}
\]
Now suppose $\sup \psi>0$. 
Replacing $\phi$ by $\phi-\inf\phi$, we compute
\[
\|\psi\|_{L^1(T\wedge \omega_\phi)}\leq \int_M (2\sup \psi-\psi)T\wedge \omega_\phi \wedge (\omega^T)^{n-p-1}\wedge \eta
\]
The same argument as above leads to \eqref{cln01} for the general case. 
\end{proof}
For a Borel subset $E$ on a Sasaki manifold $(M,\xi,\omega^T)$, we define the capacity as
\begin{equation*}
\text{cap}_{\omega^T}(E):=\sup\{\int_E \omega_{\varphi}^n \wedge \eta: \varphi \in \text{PSH}(M,\xi,\omega^T), 0 \leq \varphi \leq 1 \}
\end{equation*}
It is obvious that $\text{cap}_{\omega^T}(\cup_{k=1}^{\infty}E_k)\leq\sum\limits_{k=1}^{\infty}\text{cap}_{\omega^T}(E_k)$ for a sequence of Borel sets $E_k$.
We have the following,

\begin{prop}\label{capacity0}Let $\phi\in \text{PSH}(M, \xi, \omega^T)$ with $0\leq \phi\leq 1$ and $\psi\in \text{PSH}(M, \xi, \omega^T)$ such that $\psi\leq 0$. Then
\begin{equation}\label{capacity001}
\int_M -\psi \omega_\phi^n\wedge \eta\leq \int_M (-\psi)(\omega^T)^n\wedge \eta+n \int_M (\omega^T)^n\wedge \eta
\end{equation}
\end{prop}
\begin{proof}We only need to prove \eqref{capacity001} for canonical cutoffs $\psi_k=\max\{\psi, -k\}$ ($-\psi_k$ increases to $-\psi$ and we can apply monotone convergence theorem). We have the following \[
\begin{split}
\int_M -\psi_k\omega_\phi^n\wedge \eta=&\int_M -\psi_k\omega_\phi^{n-1}\wedge (\omega^T+\sqrt{-1}\p_B\bar\p_B \phi)\wedge \eta\\
=&\int_M -\psi_k\omega_\phi^{n-1}\wedge \omega^T\wedge \eta+\int_M -\psi_k \omega_\phi^{n-1}\wedge\sqrt{-1}\p_B\bar\p_B \phi\wedge \eta\\
=&\int_M -\psi_k\omega_\phi^{n-1}\wedge \omega^T\wedge \eta+\int_M \phi  \omega_\phi^{n-1}\wedge (-\sqrt{-1}\p_B\bar\p_B \psi_k)\wedge \eta\\
\leq& \int_M -\psi_k\omega_\phi^{n-1}\wedge \omega^T\wedge \eta+\int_M (\omega_\phi)^{n-1}\wedge \omega^T\wedge \eta\\
\leq& \int_M -\psi_k\omega_\phi^{n-1}\wedge \omega^T\wedge \eta+\int_M (\omega^T)^{n}\wedge \eta
\end{split}
\]
We can then proceed inductively to obtain \eqref{capacity001}. Note that the argument above is a special case of \eqref{cln01}. 
\end{proof}

\begin{prop}\label{capacity1}
Suppose that $u \in \text{PSH}(M,\xi,\omega^T)$ and $u\leq0$. Then for $t>0$ we have
\begin{equation*}
\text{cap}_{\omega^T}(\{u <-t\}) \leq \frac{1}{t}(\int_M (-u) (\omega^T)^n\wedge\eta+n\int_M(\omega^T)^n\wedge\eta)
\end{equation*}
\end{prop}
\begin{proof}This is a direct consequence of Proposition \ref{capacity0}. Denote $K_t=\{u<-t\}$, then
\[
\begin{split}
\int_{K_t}\omega_\phi^n\wedge \eta\leq& \frac{1}{t}\int_M -\psi \omega_\phi^n\wedge \eta\\
\leq& \frac{1}{t}\left(\int_M -\psi (\omega^T)^n\wedge \eta+n \int_M (\omega^T)^n\wedge \eta\right)
\end{split}
\]
\end{proof}

\begin{prop}\label{capacity2}
Suppose that $u_k, u \in \text{PSH}(M,\xi,\omega^T) \cap L^{\infty}$ and $u_k$ decreases to $u$. Then for $\delta>0$ we have
\begin{equation*}
\text{cap}_{\omega^T}(\{u_k>u+\delta\}) \rightarrow 0, k\rightarrow \infty.
\end{equation*}
\end{prop}
\begin{proof}This proceeds exactly the same as in \cite{GZ01}[Proposition 3.7]. We sketch the argument briefly. We assume $\text{Vol}(M)=1$ for simplicity. Fix $\delta>0$ and $\phi\in \text{PSH}(M, \xi, \omega^T)$ such that $0\leq \phi\leq 1$. We have
\[
\int_{\{u_k>u+\delta\}} \omega_\phi^n\wedge \eta\leq \delta^{-1}\int_M (u_k-u) \omega_\phi^n\wedge \eta
\]
By Stokes' theorem, we write
\begin{equation*}
\begin{split}
\int_M (u_k-u) \omega_\phi^n\wedge \eta=&\int_M (u_k-u) \wedge \omega^T\wedge \omega_\phi^{n-1}\wedge \eta+\int_M (u_k-u) \wedge dd^c_B \phi\wedge \omega_\phi^{n-1}\wedge \eta\\
=&\int_M (u_k-u) \wedge \omega^T\wedge \omega_\phi^{n-1}\wedge \eta-\int_M d(u_k-u) \wedge d^c_B \phi\wedge \omega_\phi^{n-1}\wedge \eta
\end{split}
\end{equation*}
By the Cauchy-Schwartz inequality, setting $f_k=u_k-u\geq 0$,
\[
|\int_M d(u_k-u) \wedge d^c_B \phi\wedge \omega_\phi^{n-1}\wedge \eta|^2\leq \int_M d f_k\wedge d^c_B f_k\wedge \wedge \omega_\phi^{n-1}\wedge \eta \int_M d \phi\wedge d^c_B \phi\wedge \wedge \omega_\phi^{n-1}\wedge \eta
\]
We compute
\[
\int_M d \phi\wedge d^c_B \phi\wedge \wedge \omega_\phi^{n-1}\wedge \eta=\int_M \phi (-dd^c_B\phi)\wedge \omega_\phi^{n-1}\wedge \eta\leq \int_M \phi \omega^T\wedge  \omega_\phi^{n-1}\wedge \eta\leq 1
\]
Similarly, we compute
\[
\int_M d f_k\wedge d^c_B f_k\wedge \wedge \omega_\phi^{n-1}\wedge \eta=\int_M f_k (dd^c_B u-dd^c_B u_k)\wedge \omega_\phi^{n-1}\wedge \eta\leq \int_M f_k \omega_u\wedge \omega_\phi^{n-1}\wedge \eta.
\]
Combining all these together this gives
\[
\int_M (u_k-u) \omega_\phi^n\wedge \eta\leq \int_M (u_k-u) \wedge \omega^T\wedge \omega_\phi^{n-1}\wedge \eta+ (\int_M (u_k-u) \omega_u\wedge \omega_\phi^{n-1}\wedge \eta)^{1/2}. 
\]
Suppose  $u_k-u\leq c_0$ for a fixed positive constant $c_0\geq 1$. Then we have
\[
\int_M (u_k-u) \omega_\phi^n\wedge \eta\leq \sqrt{c_0} (\int_M (u_k-u) \wedge \omega^T\wedge \omega_\phi^{n-1}\wedge \eta)^{1/2}+ (\int_M (u_k-u) \omega_u\wedge \omega_\phi^{n-1}\wedge \eta)^{1/2}. 
\]
Hence we have
\[
\int_M (u_k-u) \omega_\phi^n\wedge \eta\leq \sqrt{2c_0} (\int_M (u_k-u) \wedge (\omega^T+\omega_u)\wedge \omega_\phi^{n-1}\wedge \eta)^{1/2}
\]
We can proceed inductively by replacing $\omega_\phi$ by $\omega^T+\omega_u$ to obtain
\[
\int_M (u_k-u) \omega_\phi^n\wedge \eta\leq (\sqrt{2c_0})^n (\int_M (u_k-u) \wedge (\omega^T+\omega_u)^{n}\wedge \eta)^{1/2^n}
\]
The dominated convergence theorem implies the righthand side goes to zero, independent of $\phi$. This completes the proof. 
\end{proof}

As a consequence, we have the following, 
\begin{thm}\label{quasicontinuity}
Let $\varphi \in \text{PSH}(M,\xi,\omega^T)$, then for any $\epsilon>0$ there exists an open subset $O_{\epsilon} \subset M$ such that $\text{cap}_{\omega^T}(O_{\epsilon}) < \epsilon$ and $\varphi$ is continuous on $M-O_{\epsilon}$.
\end{thm}
\begin{proof}
By Proposition \ref{capacity1} there exists $t_0>0$ such that $\text{cap}_{\omega^t}(O_0) <\frac{\epsilon}{2}$ for the open subset $O_0=\{u<-t_0\}$. Take the cutoff $u_{t_0}=\max\{u,-t_0\} \in \text{PSH}(M,\xi,\omega^T)$, then there exists a sequence $u_k \in \cH$ decreasing to $u$. By Proposition \ref{capacity2} we can choose a subsequence $u_{k_j}$ such that
$\text{cap}_{\omega^T}(O_j) < \frac{\epsilon}{2^{j+1}}$ for the open subset $O_j=\{u_{k_j}>u+\frac{1}{j}\}$. Then for the open subset $O_{\epsilon}=\cup_{j=0}^{\infty}O_j$ we have $\text{cap}_{\omega^T}(O)<\epsilon$. Moreover $u_{K_j}$ converges uniformly to $u$ on $M-O_{\epsilon}$, hence $u$ is continuous on $M-O_{\epsilon}$.
\end{proof}

\begin{rmk}The discussions above are taken from K\"ahler setting \cite{GZ01}[Section 3]. Note that in \eqref{cln01} it is necessary to replace $\sup \psi$ by $\max\{\sup \psi, 0\}$  (similarly one needs to replace $\sup_X\psi$ by $\max\{\sup_X \psi, 0\}$ in \cite{GZ01}[Proposition 3.1])\end{rmk}

We also need the following uniqueness in Sasaki setting , see \cite{GZ}[Theorem 3.3].

\begin{thm}Suppose $u, v\in \cE_1(M, \xi, \omega^T)$ such that
\[
\omega_u^n\wedge \eta=\omega_v^n\wedge \eta,
\]
then $u-v=\text{const}$.
\end{thm}

\begin{proof}This follows exactly as in \cite{GZ}[Theorem 3.3] and we sketch the argument. The first step is that for $u\in \cE_1(M, \xi, \omega^T)$ and its canonical cutoffs $u_j=\max\{u, -j\}$, then $\nabla u_j\in L^2(d\mu_g)$ and has uniformly bounded $L^2$ norm (see \cite{GZ}[Proposition 3.2]). We can assume that $u\leq 0$ and hence $u_j\leq 0$. Then for $\phi\in \text{PSH}(M, \xi, \omega^T)\cap L^\infty$ such that $\phi\leq 0$, we know that, for any basic positive closed of $(n-1, n-1)$ type. 
\[
\int_M (-\phi)\omega\wedge T=\int_M (-\phi) (\omega_\phi-dd^c_B\phi)\wedge T=\int_M (-\phi)\omega_\phi\wedge T+\int_M d\phi\wedge d^c_B \phi\wedge T\leq \int_M (-\phi)\omega_\phi\wedge T
\]
An inductive argument applies to $T=\omega_\phi^k\wedge (\omega^T)^{n-k-1}$, we get that
\begin{equation}\label{el2}
0\leq \int_M d\phi\wedge d^c_B \phi\wedge T\leq \int_M (-\phi)\omega_\phi^n\wedge \eta. 
\end{equation}
Taking $\phi=u_j$ in \eqref{el2} and noting that the righthand side is uniformly bounded, we get $\nabla u_j$ is uniformly bounded in $L^2(d\mu_g)$, hence $\nabla u\in L^2(d\mu_g)$. 

We assume that $u, v\leq -1$ and $\text{Vol}(M)=1$. Set $f=(u-v)/2$ and $h=(u+v)/2$. We need to establish that $\nabla f=0$ by showing that $\int_M df\wedge d^c_B f\wedge (\omega^T)^{n-1}\wedge \eta=0$. If we assume $u, v$ are bounded, then we have
\begin{equation}\label{uni01}
 \int_M df\wedge d^c_B f\wedge \omega^{n-1}_h\wedge \eta\leq \sum \int_M  df\wedge d^c_B f\wedge \omega_u^k\wedge \omega^{n-1-k}_v\wedge \eta=-\int_M \frac{f}{2}(\omega_u^n-\omega_v^n)\wedge \eta,
\end{equation}
where we use the fact that $dd^c_B f=(\omega_u-\omega_v)/2$. 
We shall also establish the following a priori bound, when $u, v$ are bounded, 
\begin{equation}\label{uni02}
\int_M df\wedge d^c_B f\wedge (\omega^T)^{n-1}\wedge \eta \leq 3^n \left( \int_M df\wedge d^c_B f\wedge \omega^{n-1}_h\wedge \eta\right)^{1/2^{n-1}}. 
\end{equation}
We apply \eqref{uni01} and \eqref{uni02} to the canonical cutoffs $u_j, v_j$  (writing $f_j, h_j$ correspondingly and using Proposition \ref{weakcon3}), \[
 \lim \int_M df_j\wedge d^c_B f_j\wedge (\omega^T)^{n-1}\wedge \eta=0
\]
We can then conclude that
\[
\int_M df\wedge d^c_B f\wedge (\omega^T)^{n-1}\wedge \eta=0.
\]
This implies that $u-v$ is a constant. To establish \eqref{uni02}, we need several observations as follows. 
First observe that for $l=n-2, \cdots, 0$,
\[
\int_M (-h)\omega^{2+l}_h\wedge (\omega^T)^{n-2-l}\wedge \eta\leq \int_M (-h)(\omega^T)^n\wedge \eta\leq 1,
\]
where the last inequality follows from $-h\leq 1$ and the normalization of the volume. 
We can then apply the following inequality inductively for $T=\omega_h^l\wedge (\omega^T)^{n-l-1}$ such that
\begin{equation}\label{uni03}
\int_M df\wedge d^c_B f\wedge \omega^T\wedge T\wedge \eta\leq 3 \left(\int_Mdf\wedge d^c_B f\wedge \omega_h\wedge T\wedge \eta\right)^{1/2},
\end{equation}
which proves \eqref{uni02}. Now we establish \eqref{uni03}. We write
\[
df\wedge d^c_B f\wedge \omega^T=df\wedge d^c_B f\wedge \omega_h-df\wedge d^c_B f\wedge dd^c_B h
\]
hence we obtain, integrating by parts, 
\[
\int_M df\wedge d^c_B f\wedge \omega^T\wedge T\wedge \eta=\int_M df\wedge d^c_B f\wedge \omega_h\wedge T\wedge \eta+\int_M df\wedge d^c_Bh\wedge \frac{\omega_u-\omega_v}{2}\wedge T\wedge \eta
\]
By Cauchy-Schwartz inequality, we have
\[
|\int_M df\wedge d^c_Bh\wedge\omega_u\wedge T\wedge \eta|^2\leq 4 \int_M df\wedge d^c_Bf\wedge\omega_h\wedge T\wedge \eta \int_M dh\wedge d^c_Bh\wedge\omega_h\wedge T\wedge \eta
\]
We can get a similar control 
\[
|\int_M df\wedge d^c_Bh\wedge\omega_v\wedge T\wedge \eta|^2\leq 4 \int_M df\wedge d^c_Bf\wedge\omega_h\wedge T\wedge \eta \int_M dh\wedge d^c_Bh\wedge\omega_h\wedge T\wedge \eta
\]
Clearly we have the following ($h\leq 0, S=\omega_h^l\wedge (\omega^T)^{n-l-2}$)
\[
\int_M dh\wedge d^c_Bh\wedge\omega_h\wedge S\wedge \eta\leq \int_M (-h) \omega_h^2\wedge S\wedge \eta\leq 1. 
\]
Combining these estimate altogether we conclude that,
\[
\int_M df\wedge d^c_B f\wedge \omega^T\wedge S\wedge \eta\leq \int_M df\wedge d^c_B f\wedge \omega_h\wedge T\wedge \eta+2\left( \int_M df\wedge d^c_Bf\wedge\omega_h\wedge T\wedge \eta\right)^{1/2}
\]
The last observation is that 
\[
\int_M df\wedge d^c_B f\wedge \omega_h\wedge S\wedge \eta=\frac{1}{4}\int_M (u-v)(\omega_v-\omega_u)\wedge \omega_h\wedge S\wedge \eta\leq \int_M (-h) \omega_h^2\wedge S\wedge \eta\leq 1. 
\]
This completes the proof of \eqref{uni03} by combining two inequalities above. 
\end{proof}

\subsection{Functionals in finite energy class $\cE_1$ and compactness}
We discuss briefly well-known functionals in K\"ahler geometry and their properties over finite energy class $\cE_1$, see \cite{D4}[Section 3.8].
The energy functionals include Monge-Ampere energy $\mathbb{I}$ and Aubin's $I$-functional on $\cE_1$, see \cite{Aubin, BBGZ10, BBGZ13, BBEGZ, D4} for K\"ahler setting. These results have a direct adaption in Sasaki setting. 
Recall Aubin's $I$-functional in Sasaki setting, for $u, v\in \cH$
\begin{equation}\label{aubin01}
 I(u, v):=I(\omega_u, \omega_v)=\frac{1}{n!}\int_M (v-u) (\omega_u^n-\omega_v^n)\wedge \eta.
\end{equation}
We also recall the $J$-functional
\begin{equation}
J(u, v):=J(\omega_u, \omega_v)=\frac{1}{n!}\int_M (v-u) \omega_u^n\wedge \eta-\mathbb{I}_{\omega_u}(v),
\end{equation}
where the $\mathbb{I}_{\omega_u}(v)$-functional is given by
\begin{equation}
\mathbb{I}_{\omega_u}(v)=\frac{1}{(n+1)!}\int_M (v-u)\sum_{k=0}^n \omega_u^k\wedge \omega_v^{n-k}\wedge \eta.
\end{equation} 
We define the $\mathbb{I}$-functional (with the base $\omega^T$) on $\cH$,
\begin{equation}\label{maenergy}
\mathbb{I}_{\omega^T}(u)=\frac{1}{(n+1)!}\int_M u\sum_{k=0}^n \omega_u^k\wedge \omega^{n-k}_T\wedge \eta. 
\end{equation}
The $\mathbb{I}$-functional is also called the Monge-Amp\`ere energy, since if $t\rightarrow v_t\in \cH$ is smooth, then we have (as in K\"ahler setting), 
\begin{equation}\label{maderivative}
\frac{d}{dt}\mathbb{I}(v_t)=\frac{1}{n!}\int_M \dot v_t \omega_{v_t}^n \wedge \eta
\end{equation}
We mention that $I$ is symmetric with respect to $u, v$ but $J$ is not. $I, J$ are both defined on the metric level, independent of the choice of normalization of potentials $u, v$; while $\mathbb{I}_{\omega_u}(v)$ depends on the normalization of $u, v$.  
When $u, v$ are bounded, then Bedford-Taylor theory allows to integrate by parts and the $I$-functional takes the formula
\begin{equation}\label{i04}
I(\omega_u, \omega_v)=\frac{1}{(n+1)!}\sum_{j=0}^{n-1} \int_M d(u-v)\wedge d^c_B(u-v)\wedge \omega_u^j\wedge \omega_v^{n-1-j}\wedge \eta
\end{equation}
Hence it is nonnegative.

We need more information about $\mathbb{I}$-functional, see \cite{D4}[Section 3.7] for K\"ahler setting. These properties in Sasaki setting follow in a rather straightforward way given pluripotential theory extended to Sasaki setting. We include these facts here for completeness. 

\begin{prop}
Given $u, v\in \text{PSH}(M, \xi, \omega^T)\cap L^\infty$, the following cocycle condition holds
\begin{equation}\label{cocycle}
\mathbb{I}(u)-\mathbb{I}(v)=\frac{1}{(n+1)!}\sum_{k=0}^n\int_M (u-v)\omega_u^k\wedge \omega_v^{n-k}\wedge \eta=\mathbb{I}_{\omega_u}(v).
\end{equation}
Moreover, we have $\mathbb{I}(u)$ is concave in $u$ in the sense that,
\begin{equation}\label{maenergy01}
\frac{1}{n!}\int_M (u-v)\omega_u^n\wedge \eta\leq \mathbb{I}(u)-\mathbb{I}(v)\leq \frac{1}{n!}\int_M (u-v)\omega_v^n\wedge \eta.
\end{equation}
As a direct consequence, if $u, v\in \text{PSH}(M, \xi, \omega^T)\cap L^\infty$ such that $u\geq v$. Then $\mathbb{I}(u)\geq \mathbb{I}(v)$. 
\end{prop}

\begin{proof}This follows almost identical as in \cite{D4}[Proposition 3.8], given the pluripotential theory established in Sasaki setting in the paper. We sketch the argument. 
When $u, v\in \cH$, this follows exactly the same as in K\"ahler setting, by taking $h_t=(1-t)u+tv$ and then use \eqref{maderivative} to compute directly. When $u, v\in \text{PSH}(M, \xi, \omega^T)\cap L^\infty$, we then use $u_k, v_k\in \cH$ decreasing to $u, v$ (Lemma \ref{BK}) respectively. Using Bedford-Taylor's theorem in Sasaki setting \cite{VC}[Theorem 2.3.1] we proceed exactly as in K\"ahler setting to conclude that $\mathbb{I}(u_k)\rightarrow \mathbb{I}(u)$ etc. For the estimate \eqref{maenergy01}, we compute
\[
\begin{split}
\int_M (u-v)\omega^k_u\wedge \omega_v^{n-k}\wedge \eta=&\int_M (u-v)\omega^{k-1}_u\wedge \omega_v^{n-k+1}\wedge \eta\\&+\int_M (u-v)\sqrt{-1}\p\bar\p (u-v)\wedge \omega_u^{k-1}\wedge \omega_v^{n-k}\wedge \eta\\
=&\int_M (u-v)\omega^{k-1}_u\wedge \omega_v^{n-k+1}\wedge \eta\\
&-\int_M \sqrt{-1}\p(u-v)\wedge\bar\p (u-v)\wedge \omega_u^{k-1}\wedge \omega_v^{n-k}\wedge \eta\\
\leq & \int_M (u-v)\omega^{k-1}_u\wedge \omega_v^{n-k+1}\wedge \eta
\end{split}
\]
Using the estimate inductively for the terms in \eqref{cocycle} leads to \eqref{maenergy01}. Clearly $\mathbb{I}(u)$ is concave in $u$ given \eqref{maenergy01}. 
\end{proof}
The monotonicity property allows to define $\mathbb{I}(u)$ for $u\in \text{PSH}(M, \xi, \omega^T)$ through the limit process, using the canonical cutoffs $u_k=\max\{u, -k\}$
\[
\mathbb{I}(u)=\lim_{k\rightarrow \infty}\mathbb{I}(\max\{u, -k\}).
\]
Though the above limit is well-defined, it may equal $-\infty$. It turns out $\mathbb{I}(u)$ is finite exactly on $\cE_1(M, \xi, \omega^T)$. 
We record some further properties of $\mathbb{I}(u)$ for $u\in \cE_1(M, \xi, \omega^T)$. The proofs are almost identical and we shall skip the details, see \cite{D4}[Proposition 3.40, 3.42, 3.43; Lemma 3.41]. 

\begin{prop}\label{darvascontinuity01}Let $u\in \text{PSH}(M, \xi, \omega^T)$. Then $-\infty<\mathbb{I}(u)$ if and only if $u\in \cE_1(M, \xi, \omega^T)$. Moreover, 
\begin{equation}
|\mathbb{I}(u_0)-\mathbb{I}(u_1)|\leq d_1(u_0, u_1), u_0, u_1\in \cE_1(M, \xi, \omega^T). 
\end{equation}
\end{prop}

\begin{prop}Suppose $u_0, u_1\in \cE_1(M, \xi, \omega^T)$ and $t\rightarrow u_t$ is the finite energy geodesic connecting $u_0, u_1$. Then $t\rightarrow \mathbb{I}(u_t)$ is linear in $t$. We also have the following distance formula,
\[
d_1(u_0, u_1)=\mathbb{I}(u_0)+\mathbb{I}(u_1)-2\mathbb{I}(P(u_0, u_1))
\]
In particular, $d_1(u_0, u_1)=\mathbb{I}(u_0)-\mathbb{I}(u_1)$ if $u_0\geq u_1$. 
\end{prop}

We have the following (see \cite{D4}[Lemma 3.47])
\begin{lemma}\label{i02}Suppose $u, u^j, v, v^j\in \cE_1(M, \xi, \omega^T)$ and $u^j\searrow u$ and $v^j\searrow v$. Then the following hold:
\begin{equation}\label{i01}
I(u, v)=I(u, \max{\{u, v\}})+I(\max{\{u, v\}}, v)
\end{equation}
Moreover, $\lim_{j\rightarrow \infty} I(u^j, v^j)=I(u, v)$. 
\end{lemma} 
\begin{proof}By Proposition \ref{GBTI}, we have \[\chi_{\{v>u\}}\omega^n_{\max \{u, v\}}\wedge \eta=\chi_{\{v>u\}}\omega^n_v\wedge \eta.\]
Hence it follows that
\[
I(u, \max{\{u, v\}})=\frac{1}{(n+1)!}\int_{\{v>u\}} (u-v)(\omega_v^n-\omega_u^n)\wedge \eta
\]
Interchange $u\leftrightarrow v$, we get $I(v, \max{\{u, v\}})=\int_{\{u>v\}} (u-v)(\omega_v^n-\omega_u^n)\wedge \eta$. This proves \eqref{i01}. We write
\[
I(u^j, v^j)=I(u^j, \max{\{u^j, v^j\}})+I(v^j, \max{\{u^j, v^j\}})
\]
Since $u^j, v^j\leq \max\{u^j, v^j\}$, we can apply Proposition \ref{weakcon3} to conclude $I(u^j, \max{\{u^j, v^j\}})\rightarrow I(u, \max{\{u, v\}})$ and $I(v^j, \max{\{u^j, v^j\}})\rightarrow I(v, \max{\{u, v\}})$, using the formula \eqref{aubin01}. This completes the proof. 
\end{proof}

We have the following well-known inequalities,
\begin{prop}\label{i00}For $u, v\in \text{PSH}(M, \xi, \omega^T)\cap L^\infty$, we have
\[
\frac{1}{n+1}I(u, v)\leq J(u, v)\leq \frac{n}{n+1} I(u, v)
\]
Moreover, $J(u, v)$ is convex in $v$ since $\mathbb{I}_{\omega^T}(v)$ is concave in $v$. 
\end{prop}

\begin{proof} This is well-known, by direct computation \cite{PG}[Proposition 4.2.1] for $u, v\in \cH$. A direct approximation argument using Lemma \ref{BK} shows that this can be generalized to for $u, v\in \text{PSH}(M, \xi, \omega^T)\cap L^\infty$. 
\end{proof}

The functionals ($I, J, \mathbb{I}$) are well-defined for $u, v\in \cE_1(M, \xi, \omega^T)$  (see Proposition \eqref{boundedenergy}). Note that Proposition \ref{maenergy01} and
 Proposition \ref{i00} both hold in $\cE_1(M, \xi, \omega^T)$ (see \cite{BBGZ10, BBGZ13} for K\"ahler setting). 
This follows by an approximation argument applying Proposition \ref{weakcon3}. 
 Next we prove the following, as a direct adaption of \cite{BBEGZ}[Theorem 1.8],
\begin{lemma}\label{bbegz01}There exists a positive $C=C(n)$ such that for $u, v, w\in \cE_1(M, \xi, \omega^T)$, then
\begin{equation}\label{i03}
I(u, v)\leq C(I(u, w)+I(v, w))
\end{equation}
\end{lemma}
\begin{proof}
With Lemma \ref{i02}, we only need to argue \eqref{i03} holds for bounded potentials, with $u, v, w$ replaced by canonical cutoffs $u_k, v_k, w_k$. The proof follows exactly as in  \cite{BBEGZ}[Theorem 1.8, Lemma 1.9]. 
and we include the proof for completeness. 
For $u, v, \psi\in \text{PSH}(M, \xi, \omega^T)\cap L^\infty$, set
\[
\|d(u-v)\|_\psi:=\left(\int_M d(u-v)\wedge d^c_B(u-v)\wedge \omega_\psi^{n-1}\wedge \eta\right)^{\frac{1}{2}}
\]
Using \eqref{i04}, it is straightforward to see that
\begin{equation}
\label{i05}
\|d(u-v)\|_{\frac{u+v}{2}}^2\leq I(u, v)\leq 2^{n-1}\|d(u-v)\|_{\frac{u+v}{2}}^2. 
\end{equation}
We need the following, there exists a constant $C=C(n)$ for $u, v, \psi\in \text{PSH}{M, \xi, \omega^T}\cap L^\infty$, we have the following (see \cite{BBEGZ}[Lemma 1.9]),
\begin{equation}\label{i06}
\|d(u-v)\|^2_\psi\leq C I(u, v)^{1/2^{n-1}}\left(I(u, \psi)^{1-1/2^{n-1}}+I(v, \psi)^{1-1/2^{n-1}}\right)
\end{equation}
With \eqref{i06} we prove \eqref{i03}. Taking $\phi=\frac{u+v}{2}$, the triangle inequality gives,
\[
\|d(u-v)\|_\phi\leq \|d(u-w)\|_\phi+\|d(v-w)\|_\phi.
\]
Using \eqref{i05} and \eqref{i06} we have
\[
\begin{split}
I(u, v)\leq 2^{n-1} \|d(u-v)\|_{\phi}^2\leq&  C (\|d(u-w)\|_\phi^2+\|d(v-w)\|_\phi^2)\\
\leq& CI(u, w)^{1/2^{n-1}}\left(I(u, \phi)^{1-1/2^{n-1}}+I(w, \phi)^{1-1/2^{n-1}}\right)\\
&+CI(v, w)^{1/2^{n-1}}\left(I(v, \phi)^{1-1/2^{n-1}}+I(w, \phi)^{1-1/2^{n-1}}\right)
\end{split}
\]
By Proposition \ref{i00}, we have
\[
I(u, \phi)\leq nI(u, v), I(v, \phi)\leq nI(v, u), I(w, \phi)\leq n (I(w, u)+I(w, v))
\]
It follows that
\[
I(u, v)\leq C(I(u, w)^{\frac{1}{2^{n-1}}}+I(v, w)^{\frac{1}{2^{n-1}}}) (I(u, v)^{1-1/2^{n-1}}+I(u, w)^{1-1/2^{n-1}}+I(v, w)^{1-1/2^{n-1}})
\]
We assume $I(u, v)\geq \max\{I(u, w), I(v, w)\}$ (otherwise we are done). Hence it follows
\[
I(u, v)^{1/2^{n-1}}\leq C (I(u, w)^{\frac{1}{2^{n-1}}}+I(v, w)^{\frac{1}{2^{n-1}}})
\]
This is sufficient to prove that
\[
I(u, v)\leq C (I(u, w)+I(v, w))
\]
Now we establish \eqref{i06} (see \cite{BBEGZ}[Lemma 1.9]). First observe that
\[
\|d(u-v)\|_\psi\leq \|d(u-\psi)\|_\psi+\|d(v-\psi)\|_{\psi}\leq I(u, \psi)^{1/2}+I(v, \psi)^{1/2}
\]
Hence we have 
\[
\|d(u-v)\|_\psi^2\leq 2(I(u, \psi)+I(v, \psi))
\]
Hence if $I(u, v)\geq I(u, \psi)+I(v, \psi)$, clearly we have
\begin{equation}
\|d(u-v)\|_\psi^2\leq 2(I(u, \psi)+I(v, \psi))\leq C I(u, v)^{1/2^{n-1}}\left( I(u, \psi)^{1-\frac{1}{2^{n-1}}}+I(v, \psi)^{1-\frac{1}{2^{n-1}}}\right)
\end{equation}
Now we suppose $I(u, v)\leq  I(u, \psi)+I(v, \psi)$. 
Taking $\phi=\frac{u+v}{2}$, we consider 
\[
b_p:=\int_M d(u-v)\wedge d^c_B(u-v)\wedge \omega^p_\psi\wedge \omega^{n-p-1}_\phi\wedge \eta. 
\]
By \eqref{i05}, $b_0\leq I(u, v)$ and $b_{n-1}=\|d(u-v)\|_{\psi}^2$. We claim that, $p=0, \cdot, n-2$,
\begin{equation}\label{i07}
b_{p+1}\leq b_p+4\sqrt{b_p I(\psi, \phi)}
\end{equation}
We compute
\[
\begin{split}
b_{p+1}-b_p=& \int_M d(u-v)\wedge d^c_B(u-v)\wedge dd^c_B (\psi-\phi)\omega^p_\psi\wedge \omega^{n-p-2}_\phi\wedge \eta\\
=&-\int_M d(u-v)\wedge dd^c_B(u-v)\wedge d^c_B (\psi-\phi)\omega^p_\psi\wedge \omega^{n-p-2}_\phi\wedge \eta\\
=&-\int_M d(u-v)\wedge (\omega_u-\omega_v)\wedge d^c_B (\psi-\phi)\omega^p_\psi\wedge \omega^{n-p-2}_\phi\wedge \eta
\end{split}
\]
Using Cauchy-Schwarz inequality, we compute
\[
\begin{split}
&\left|\int_M d(u-v)\wedge \omega_u\wedge d (\psi-\phi)\omega^p_\psi\wedge \omega^{n-p-2}_\phi\wedge \eta\right|\leq \left(\int_M d(u-v)\wedge d^c_B(u-v)\wedge \omega_u \wedge\omega^p_\psi\wedge \omega^{n-p-2}_\phi\wedge \eta\right)^{1/2}\\
&\;\;\;\;\quad\times  \left(\int_M d(\psi-\phi)\wedge d^c_B(\psi-\phi)\wedge \omega_u \wedge\omega^p_\psi\wedge \omega^{n-p-2}_\phi\wedge \eta\right)^{1/2}\leq 2 \sqrt{b_p I(\psi, \phi)},
\end{split}
\]
where we have used that $\omega_u\leq 2\omega_\phi$ and \eqref{i04}. 
We can get the same estimate for 
\[
\left|\int_M d(u-v)\wedge \omega_v\wedge d (\psi-\phi)\omega^p_\psi\wedge \omega^{n-p-2}_\phi\wedge \eta\right|.
\]
This establishes \eqref{i07}. By Proposition \ref{i00}, we know that
\[
I(\psi, \phi)\leq (n+1)J(\psi, \phi)\leq  \frac{n}{2}(I(\psi, u)+I(\psi, v))
\]
Denote $a=(I(\psi, u)+I(\psi, v))$. We write \eqref{i07} as
\[
b_{p+1}\leq b_p+4\sqrt{b_p a}, p=0, \cdots, n-2
\]
Note that $b_0=I(u, v)\leq a$, hence it is evident that $b_p\leq C a$. Hence it follows that, for $p=0, \cdots, n-2$, \[b_{p+1}\leq C \sqrt{b_p a}\]
A direct computation gives that,
\[
b_{n-1}\leq C b_0^{1/2^{n-1}} a^{1-\frac{1}{2^{n-1}}}
\]
This completes the proof. 
\end{proof}

More generally, we have the following \cite{D4}[Proposition 3.48]
\begin{prop}\label{darvascontinuity02}Suppose $C>0$ and $\phi, \psi, u, v\in \cE_1(M, \xi, \omega^T)$ satisfies
\[
-C\leq \mathbb{I}(\phi), \mathbb{I}(\psi), \mathbb{I}(u), \mathbb{I}(v), \sup_M \phi, \sup_M \psi, \sup_M u, \sup_M v\leq C
\] 
Then there exists a continuous function $f_C:\R^{+}\rightarrow \R^{+}$ depending only on $C$ with $f_C(0)=0$ such that
\begin{equation}
\begin{split}
&\left|\int_M \phi(\omega_u^n-\omega_v^n)\wedge \eta\right|\leq f_C(I(u, v))\\
&\left|\int_M (u-v)(\omega_\phi^n-\omega_\psi^n)\wedge \eta\right|\leq f_C(I(u, v))
\end{split}
\end{equation}
\end{prop}
\begin{proof}The proof is similar in philosophy as Lemma \ref{bbegz01} and follows almost identically as in K\"ahler setting, see \cite{D4}[Proposition 3.48]. Hence we skip the details. 
\end{proof}

As a consequence, we have the following \cite{D4}[Theorem 3.46]

\begin{thm}\label{d101}Suppose $u_k, u\in \cE_1(M, \xi, \omega^T)$. The the following hold:
\begin{enumerate}
\item $d_1(u_k, u)\rightarrow 0$ if and only if $\int_M |u_k-u|\omega^n_T\wedge \eta\rightarrow 0$ and $\mathbb{I}(u_k)\rightarrow \mathbb{I}(u)$.
\item If $d_1(u_k, u)\rightarrow 0$, then $\omega^n_{u_k}\wedge \eta\rightarrow \omega_u^n\wedge \eta$ weakly and $\int_M |u_k-u|\omega^n_v\wedge \eta\rightarrow 0$ for $v\in \cE_1(M, \xi, \omega^T)$. 
\end{enumerate}
\end{thm}

\begin{proof}If $d_1(u_k, u)\rightarrow 0$, then Proposition \ref{darvascontinuity01} and Proposition \ref{darvascontinuity02} implies (1) and (2). For the reverse direction in (1), it follows almost identically as in K\"ahler setting, see \cite{D4}[Proposition 3.52], using Proposition \ref{darvascontinuity02} and approximation argument. We sketch the process. First we have
\[
\int_M u_k \omega_u^n\wedge \eta\rightarrow \int_M u \omega_u^n\wedge \eta
\] 
And then one argues that
\[
I(u, u_k)\leq (n+1)\left(\mathbb{I}(u_k)-\mathbb{I}(u)-\int_M (u-u_k)\omega_u^n\wedge \eta\right)
\]
Hence this shows that $I(u, u_k)\rightarrow 0$. Using Proposition \ref{darvascontinuity02} and Lemma \ref{i02}, one can then show
\[
\int_M |u_k-u|\omega_u^n\wedge \eta, \int_M |u_k-u|\omega_{u_k}^n\wedge \eta \rightarrow 0, k\rightarrow \infty. 
\]
This gives the desired convergence $d_1(u_k, u)\rightarrow 0$. 
\end{proof}

As an application of  results established above, we have the following compactness result in Sasaki setting, following \cite{D4}[Theorem 4.45]. 

\begin{thm}Let $u_j\in \cE_1(M, \xi, \omega^T)$ be a $d_1$-bounded sequence for which the entropy
\[\sup_jH(u_j)<\infty.\] Then $\{u_j\}$ contains a $d_1$-convergence sequence. 
\end{thm}

\begin{proof}We sketch the proof for completeness; for details see \cite{D4}[Theorem 4.45]. First $d_1$ bounded implies that $\mathbb{I}$ and $\sup u$ are both bounded. Together with Proposition \ref{compactness001}, this implies that $d_1$ bounded set is precompact in $L^1$. That is, there exists $u\in \cE_1(M, \xi, \omega^T)$ such that after passing by subsequence,
\[
\int_M |u_k-u|(\omega^T)^n\wedge \eta\rightarrow 0. 
\]
Moreover, we have (see \cite{D4}[Proposition 4.14, Corollary 4.15])
\[
 \limsup \mathbb{I}(u_k)\leq \mathbb{I}(u).
\]
Since all elements in $\cE_1(M, \xi, \omega^T)$ have zero Lelong number, we apply Zeriahi's uniform version  of the famous Skoda integrability
theorem \cite{Zeriahi} (we apply Zeriahi's theorem in each foliation chart) to obtain: for any $p\geq 1$, there exists  $C=C(p)$ such that
\[
\int_M e^{-p u_j} (\omega^T)^n\wedge \eta\leq C.
\]
Since $\sup u_j\leq C$, we have 
\[
\int_M e^{p |u_j|} (\omega^T)^n\wedge \eta\leq C.
\]
Now we need to use the assumption that $H(u_j)$ is uniformly bounded above. We proceed as in the proof of \cite{D4}[Theorem 4.45] to conclude
\begin{equation*}\label{compact05}
\int_M |u_j-u|\omega_{u_j}^n\wedge \eta\rightarrow 0.
\end{equation*}
By Proposition \ref{maenergy01} (which holds for $\cE_1$) we can then conclude that $\liminf \mathbb{I}(u_j)\geq \mathbb{I}(u)$. This gives $\lim \mathbb{I}(u_j)=\mathbb{I}(u)$. Hence $d_1(u_j, u)\rightarrow 0$, as a consequence of Theorem \ref{d101}.  
\end{proof}

Finally we have the extension of $\cK$-energy, see  \cite{BDL}[Theorem 1.2] for K\"ahler setting. 

\begin{thm}The $\cK$-energy can be extended to a functional $\cK: \cE_1\rightarrow \R\cup\{+\infty\}$.
Such a $\cK$-energy in $\cE^1$ is the greatest $d_1$-lsc extension of $\cK$-energy on $\cH$. Moreover, $\cK$-energy is convex along the finite energy geodesics of $\cE^1$.  
\end{thm}

\begin{proof}As in K\"ahler setting \cite{chen00}, we can write the $\cK$-energy as the following,
\begin{equation*}
\cK(\phi)=H(\phi)+\mathbb{J}_{\omega^T, -Ric}(\phi)
\end{equation*}
where $H(\phi)$ is the entropy part and $\mathbb{J}$ is the entropy part, taking the formula respectively, 
\[
\begin{split}
H(\phi)=&\int_M \log \frac{\omega_\phi^n\wedge \eta}{\omega^n_T\wedge \eta} dv_\phi\\
\mathbb{J}_{-Ric}(\phi)=&\frac{n\underline{R}}{(n+1)!}\int_M \phi \sum_{k=0}^n \omega^k_T\wedge \omega_\phi^{n-k}\wedge \eta-\frac{1}{n!}\int_M \phi \sum_{k=0}^{n-1}Ric\wedge \omega^k_T\wedge \omega_\phi^{n-1-k}\wedge \eta
\end{split}
\]
As a direct consequence of this formula, $\cK(\phi)$ is well-defined for $\phi\in \cH_\Delta$. More importantly, for $\phi_0, \phi_1\in \cH$ and $\phi_t\in \cH_\Delta$ being the geodesic connecting $\phi_0, \phi_1$, $\cK(\phi_t)$ is convex with respect to $t\in [0, 1]$. 

Now we extend $H(\phi)$ and $\mathbb{J}_{-Ric}$ to $\cE_1$ separately. As in \cite{BDL}, the extension of $\mathbb{J}_{-Ric}$ to $\cE_1$ is $d_1$-continuous, while since $d_1(u_k, u)\rightarrow 0$ implies that $\omega_{u_k}^n\wedge \eta\rightarrow \omega_u^n\wedge \eta$ weakly (Theorem \ref{d101}), this implies that the extension of $\phi\rightarrow H(\phi)$ to $\cE_1$ is $d_1$ lsc. Moreover, by \cite{he182}[Lemma 5.4], the extension of $\cK$ is the greatest lsc extension. In the end, the convexity of the extended $\cK$-energy along the finite energy geodesic segments follows exactly as in \cite{BDL}[Theorem 4.7]. 
\end{proof}

\end{document}